\documentclass[ijoo]{informs4}

\RequirePackage{tgtermes}
\RequirePackage{newtxtext}
\RequirePackage{newtxmath}
\RequirePackage{endnotes}

\usepackage[hypertexnames=false,colorlinks=true,breaklinks=true,bookmarks=true,urlcolor=blue,citecolor=blue,linkcolor=blue,bookmarksopen=false,draft=false]{hyperref}

\SingleSpacedXI

\makeatletter
\def\theARTICLEABSTRACT{%
  \vspace*{18pt}
  \begin{minipage}{\textwidth}\parindent1em\ABSfont
  \smallskip\noindent\theABSTRACT\endgraf
  \vskip5pt
  \theFUNDING
  \vskip5pt
  \theKEYWORDS
  \noindent\hrulefill
  \end{minipage}%
  \vspace*{0pt}
}
\makeatother

\makeatletter
\renewcommand{\theARTICLETOPLEFT}{}
\makeatother

\RRHSecondLine{}
\LRHSecondLine{}

\renewcommand{\theJOURNAL}{}

\makeatletter
\renewcommand{\theARTICLETOPLEFT}{}
\renewcommand{\theARTICLETOPRIGHT}{}
\makeatother

\usepackage{type1cm}

\usepackage{tikz}
\usetikzlibrary{arrows.meta,patterns,patterns.meta}
\tikzset{
  ray/.style={line width=1pt, black, -{Stealth[length=2.6mm]}},
  hundred/.style={red!75!black, dashed, line width=1.4pt},
  breach/.style={blue!65!black, line width=0.9pt},
  origindot/.style={circle, fill=black, inner sep=1.6pt},
  bdot/.style={circle, fill=blue!70!black, inner sep=1.6pt},
  genpt/.style={circle, fill=black, inner sep=1.1pt},
  hatch/.style={pattern={Lines[angle=45,distance=4pt]}, pattern color=cyan!70!black}
}
\definecolor{crimson}{RGB}{190,20,45}

\usepackage{natbib}
 \bibpunct[, ]{[}{]}{,}{a}{}{,}%
 \def\bibfont{\small}%

 \makeatletter
\let\NAT@parse@orig\NAT@parse
\renewcommand\NAT@parse[1]{%
  \NAT@parse@orig{#1}%
  \ifNAT@swa
    \xdef\NAT@date{(\NAT@date)}%
  \fi
}
\makeatother

\EquationsNumberedThrough
\def\TheoremsNumberedBySectionShared{%
\theoremstyle{TH}%
\newtheorem{theorem}{Theorem}[section]%
\newtheorem{lemma}[theorem]{Lemma}%
\newtheorem{proposition}[theorem]{Proposition}%
\newtheorem{corollary}[theorem]{Corollary}%
\theoremstyle{EX}%
\newtheorem{remark}[theorem]{Remark}%
\newtheorem{example}[theorem]{Example}%
\def\ECHowTheorems{\setcounter{theorem}{0}}%
}
\TheoremsNumberedBySectionShared
\ECRepeatTheorems

\DeclareMathOperator{\conv}{conv}
\DeclareMathOperator{\cone}{cone}
\DeclareMathOperator{\ccone}{\overline{cone}}
\DeclareMathOperator{\vol}{vol}

\makeatletter
\def\FigureCaptionFontStyle{\mdseries}%
\def\TableCaptionFontStyle{\mdseries}%
\def\FigureNameFontStyle{\bfseries}%
\def\TableNameFontStyle{\bfseries}%
\long\def\@makefigurecaption#1#2{\EGT\FigureCaptionFontStyle
  \setbox9=\hbox{\FigureNameFontStyle #1\kern16pt}\hangindent\wd9\noindent
  \HD{9}{0}\hbox{\FigureNameFontStyle #1\kern16pt}\ignorespaces #2\HD{0}{9}\endgraf}
\long\def\@maketablecaption#1#2{\EGT\TableCaptionFontStyle
  \setbox9=\hbox{\TableNameFontStyle #1\kern16pt}\hangindent\wd9\noindent
  \HD{9}{0}\hbox{\TableNameFontStyle #1\kern16pt}\ignorespaces #2\HD{0}{5}\endgraf}
\makeatother

\usepackage{etoolbox}
\AtEndEnvironment{example}{%
  \unskip\nobreak\quad\hfill\nobreak$\clubsuit$%
  \parfillskip=0pt \finalhyphendemerits=0 \endgraf}

\makeatletter
\def\Halmos{%
  \unskip\nobreak\hfil\penalty50\hskip1em\null\nobreak\hfil
  $\square$%
  \parfillskip=0pt \finalhyphendemerits=0 \endgraf}
\makeatother

\makeatletter
\def\TheoremHeaderFont{\sffamily\mdseries\scshape\lfstyle}
\def\ExampleHeaderFont{\sffamily\mdseries\scshape\lfstyle}
\makeatother

\makeatletter
\renewcommand*{\top}{%
  {\mathpalette\@transpose{}}%
}
\newcommand*{\@transpose}[2]{%
  \raisebox{\depth}{\scriptsize$\m@th#1\mathsf{T}$}%
}
\makeatother

\MANUSCRIPTNO{}

\begin{document}
%%%%%%%%%%%%%%%%

\RUNAUTHOR{Fampa and Lee}

\RUNTITLE{On the 100\% Rule of Linear Programming}

\TITLE{On the 100\% Rule of Linear Programming}

\ARTICLEAUTHORS{%
\AUTHOR{Marcia Fampa}
\AFF{Universidade Federal do Rio de Janeiro, Rio de Janeiro, Brazil,
\EMAIL{fampa@cos.ufrj.br}}

\AUTHOR{Jon Lee}
\AFF{University of Michigan, Ann Arbor, Michigan 48109,
\EMAIL{jonxlee@umich.edu}}
}

\ABSTRACT{%
The so-called ``100\% rule'' is a somewhat classical and yet somehow obscure
rule that gives some local sensitivity-analysis information based on
the very limited local sensitivity-analysis information that is often made available
by linear-programming software, derived from an optimal basis that is hidden from the user.
Introduced in a classic undergraduate textbook in 1977,
half a century ago, it was only picked up
on by researchers in a very limited way, and while it appears in some
undergraduate textbooks, it tends not to appear in most modern graduate-level
textbooks. Our contributions are
to celebrate and revive interest in the ``100\% rule''
and to extend it in a meaningful way. We introduce a cone rule that
certifies primal feasibility for more right-hand sides than the 100\% rule
does, using the same information, and that can be applied in linear time.
We then study when this cone rule is tight, that is, when it
captures the entire cone of right-hand sides for which the basis remains
feasible. We characterize tightness in every dimension for right-hand sides
interior to that cone (when no column of the basis is a coordinate direction), make
the characterization explicit for $m=2$ and $m=3$, and settle the question for
tridiagonal bases in every dimension: tightness is impossible once
$m\geq4$, and for $m=3$ we give an explicit criterion in the entries of the basis. We give further results for symmetric,
integer and unimodular bases. 
Finally, when tightness fails, we measure the fraction of the feasibility cone captured by the rule using normalized solid angles. For the autoregressive family at $m=3$, we obtain a closed-form candidate for the infimum of this fraction over all right-hand sides, whose optimality is supported numerically. We also obtain a closed-form expression for the best case over most of the correlation range, subject to two numerically verified properties, and estimate the best-case value numerically more generally.
}%

\FUNDING{M. Fampa was supported in part by CNPq grant 304340/2026-0.
J. Lee was supported in part by ONR grant N00014-24-1-2694.}


\maketitle

\section{Introduction}\label{sec:int}

Linear-programming solvers typically provide limited post-optimality information to casual users, often in the form of independent allowable ranges for individual right-hand side or objective coefficients. However, the optimal basis typically is inaccessible to such a user, particularly in the presence of degeneracy. As a consequence, such users may wish to certify optimal-basis preservation for simultaneous perturbations, using only the limited sensitivity information provided by the solver.

The classical 100\% rule addresses this issue by combining individual perturbation ranges into a simple sufficient condition for preserving primal or dual feasibility of an optimal basis. Although computationally attractive, the rule can be conservative because it effectively restricts perturbations to a generalized cross polytope (i.e., a generalized $\ell_1$ ball; see \S\ref{sec:100}). This motivates the search for larger certifiable perturbation regions that still avoid explicit knowledge of the underlying basis.

For concreteness and ease of
exposition, we refer to the
standard-form linear-programming formulation
\begin{equation}\tag{P}\label{P}
\min\{ c^\top x ~:~ Ax=b,~ x\geq \mathbf{0}\},
\end{equation}
and its dual
\begin{equation}\tag{D}\label{D}
\max\{ y^\top b ~:~ y^\top A \leq c^\top\},
\end{equation}
where $A\in\mathbb{R}^{m\times n}$
has full row-rank $m\geq 2$.
Following the notation of
\citep*{LeeLP},
a \emph{basis} is an ordered
set $\beta=\{\beta_1\,,\ldots,\beta_m\}
\subset \{1,\ldots,n\}$, such that $A_\beta:=[A_{\beta_1}|\cdots| A_{\beta_m}]$ is non-singular.
Let $\eta:=\{1,\ldots,n\}\setminus \beta$, in an arbitrary order, and then
 $A_\eta:=[A_{\eta_1}|\cdots| A_{\eta_{n-m}}]$.
The basis $\beta$ is \emph{primal feasible}
if $\bar{x}_\beta:=A_\beta^{-1}b\geq \mathbf{0}$, and
\emph{dual feasible} if $c_\eta^\top- c_\beta^\top
A_\beta^{-1}A_\eta\geq \mathbf{0}^\top$.
The basis $\beta$ is optimal if it is both
primal feasible and dual feasible.
As is very common, referring to \ref{P}, we call $b\in\mathbb{R}^m$ the nominal \emph{right-hand side}, and
we call $c \in\mathbb{R}^n$ the nominal \emph{objective vector}.

The classic 100\% rule is from \citep*[Sec. 3.7]{BHM}.\footnote{\citep*{BHM} was the textbook for OR 320, taught by Bob Bland at Cornell University in the Fall semester of 1979. The second author first learned linear programming as a student in that class.}
We note that while the 100\% rule appears in some other undergraduate textbooks, for example \citep*{hillier2001introduction} and \citep*{winston2004operations}, it tends not to appear in popular graduate-level textbooks, for example \citep*{chvatal1983linear} and \citep*{bertsimas1997introduction}.

The motivation for the 100\% rule, back in the 1970s,
was to take the existing available
sensitivity-analysis output from
linear-programming codes, on individual ranges of
right-hand side and objective coefficients, and give easily checkable
sufficient conditions for the optimal (unseen) basis
to be:
\begin{itemize}
\item primal feasible for a new right-hand side vector $\hat b\in\mathbb{R}^m$, and
\item dual feasible for a new objective-coefficient vector $\hat c\in\mathbb{R}^n$.
\end{itemize}
If both conditions hold, then
the basis $\beta$ is optimal. Specifically,
$\bar{x}\in\mathbb{R}^n$ defined by $\bar{x}_\beta:= A_\beta^{-1}\hat b\in\mathbb{R}^m$ and
$\bar{x}_\eta:=\mathbf{0}\in\mathbb{R}^{n-m}$ is optimal for \ref{P},
and $\bar{y}^\top:=\hat c_{\beta}^\top A_{\beta}^{-1}\in\mathbb{R}^m$
is optimal for \ref{D}.

We will focus our attention
on the right-hand side; we do not treat the objective-coefficient vector, or
simultaneous changes in both the right-hand side and objective-coefficient vector, but the approach to those is similar. 
Throughout, $\mathcal{B}_\beta$ denotes the cone of right-hand sides for
which $\beta$ remains primal feasible, and, for a given
$b\in\mathcal{B}_\beta$\,, $\mathcal{K}_\beta(b)$ will denote a specific
cone (to be formally defined), which can be computed using the same
information a solver provides for implementing the classical 100\% rule,
and which certifies primal feasibility for typically more right-hand sides
than the 100\% rule does. Our central question throughout is: when do we have
$\mathcal{K}_\beta(b)=\mathcal{B}_\beta$? That is, when does this certifiable
region capture the \emph{entire} feasibility cone?

In what follows, we assume familiarity with very basic concepts of polyhedral theory; see \citep*{ziegler}, for example.

\medskip

\noindent {\bf Organization and contributions.}
In \S\ref{sec:100}, we briefly review the classical 100\% rule for
right-hand side sensitivity analysis. 
In \S\ref{sec:conerule}, we describe our cone rule and establish its linear-time complexity. 
In \S\ref{sec:m2}, we make a detailed study of the cone (associated with our cone rule)
for $m=2$, seeking to understand when it verifies the entire set of right-hand sides for which a given basis is feasible
(a potentially bigger cone).
In \S\ref{sec:transfer}, we set down five short lemmas, valid in every dimension,
that transfer the defining relation into $\bar{x}_\beta$-coordinates, and combine
them into a characterization of when $\mathcal{K}_\beta(b)=\mathcal{B}_\beta$ for
right-hand sides interior to $\mathcal{B}_\beta$\,, in every dimension $m\geq3$
(for bases with no column a coordinate direction), which can be checked in
polynomial time; \S\ref{sec:general_m}, \S\ref{sec:m3} and \S\ref{subsec:tridiag}
all rest on them.
In \S\ref{sec:general_m}, armed with what we learn for the $m=2$ case, we study
the case of general $m$: we establish that $\mathcal{K}_\beta(b)\not=\mathcal{B}_\beta$
for almost every $b$ whenever $\mathcal{B}_\beta$ is not an orthant, and for every
$b$ for suitable cones, although equality remains possible in every dimension.
In \S\ref{sec:m3}, we make that characterization explicit for $m=3$, where it
reduces to closed-form conditions on $A_\beta^{-1}$ (with bases having a coordinate
direction among their columns treated separately), together with an orthant form of
the criterion, a
refinement describing when a \emph{non-negative} $b$ suffices, a duality that
recasts the criterion in terms of $A_\beta$ rather than $A_\beta^{-1}$, and an
explicit symmetric-basis construction.
In \S\ref{subsec:tridiag}, we study tridiagonal bases: we demonstrate that
$\mathcal{K}_\beta(b)\not=\mathcal{B}_\beta$ for every $b$ once $m\geq4$, give an
explicit criterion in the entries of $A_\beta$ for when an interior $b$ achieves
$\mathcal{K}_\beta(b)=\mathcal{B}_\beta$ at $m=3$, and show that every
tridiagonal M-matrix basis has
$\mathcal{K}_\beta(b)\not=\mathcal{B}_\beta$ for every $b$, for any $m\geq3$,
with the classical AR(1) precision matrix as one instance; for that family at
$m=3$ we then describe the best case, the infimum and the average of the solid-angle
ratio $\sigma_3$ as a function of the correlation $\rho$.
In \S\ref{subsec:unimodular}, we consider integer unimodular bases, and find
that they neither help nor hurt this picture; in particular
$\mathcal{K}_\beta(b)\not=\mathcal{B}_\beta$ persists for infinitely many interior
integer right-hand sides.
In \S\ref{sec:expr}, we carry out computational experiments. We estimate, as a
function of $m$ and for several ensembles of bases, a median radial expansion
factor of $\mathcal{K}_\beta(b)$ over the classical $\{b\}+\Delta_{100\%}$\,: it stays
bounded for unstructured bases, but grows rapidly in $m$ for generic tridiagonal
and for symmetric bases. We also measure, through normalized solid angle, what
fraction of $\mathcal{B}_\beta$ the cone rule captures when it falls short of
equality.
In \S\ref{sec:conc}, we discuss directions for further work.

\medskip

\noindent {\bf Notation.}
Throughout, $\subset$ denotes an arbitrary (not necessarily proper) subset.
For $i=1,\ldots,m$, $\mathbf{e}_i$ denotes the $i$-th standard unit vector,
$\mathbf{e}$ denotes the all-ones vector, and $\mathbf{0}$ denotes the zero
vector.
For $S\subset\mathbb{R}^m$, $\cone(S)$ (resp.,
$\overline{\cone}(S)$, $\conv(S)$) is the conical hull (resp. closure of the conical hull, convex hull)
of $S$. For $x\in\mathbb{R}^m$, $\operatorname{supp}(x):=\{i : x_i\not=0\}$.
For convenience, we write $H:=A_\beta^{-1}$, leaving its dependence on $\beta$ implicit;
context always makes it clear.
We say that a vector has \emph{uniform signs} if its entries are all
non-negative or all non-positive, and that it has \emph{mixed signs} otherwise.

%%%%%%%%%%%%%%%%%%%%%

\section{The 100\% rule for the right-hand side}\label{sec:100}

Let us first consider changes to a single right-hand side element $b_i$\,.
Suppose that $\beta$ is a primal-feasible basis for \ref{P}, and consider the problem
\begin{equation}\tag{$\mbox{P}_i$}\label{Pi}
\min\{ c^\top x ~:~ Ax=b+\delta_i \mathbf{e}_i\,,~ x\geq \mathbf{0}\},
\end{equation}
where $\delta_i\in\mathbb{R}$.
The basis $\beta$ is feasible for \ref{Pi} if
$A_\beta^{-1}(b + \delta_i \mathbf{e}_i)\geq \mathbf{0}$.
For convenience, let $h^i:=A_\beta^{-1}
\mathbf{e}_i$\,, and so
\[
H:=[h^1, h^2, \ldots, h^m] = A_\beta^{-1}.
\]
Then, the condition $A_\beta^{-1}(b + \delta_i\mathbf{e}_i)\geq \mathbf{0}$
can be re-expressed as $\bar{x}_{\beta} + \delta_i h^i \geq \mathbf{0}$.
It is straightforward to check that $\beta$ is
 feasible for \ref{Pi} if and only if
$\delta_i\in [L_i\,,U_i]$, where
\[
L_i := \max_{k ~:~ h^i_k > 0} \left\{ -\bar{x}_{\beta_k} /h^i_k \right\}\leq 0,
\]
and
\[
U_i := \min_{k ~:~ h^i_k < 0} \left\{ -\bar{x}_{\beta_k} /h^i_k \right\}\geq 0.
\]
It can be the case that $h^i_k \leq 0$ for all $k$,
in which case we define $L_i:=-\infty$, and it can be the case that
$h^i_k \geq 0$ for all $k$,
in which case we define $U_i:=+\infty$.
But notice that we cannot have
both $L_i=-\infty $ and $U_i=+\infty$ (for a given $i$), or we would have $h^i=\mathbf{0}$, which is not possible as $H$ is invertible.
In the language of \S\ref{sec:int}, then: $h^i$ has mixed signs precisely when
both $L_i$ and $U_i$ are finite, and $h^i$ has uniform signs precisely when
exactly one of them is infinite --- $L_i=-\infty$ if $h^i\leq\mathbf{0}$, and
$U_i=+\infty$ if $h^i\geq\mathbf{0}$. This is the only role the uniform/mixed terminology
plays, and it is how we use it in \S\ref{sec:m2} and \S\ref{sec:m3}.
It is important to emphasize that
so far we are referring to changing one right-hand side element
and holding all others constant.

A concise and geometric way to state the \emph{100\% rule} for primal feasibility of $\beta$ for
a right-hand side $b+\delta$,
$\delta\in \mathbb{R}^m$, is
\begin{align*}
&\delta\in \Delta_{100\%}:=\conv \left(
\{ L_i \mathbf{e}_i ~:~ L_i >-\infty\}
\cup
\{ U_i \mathbf{e}_i ~:~ U_i <+\infty\}
\right)\\
&\qquad\qquad\qquad +
\cone \left(
\{ -\mathbf{e}_i ~:~ L_i =-\infty\}
\cup
\{ \mathbf{e}_i ~:~ U_i =+\infty\}
\right),
\end{align*}
which is an axis-aligned ``generalized cross polytope''. Recall that a
standard \emph{cross polytope} is the convex hull of
$\{\pm r\mathbf{e}_i ~:~ i=1,\ldots,m\}$ --- equivalently, a ball in the
$\ell_1$ norm --- which is the case $L_i=-r$ and $U_i=r$, for
$i=1,\ldots,m$, for a constant $r>0$. The region above is
\emph{generalized} in two respects: the $2m$ bounds need be neither
symmetric about the origin (that is, $L_i=-U_i$) nor equal across $i$;
and the $\cone$ term, which is trivial when each $L_i$ and $U_i$ is
finite --- in which case $\Delta_{100\%}$ is a genuine polytope ---
contributes recession directions when some $L_i=-\infty$ or
$U_i=+\infty$, making $\Delta_{100\%}$ an unbounded polyhedron instead.
Now, we can define, for a given $\delta\in\mathbb{R}^m$, and for $i=1,\ldots,m$,
\[
\lambda_i
:=
\begin{cases}
\delta_i/L_i\,, & \mbox{if $\delta_i < 0$;}\\
0, & \mbox{if
$\delta_i =0$;}\\
\delta_i/U_i\,, & \mbox{if $\delta_i >0$,}
\end{cases}
\]
where we consider $\delta_i/L_i:=0$ if $L_i=-\infty$,
and $\delta_i/U_i:=0$ if $U_i=+\infty$.
If $\bar{x}_\beta$ has a zero entry, a finite bound can equal $0$; we then read
$L_i=0$ as $0^-$ and $U_i=0$ as $0^+$, so that $\delta_i/L_i:=+\infty$ for
$\delta_i<0$ and $\delta_i/U_i:=+\infty$ for $\delta_i>0$. This agrees with the
definition of $\Delta_{100\%}$\,: a zero bound forbids any move of $b_i$ in that
direction, and an infinite $\lambda_i$ makes $\sum_i\lambda_i\leq1$ fail.
It is easy to check that
 $\delta\in \Delta_{100\%}$ if and only if
 $\sum_{i=1}^m \lambda_i \leq 1$,
which then justifies the moniker
``100\% rule''.
Note that in the undergraduate textbook \citep*[Sec. 3.7]{BHM}, which introduced
the 100\% rule and coined the name, there is not any explicit discussion
about how to handle infinite bounds in that context, but our development
explicitly accommodates this possibility.
We also note that the idea of the 100\% rule was \emph{hiding in plain sight} even earlier, in the German-language book
\citep*[pp. 82--83,93,179--180]{Gal}.

\medskip

\noindent {\bf Some related literature.}
For the 100\% rule, \citep*{Cai} clarifies the treatment of infinite bounds, refines the
rule in the case of objective-coefficient sensitivity, and
also naturally extends the ideas to changes in a row or column of the
constraint matrix $A$.
\citep*{Wendell1,Wendell2} develop the idea of a so-called ``tolerance region'', incomparable
with the region of the 100\% rule, and requiring
knowledge of the optimal basis (and more demanding calculations) to apply;
these regions are \emph{boxes}, allowing for independent consideration of all coefficients involved. \citep*{Wondolowski} extends the tolerance-region approach to allow for asymmetric boxes and a good treatment of degeneracy, but again
the method requires knowledge of the optimal basis (and more demanding calculations) to apply (also see \citep*{Wang}, for a similar approach).
\citep*{Filippi} takes the tolerance approach even further, with similar shortcomings.

%%%%%%%%%%%%%%%%

\section{A cone right-hand side rule and its linear complexity}\label{sec:conerule}

For convenience, and without loss of generality, we can assume, when we like, that $b\geq\mathbf{0}$. Indeed, we can always achieve this by preprocessing \ref{P}, multiplying rows of the system $Ax=b$ by $-1$, as needed.

We can consider the ``trial
right-hand side'' $\hat{b}:=b+\delta$, and
then the set of trial right-hand sides
for which the 100\% rule certifies that $\beta$ is feasible
is just the Minkowski sum
$\{b\}+\Delta_{100\%}$\,.
But the full set of trial right-hand
sides for which $\beta$ is feasible
is just the pointed full-dimensional simplicial convex cone
$\mathcal{B}_\beta := \{\hat b ~:~ A_\beta^{-1} \hat b \geq \mathbf{0}\}$, and so
$\{b\}+\Delta_{100\%}\subset \mathcal{B}_\beta$\,. Indeed, $\mathcal{B}_\beta=\cone(\left\{g^1,\ldots,g^m\right\})$, where $g^i:=A_{\beta_i}$ (the $i$-th column of $A_\beta$\,, equivalently the $\beta_i$-th column of $A$), for $i=1,\ldots,m$; these are exactly the $m$ extreme rays of $\mathcal{B}_\beta$\,. We use this notation $g^1,\ldots,g^m$ for the generators of $\mathcal{B}_\beta$ throughout this work. 

Now, a nice and motivating feature of the 100\% rule is
that we can apply it to a trial right-hand side $\hat b$, using
only knowledge of $b$ and the $L_i$ and $U_i$ ($i=1,\ldots,m$),
as might readily be available as output of a linear-programming code.
While applying the 100\% rule needs only this limited information, checking membership in
$\mathcal{B}_\beta$ requires knowledge of $A_\beta$\,.
Note that we cannot, in general,
determine the $\beta$ found by a solver from a basic
solution provided by the solver.
Even with a knowledge
of $A$, in the presence of
degeneracy, there is not a
unique way to construct a basis $\beta$
yielding an optimal basic solution produced by a solver.
So, hot-starting a solver
with a different basis, say $\hat \beta$, than the unknown basis $\beta$ found by the solver could require the solver to enter a potentially long sequence of (dual-simplex) pivots to
recover an optimal basis
for a revised right-hand side
in $\mathcal{B}_\beta$\,.
So there is a real advantage
to working with $\{b\}+\Delta_{100\%}$ compared to the typically larger set $\mathcal{B}_\beta$\,.

We wish to consider larger sets
than $\{b\}+\Delta_{100\%}$
that $\mathcal{B}_\beta$
contains but that are
still only computed from
 $b$ and the $L_i$ and $U_i$\,.
 A simple such set is the convex polyhedral cone

 \begin{align*}
\mathcal{K}_\beta(b):=&\ccone(\{b\}+\Delta_{100\%})\\
=&\cone(
\{ b+L_i\mathbf{e}_i ~:~
   \mbox{ for $i$ such that $L_i>-\infty$} \} \\
&\qquad\qquad   \cup\{ -\mathbf{e}_i ~:~
   \mbox{ for $i$ such that $L_i=-\infty$}
   \}\\
&\qquad\qquad \cup \{ b+U_i\mathbf{e}_i ~:~
   \mbox{ for $i$ such that $U_i<+\infty$} \}\\
&\qquad\qquad   \cup\{ \mathbf{e}_i ~:~
   \mbox{ for $i$ such that $U_i=+\infty$}
   \}
   ).
 \end{align*}
We write $\mathcal{K}_\beta(b)$
because the
calculation of $\Delta_{100\%}$
(equivalently, the $L_i$ and $U_i$)
depends on both $\beta$ and $b$, while
we do wish to emphasize that we
are still motivated by the
situation where we do not know $\beta$.

Indeed, as $\mathcal{B}_{\beta}$ is a convex cone containing  $\{b\}+\Delta_{100\%}$\,, it also contains every non-negative conical combination of points in  $\{b\}+\Delta_{100\%}$\,. Hence, $\mathcal{K}_\beta(b)\subset \mathcal{B}_{\beta}$\,.
 Note that $\Delta_{100\%}$ may be unbounded when some of the bounds $L_i$ or $U_i$ are infinite. Consequently, it is natural to work with the closed conical hull operator
 $\ccone(\cdot)$.

For a given trial right-hand side  $\hat{b}\in\mathbb{R}^m$, checking whether
$\hat{b}\in \mathcal{K}_\beta(b)$ amounts to
solving a system of $m$ equations in $2m$ non-negative variables (weights on the $2m$ generators), which is solvable in
polynomial time by linear programming, assuming rational input.
This is more work than applying the 100\% rule,
but it does not require knowing
any more information than the 100\% rule employs.
If we had many trial right-hand sides $\hat b$ to process,
we could then treat $\hat b$ as a variable in $\mathbb{R}^m$, and apply Fourier-Motzkin elimination to the system, projecting out the $2m$ non-negative variables,
arriving at a linear-inequality system in the variable $\hat{b}$ alone. We could then remove redundant inequalities
using linear programming, to arrive finally at
a linear system with one inequality per facet of $\mathcal{K}_\beta(b)$. All of this is a lot of work,
but then we could test
many trial right-hand sides very quickly.

Alternatively, for a single trial right-hand side $\hat b$, we can test membership in $\mathcal{K}_\beta(b)$ in $O(m)$ time, without invoking a general-purpose linear-programming solver at all, by exploiting the special structure of the generators. Every generator of $\mathcal{K}_\beta(b)$ is either of the form $b+c_i\mathbf{e}_i$ (for finite $L_i$ or $U_i$) or a bare $\pm\mathbf{e}_i$ (for an infinite bound). Writing $\omega$ for the total weight assigned to the $b+c_i\mathbf{e}_i$-type generators, feasibility of $\hat b = \omega b + v$ decouples coordinate by coordinate: for the residual $v_i(\omega):=\hat b_i - \omega b_i$\,, the minimum weight $s_i(v_i)$ that coordinate $i$ must contribute toward $\omega$ is

\[
s_i(v_i) =
\begin{cases}
\max\!\left\{v_i/U_i\,,\ v_i/L_i\right\}, & \text{if $L_i\,,U_i$ are both finite;}\\[2pt]
\max\!\left\{0,\ v_i/L_i\right\}, & \text{if only $L_i$ is finite;}\\[2pt]
\max\!\left\{0,\ v_i/U_i\right\}, & \text{if only $U_i$ is finite,}
\end{cases}
\]
and we can always absorb any excess weight beyond this minimum without affecting $v_i$\,. Here a zero bound is read as in \S\ref{sec:100} ($L_i=0^-$, $U_i=0^+$, and $0/0:=0$), so that, for example, $v_i/U_i=+\infty$ when $U_i=0<v_i$\,. Each $s_i$\,, being a maximum of at most two linear functions, is convex and piecewise linear in $v_i$\,, and hence $f(\omega):=\sum_i s_i(v_i(\omega)) - \omega$ is convex and piecewise linear in the scalar $\omega$; moreover $\hat b\in\mathcal{K}_\beta(b)$ if and only if $\min_{\omega\geq0} f(\omega)\leq 0$, the constraint $\omega\geq0$ reflecting that the total weight is non-negative. (If some bound is $0$, the corresponding $s_i$ is $+\infty$ on one side of $v_i=0$; then $f$ is convex with values in $\mathbb{R}\cup\{+\infty\}$, finite on an interval of $\omega$ whose endpoints are breakpoints, and everything below applies on that interval.) For each coordinate with $b_i\not=0$, $s_i(v_i(\omega))$ has a single breakpoint, at $\omega_i^*:=\hat b_i/b_i$\,, with an explicit slope on either side; coordinates with $b_i=0$ instead contribute a constant term to $f$. The minimizer $\omega^*$ of $f$ over $\omega\geq0$ is the breakpoint at which the running sum of slopes, accumulated as we process the non-negative $\omega_i^*$ in increasing order, first becomes non-negative --- or $\omega^*=0$ if that sum is already non-negative at $\omega=0$, which is the case in particular when every breakpoint is negative. Because $\omega^*$ is really just a weighted median of the $m$ breakpoints, we can find it (and hence membership in $\mathcal{K}_\beta(b)$) not merely by sorting, but in linear time, via a prune-and-search adaptation of the classical median-of-medians selection algorithm (see \citep*{BFPRT1973}). A useful byproduct of the sweep: if the running slope never becomes non-negative, then $f(\omega)\to-\infty$ as $\omega\to\infty$, automatically certifying $\hat b\in\mathcal{K}_\beta(b)$, with no further evaluation of $f$ required. This puts \emph{a single membership query for the cone rule at the same $O(m)$ complexity as the classical 100\% rule.}

\begin{remark}\label{rem:dualsimplex}
It is worth recording how $L_i$ and $U_i$ read in pivoting terms, as they are
ratio tests of exactly the kind the dual simplex method performs. Fix an index
$i$ and recall from \S\ref{sec:100} that
$U_i=\min_{k\,:\,h^i_k<0}\bigl\{-\bar{x}_{\beta_k}/h^i_k\bigr\}$. As $\delta$
increases from $0$, the basic solution moves along $\bar{x}_\beta+\delta h^i$,
and the index $k^*$ attaining the minimum (any one of them, if several attain it)
is the first basic variable to reach $0$. Thus $b+U_i\mathbf{e}_i$ lies on the facet of $\mathcal{B}_\beta$ spanned by
$\{g^j:j\not=k^*\}$; moving beyond it makes $\beta$ infeasible, and $\beta_{k^*}$
is precisely the variable that a dual simplex method would then drive out of the
basis. Symmetrically for $L_i$\,. So the ratio tests behind $L_i$ and $U_i$ --- one
pass over $h^i$, that is $O(m)$ per index and $O(m^2)$ for all $m$ of them, given
$H$ --- determine, for each coordinate direction, both how far $b$ may move and
which basic variable would leave first. A solver performs these tests when it
reports the ranges, but typically reports only the bounds; the $O(m)$ membership
test above uses only those bounds, and so certifies feasibility of $\beta$ without
identifying any leaving variable.

This also anticipates why the question studied in this paper is delicate. A
finite-bound point $b+\theta_i\mathbf{e}_i$ lies on the ray generated by $g^j$
exactly when the $m-1$ basic variables other than $\beta_j$ vanish
\emph{simultaneously}. For $m=2$ that is a single variable, and occurs
generically; for $m\geq3$ it requires a coincidence, and it is that coincidence
which the later sections analyze.
\end{remark}

When $\mathcal{K}_\beta(b)\subsetneqq\mathcal{B}_\beta$\,, it is natural to ask how much of $\mathcal{B}_\beta$ the cone rule
does capture. As $\mathcal{K}_\beta(b)$ and $\mathcal{B}_\beta$ are both cones, this can
be made quantitative. Write $\tilde\Omega_m(\cdot)$ for the normalized solid angle
measure of a cone in $\mathbb{R}^m$ --- the fraction of $\mathbb{R}^m$ that it
occupies --- and set
\begin{equation}\label{eq:sigmam}
\sigma_m \;:=\; \frac{\tilde\Omega_m\bigl(\mathcal{K}_\beta(b)\bigr)}
                     {\tilde\Omega_m\bigl(\mathcal{B}_\beta\bigr)}\,.
\end{equation}
If $\bar{x}_\beta>\mathbf{0}$, every finite bound is non-zero, so $\Delta_{100\%}$
contains a neighborhood of $\mathbf{0}$ and $\mathcal{K}_\beta(b)$ is
full-dimensional; then $\sigma_m\in(0,1]$, and $\sigma_m=1$ precisely when
$\mathcal{K}_\beta(b)=\mathcal{B}_\beta$\,, as two nested closed cones with the same
solid angle coincide. Appendix~\ref{app:solidangle} explains how to compute
$\sigma_m$\,; it is studied for the AR(1) family in \S\ref{subsec:ar1} and
experimentally in \S\ref{sec:expr}.

%%%%%%%%%%%%%%%%%%%%%%%%%%%%%%

\section{The case of dimension $m=2$}\label{sec:m2}

The case of $m=2$ is quite interesting and even inspiring, so we investigate it exhaustively.
In what follows, we denote the
closed quadrants in counter-clockwise order, beginning from the non-negative quadrant, as 
$\mbox{\rm O}_{++}\,, 
\mbox{\rm O}_{-+}\,,
\mbox{\rm O}_{--}\,,
\mbox{\rm O}_{+-}$\,.

We begin with a very simple lemma, which we also use later, when we consider $m>2$. 

\begin{lemma}\label{lemma:main}
Suppose that $m\geq2$. If every extreme ray of the (pointed simplicial) cone
$\mathcal{B}_{\beta}$ contains a non-zero point of the set
\[
    \{b+L_i\mathbf{e}_i~:~L_i>-\infty\}\cup\{b+U_i\mathbf{e}_i~:~U_i<+\infty\},
\]
then $\mathcal{K}_\beta(b)=\mathcal{B}_{\beta}$\,.
\end{lemma}

\begin{proof}{Proof.}
Clearly $\mathcal{K}_\beta(b)\subset\mathcal{B}_{\beta}$\,, so it suffices to prove that $\mathcal{B}_{\beta}\subset\mathcal{K}_\beta(b)$.
Let $g$ be a generator of an extreme ray of $\mathcal{B}_\beta$\,. By hypothesis, some positive multiple of $g$ is one of the finite points
$b+L_i\mathbf{e}_i$ or $b+U_i\mathbf{e}_i$\,. As these points are generators of $\mathcal{K}_\beta(b)$, then $g\in \mathcal{K}_\beta(b)$.
Hence, $\mathcal{K}_\beta(b)$ contains every extreme ray of $\mathcal{B}_{\beta}$\,. Because $\mathcal{B}_\beta$ is a polyhedral cone, its extreme rays generate it. Therefore, $\mathcal{B}_{\beta}\subset\mathcal{K}_\beta(b)$.\Halmos
\end{proof}

For $m=2$, Lemma~\ref{lemma:main}  particularizes to: If the set of finite points
$
    \{b+L_i\mathbf{e}_i~:~L_i>-\infty\}\cup\{b+U_i\mathbf{e}_i~:~U_i<+\infty\}
$
    contains one non-zero point on each of the two extreme rays of $\mathcal{B}_{\beta}$\,, then $\mathcal{K}_\beta(b)=\mathcal{B}_{\beta}$\,.

In Figures~\ref{fig:one}--\ref{fig:four}, we demonstrate what can
happen in dimension 2. The figures tease out
when $\mathcal{K}_\beta(b)$ and $\mathcal{B}_{\beta}$ are the same.
What is also illustrated is how much larger $\mathcal{K}_\beta(b)$ can be,
in comparison to $\{b\}+\Delta_{100\%}$\,.

The black dot represents the origin,
and the blue dot represents the nominal right-hand side $b$ (always selected to be non-negative).
The dashed red lines indicate the boundary of the standard 100\%-rule region,
$\{b\}+\Delta_{100\%}$\,. We depict the triangular cone
$\mathcal{B}_\beta$ by its two extreme rays, shown in
black. A drawing having  $\mathcal{K}_\beta(b) \subsetneqq \mathcal{B}_\beta$
is the left-hand one in Figure~\ref{fig:three}. We indicate the difference of the sets
with cyan hatching. We see in Figure~\ref{fig:three}
that the equality of $\mathcal{B}_\beta$ (same in both) with $\mathcal{K}_\beta(b)$
depends on the placement of the nominal right-hand side $b$.

\begin{figure}
\FIGURE
{\begin{tikzpicture}[scale=1.5]
  \coordinate (O) at (0,0);
  \coordinate (B) at (1,1);
  \draw[ray] (O) -- (1.3,2.6);
  \draw[ray] (O) -- (2.6,1.3);
  \draw[breach] (0.5,1) -- (2,1);
  \draw[breach] (1,0.5) -- (1,2);
  \draw[hundred] (0.5,1) -- (1,2) -- (2,1) -- (1,0.5) -- cycle;
  \node[genpt] at (0.5,1) {};
  \node[genpt] at (2,1) {};
  \node[genpt] at (1,0.5) {};
  \node[genpt] at (1,2) {};
  \node[origindot] at (O) {};
  \node[below left=1pt] at (O) {$\mathbf{0}$};
  \node[bdot] at (B) {};
  \node at (1.45,2.7) {$g^1$};
  \node at (2.75,1.15) {$g^2$};
  \node[below right] at (B) {$b$};
\end{tikzpicture}}
{Example with all four bounds $L_1\,,L_2\,,U_1\,,U_2$ finite, and where $\mathcal{K}_\beta(b)= \mathcal{B}_\beta$\,.
 If we move $b$ to the origin, then $\mathcal{K}_\beta(b)=\{\mathbf{0}\}\not=\mathcal{B}_\beta$\,.\label{fig:one}}
{}
\end{figure}

\begin{figure}
\FIGURE
{\begin{tikzpicture}[scale=0.62]
  \coordinate (O) at (0,0);
  \coordinate (B) at (4,2);
  \draw[ray] (O) -- (2.6,5.2);
  \draw[ray] (O) -- (8.3,0);
  \draw[breach] (1,2) -- (4,2);
  \draw[breach,-{Stealth[length=2.6mm]}] (4,2) -- (8.3,2);
  \draw[breach] (4,0) -- (4,8);
  \node[genpt] at (1,2) {};
  \node[genpt] at (4,0) {};
  \node[genpt] at (4,8) {};
  \draw[hundred] (1,2) -- (4,8);
  \draw[hundred] (1,2) -- (4,0);
  \draw[hundred,-{Stealth[length=2.6mm]}] (4,8) -- (8.3,8);
  \draw[hundred,-{Stealth[length=2.6mm]}] (4,0) -- (8.3,0);
  \node[origindot] at (O) {};
  \node[left] at (O) {$\mathbf{0}$};
  \node[bdot] at (B) {};
  \node[below right] at (B) {$b$};
  \node[left=2pt] at (2.6,5.2) {$g^1$};
  \node at (8.6,0) {$g^2$};
\end{tikzpicture}}
{Example with exactly one infinite bound among $L_1\,,L_2\,,U_1\,,U_2$\,, and where $\mathcal{K}_\beta(b)= \mathcal{B}_\beta$\,. Moving $b$ to the origin gives instead $\mathcal{K}_\beta(b)=\cone(\{\mathbf{e}_1\})\not=\mathcal{B}_\beta$\,.\label{fig:two}}
{}
\end{figure}

\begin{figure}
\FIGURE
{\begin{tikzpicture}[scale=0.95]
  \begin{scope}
    \coordinate (O) at (0,0);
    \coordinate (B) at (0,1);
    \fill[hatch] (0,0) -- (3.6,-1.2) -- (3.6,0) -- cycle;
    \draw[ray] (O) -- (3.6,-1.2);
    \draw[ray] (O) -- (-1.8,3.6);
    \draw[ray, black!55] (O) -- (3.9,0);
    \draw[ray, black!55] (O) -- (0,3.4);
    \draw[breach] (-0.5,1) -- (0,1);
    \draw[breach,-{Stealth[length=2.6mm]}] (0,1) -- (3.9,1);
    \draw[breach,-{Stealth[length=2.6mm]}] (0,0) -- (0,3.4);
    \node[genpt] at (-0.5,1) {};
    \node[genpt] at (0,0) {};
    \draw[hundred] (-0.5,1) -- (0,0);
    \draw[hundred,-{Stealth[length=2.6mm]}] (-0.5,1) -- (-0.5,3.4);
    \draw[hundred,-{Stealth[length=2.6mm]}] (0,0) -- (3.9,0);
    \node[origindot] at (O) {};
    \node[left] at (O) {$\mathbf{0}$};
    \node[bdot] at (B) {};
    \node[right] at (B) {$b$};
    \node at (3.95,-1.35) {$g^1$};
    \node[above] at (-1.8,3.6) {$g^2$};
  \end{scope}
  \begin{scope}[xshift=6.6cm]
    \coordinate (O) at (0,0);
    \coordinate (B) at (1,1);
    \draw[ray] (O) -- (3.6,-1.2);
    \draw[ray] (O) -- (-1.8,3.6);
    \draw[ray, black!55] (O) -- (3.9,0);
    \draw[ray, black!55] (O) -- (0,3.4);
    \draw[breach] (-0.5,1) -- (1,1);
    \draw[breach,-{Stealth[length=2.6mm]}] (1,1) -- (3.9,1);
    \draw[breach] (1,-0.333) -- (1,1);
    \draw[breach,-{Stealth[length=2.6mm]}] (1,1) -- (1,3.4);
    \node[genpt] at (-0.5,1) {};
    \node[genpt] at (1,-0.333) {};
    \draw[hundred] (-0.5,1) -- (1,-0.333);
    \draw[hundred,-{Stealth[length=2.6mm]}] (-0.5,1) -- (-0.5,3.4);
    \draw[hundred,-{Stealth[length=2.6mm]}] (1,-0.333) -- (3.9,-0.333);
    \node[origindot] at (O) {};
    \node[left] at (O) {$\mathbf{0}$};
    \node[bdot] at (B) {};
    \node[above right] at (B) {$b$};
    \node at (3.95,-1.35) {$g^1$};
    \node[above] at (-1.8,3.6) {$g^2$};
  \end{scope}
\end{tikzpicture}}
{Example with exactly two infinite bounds among $L_1\,,L_2\,,U_1\,,U_2$\,, and where $\mathcal{K}_\beta(b) \subsetneqq \mathcal{B}_\beta$ (left) or $\mathcal{K}_\beta(b)= \mathcal{B}_\beta$
        (right).
  Note that $b_1=0$ for the left-hand drawing.
Furthermore, if we move $b$ to the origin, then $\mathcal{K}_\beta(b)=\cone(\{\mathbf{e}_1\,,\mathbf{e}_2\})\not=\mathcal{B}_\beta$\,.\label{fig:three}}
{}
\end{figure}

From Figures~\ref{fig:one}--\ref{fig:three}, we can see that
$b=\mathbf{0}$ can be rather problematic. But if $\mathcal{B}_\beta$
happens to be an orthant ($m\geq 2$), then $\mathcal{K}_\beta(b)=\mathcal{B}_\beta$\,, for all $b$ in that orthant (including $\mathbf{0}$): indeed, when $\mathcal{B}_\beta=\cone(\{\varepsilon_1\mathbf{e}_1\,,\ldots,\varepsilon_m\mathbf{e}_m\})$ is an orthant, with fixed signs $\varepsilon_i\in\{\pm1\}$, each column $h^i$ of $A_\beta^{-1}$ has a single non-zero entry, of sign $\varepsilon_i$\,; so one of $L_i\,,U_i$ is infinite for every $i$, and the corresponding bare generator is $\varepsilon_i\mathbf{e}_i$\,. It is present in $\mathcal K_\beta(b)$'s generating set \emph{regardless of $b$}, so $\mathcal K_\beta(b)\supset\cone(\{\varepsilon_1\mathbf{e}_1\,,\ldots,\varepsilon_m\mathbf{e}_m\})=\mathcal{B}_\beta$ for every $b$.

\begin{theorem}\label{thm:one}
    Suppose that $m=2$, $b\geq\mathbf{0}$, and at most one of $L_1\,, L_2\,, U_1\,, U_2$ is infinite. Then $\mathcal{K}_\beta(b)= \mathcal{B}_\beta$ if and only if $b\not=\mathbf{0}$.
\end{theorem}

\begin{proof}{Proof.}
$\mathcal{B}_\beta$ is a pointed 2-dimensional polyhedral cone, hence has two extreme rays; write $\bar{x}_\beta:=A_\beta^{-1}b=(\bar{x}_{\beta_1}\,,\bar{x}_{\beta_2})^\top\geq\mathbf{0}$ (well defined and non-negative, as $\beta$ is primal feasible for $b$), and $\bar k:=3-k$ for $k=1,2$.

\emph{Claim.} Suppose that $h^k$ has mixed signs for some $k$ (equivalently, both $L_k$ and $U_k$ are finite). A direct computation --- treating separately the two cases $h^k_k>0>h^k_{\bar k}$ and $h^k_k<0<h^k_{\bar k}$ --- demonstrates that $L_k$ and $U_k$ are exactly the perturbations that zero out, respectively, one of the two coordinates of $\bar{x}_\beta+\delta h^k$, and that at each of the two resulting points the remaining coordinate is strictly positive when $\bar{x}_\beta\not=\mathbf{0}$. Consequently, $b+L_k\mathbf{e}_k$ and $b+U_k\mathbf{e}_k$ lie, respectively, on the extreme rays of $g^k$ and $g^{\bar k}$ (in one order or the other, depending on the sign of $h^k_k$), and \emph{both are non-zero whenever $\bar{x}_\beta\not=\mathbf{0}$}, that is, whenever $b\not=\mathbf{0}$.

($\Rightarrow$) We demonstrate the contrapositive: if $b=\mathbf{0}$, then $\mathcal{K}_\beta(b)\not=\mathcal{B}_\beta$\,. If $b=\mathbf{0}$ then $\bar{x}_\beta=\mathbf{0}$, so every candidate value $-\bar{x}_{\beta_k}/h^i_k$ defining $L_i$ or $U_i$ ($i=1,2$) equals $0$; hence $L_i\in\{0,-\infty\}$ and $U_i\in\{0,+\infty\}$. By hypothesis at most one of $L_1\,,L_2\,,U_1\,,U_2$ is infinite, so at least three of the four equal $0$ exactly, contributing only the trivial generator $0$ to $\mathcal{K}_\beta(\mathbf{0})$; hence at most one non-trivial (signed unit) vector generates $\mathcal{K}_\beta(\mathbf{0})$, so $\dim\mathcal{K}_\beta(\mathbf{0})\leq1<2=\dim\mathcal{B}_\beta$\,, and $\mathcal{K}_\beta(\mathbf{0})\not=\mathcal{B}_\beta$\,.

($\Leftarrow$) Suppose that $b\not=\mathbf{0}$. By hypothesis at most one of $L_1\,,L_2\,,U_1\,,U_2$ is infinite, so at least one index $k$ has both $L_k\,,U_k$ finite, that is, $h^k$ has mixed signs. By the claim, $b+L_k\mathbf{e}_k$ and $b+U_k\mathbf{e}_k$ are non-zero points lying on the two \emph{different} extreme rays of $\mathcal{B}_\beta$\,. By Lemma~\ref{lemma:main}, $\mathcal{K}_\beta(b)=\mathcal{B}_\beta$\,.\Halmos
\end{proof}

\begin{example}\label{exa:four}
Figure~\ref{fig:four} depicts an interesting pair of related examples with two infinite bounds. It is a related pair of examples with the same two infinite bounds among $L_1\,,L_2\,,U_1\,,U_2$\,, and the same $b\in\mathbb{R}^2_{++}$\,.
The difference is in one
of the extreme rays of $\mathcal{B}_\beta$\,.
On the left, we have $\mathcal{K}_\beta(b) \subsetneqq \mathcal{B}_\beta$\,, and on the right, we have  $\mathcal{K}_\beta(b)= \mathcal{B}_\beta$\,.
Concretely, both drawings have $b=(1,1)^\top$, $L_1=-\tfrac12$\,, $U_1=+\infty$,
$L_2=-\infty$ and $U_2=1$, so that $v^1=(\tfrac12,1)^\top$ and $v^2=(1,2)^\top$ lie on
the same ray, that of $g^1$\,. On the left, $g^2$ lies in the open quadrant
$\mbox{\rm O}_{--}$\,, and no generator of $\mathcal{K}_\beta(b)$ lies on its ray;
indeed, by Theorem~\ref{thm:nonneg}(a) below, no non-negative $\tilde b$ in the
interior of this $\mathcal{B}_\beta$ has $\mathcal{K}_\beta(\tilde b)=
\mathcal{B}_\beta$\,. On the right, $g^2$ is a positive multiple of $-\mathbf{e}_2$\,, so
$A_\beta$ has a zero entry, and the bare generator $-\mathbf{e}_2$ of
$\mathcal{K}_\beta(b)$ (present because $L_2=-\infty$) spans the ray of $g^2$\,.
\end{example}

\begin{figure}
\FIGURE
{\begin{tikzpicture}[scale=0.9]
  \begin{scope}
    \coordinate (O) at (0,0);
    \coordinate (B) at (1,1);
    \fill[hatch] (0,0) -- (-0.4,-1.6) -- (0,-1.6) -- cycle;
    \draw[ray] (O) -- (1.5,3);
    \draw[ray] (O) -- (-0.5,-2.0);
    \draw[breach] (0.5,1) -- (1,1);
    \draw[breach,-{Stealth[length=2.6mm]}] (1,1) -- (2.9,1);
    \draw[breach] (1,2) -- (1,1);
    \draw[breach,-{Stealth[length=2.6mm]}] (1,1) -- (1,-1.7);
    \node[genpt] at (0.5,1) {};
    \node[genpt] at (1,2) {};
    \draw[hundred] (0,0) -- (1,2);
    \draw[hundred,-{Stealth[length=2.6mm]}] (1,2) -- (1.98,3.96);
    \draw[hundred,-{Stealth[length=2.6mm]}] (0,0) -- (0,-1.7);
    \node[origindot] at (O) {};
    \node[left] at (O) {$\mathbf{0}$};
    \node[bdot] at (B) {};
    \node[right] at (1.5,3) {$g^1$};
    \node[below] at (-0.5,-2.0) {$g^2$};
    \node[below right] at (B) {$b$};
  \end{scope}
  \begin{scope}[xshift=6.8cm]
    \coordinate (O) at (0,0);
    \coordinate (B) at (1,1);
    \draw[ray] (O) -- (1.5,3);
    \draw[ray] (O) -- (0,-2.0);
    \draw[breach] (0.5,1) -- (1,1);
    \draw[breach,-{Stealth[length=2.6mm]}] (1,1) -- (2.9,1);
    \draw[breach] (1,2) -- (1,1);
    \draw[breach,-{Stealth[length=2.6mm]}] (1,1) -- (1,-1.7);
    \node[genpt] at (0.5,1) {};
    \node[genpt] at (1,2) {};
    \draw[hundred] (0,0) -- (1,2);
    \draw[hundred,-{Stealth[length=2.6mm]}] (1,2) -- (1.98,3.96);
    \draw[hundred,-{Stealth[length=2.6mm]}] (0,0) -- (0,-1.7);
    \node[origindot] at (O) {};
    \node[left] at (O) {$\mathbf{0}$};
    \node[bdot] at (B) {};
    \node[right] at (1.5,3) {$g^1$};
    \node[below] at (0,-2.0) {$g^2$};
    \node[below right] at (B) {$b$};
  \end{scope}
\end{tikzpicture}}
{Related pair of examples with exactly two infinite bounds among $L_1\,,L_2\,,U_1\,,U_2$\,, $b\in\mathbb{R}^2_{++}$\,, and where $\mathcal{K}_\beta(b) \subsetneqq \mathcal{B}_\beta$ (left) or $\mathcal{K}_\beta(b)= \mathcal{B}_\beta$
        (right).\label{fig:four}}
{}
\end{figure}

Unfortunately, as Figure~\ref{fig:four} illustrates,
we cannot always distinguish between these two cases
based solely on $b$ and the $L_i$ and $U_i$\,,
when exactly two of $L_1\,, L_2\,, U_1\,, U_2$ are infinite: the two drawings
differ only in whether a column of $A_\beta$ is a coordinate direction. When all
entries of $A_\beta$ are non-zero, the next result, a characterization for exactly two
infinite bounds (compare Figures~\ref{fig:three}--\ref{fig:four}), demonstrates that $b$
and the $L_i$ and $U_i$ do decide the question.

\begin{theorem}\label{thm:two}
 Suppose that $m=2$, $b\geq\mathbf{0}$, all entries of $A_\beta$ are non-zero, and exactly two of $L_1\,, L_2\,, U_1\,, U_2$ are infinite (and so
 exactly one of $L_i$ and $U_i$ is finite for each $i$).
Let $v^i$ be the finite one of $b+L_i\mathbf{e}_i$ and $b+U_i\mathbf{e}_i$\,, for $i=1,2$.
Then $\mathcal{K}_\beta(b)= \mathcal{B}_\beta$ if and only if $v^1$ and $v^2$ are both non-zero and $v^1$ is not a positive multiple of $v^2$.
\end{theorem}

\begin{proof}{Proof.}
Because each $h^i$ has uniform signs, the finite bound for index $i$ zeroes out
one coordinate of $\bar{x}_\beta+\theta h^i$ (the first to reach $0$ as $\theta$ moves
from $0$ to that bound), so each non-zero $v^i$ lies on an extreme ray of
$\mathcal{B}_\beta$\,; and $v^i=\mathbf{0}$ exactly when both coordinates vanish
together. The generators of $\mathcal{K}_\beta(b)$ are $v^1$, $v^2$ and two bare
vectors $\pm\mathbf{e}_1\,,\pm\mathbf{e}_2$\,. As all entries of $A_\beta$ are
non-zero, neither extreme ray of $\mathcal{B}_\beta$ is a coordinate direction, so
the bare vectors lie on no extreme ray. If $v^1,v^2$ are non-zero and on different
extreme rays, Lemma~\ref{lemma:main} gives $\mathcal{K}_\beta(b)=\mathcal{B}_\beta$\,.
Conversely, if $\mathcal{K}_\beta(b)=\mathcal{B}_\beta$\,, then both extreme rays of
$\mathcal{B}_\beta$ are extreme rays of $\mathcal{K}_\beta(b)$, hence each is the ray of
some generator; the bare vectors are excluded, so each ray carries one of $v^1,v^2$,
which forces $v^1,v^2$ non-zero and on different rays. Finally, two non-zero
$v^i$ on extreme rays lie on different rays if and only if neither is a positive
multiple of the other.\Halmos
\end{proof}

We do have the following existential result, which is constructive but uses knowledge of $\mathcal{B}_\beta$\,.

\begin{theorem}\label{thm:exist}
For $m=2$, there is always a choice of $b$ (not necessarily non-negative), so that $\mathcal{K}_\beta(b)=\mathcal{B}_\beta$\,.
\end{theorem}

\begin{proof}{Proof.}
We consider the extreme-ray generators $g^i$ of
$\mathcal{B}_\beta$\,, for $i=1,2$.

If the angle between $g^1$ and $g^2$ is acute, then
there exists a vertical or horizontal line
that meets $\cone(\{g^i\})$ at a non-zero point,
say $b^i$, for $i=1,2$.
Then, choosing $b$ to be anywhere on the line segment
between $b^1$ and $b^2$ (even an endpoint will do),
for such a point $b$, we have $\mathcal{K}_\beta(b)=\mathcal{B}_\beta$\,.

If the angle between $g^1$ and $g^2$ is $\pi/2$,
then the same construction works, unless
$\mathcal{B}_\beta$ is a quadrant, in which case we simply
choose $b$ to be any point in the interior of $\mathcal{B}_\beta$\,,
and we have $\mathcal{K}_\beta(b)=\mathcal{B}_\beta$\,.

So, we may assume now that the angle between $g^1$ and $g^2$ is obtuse.
Suppose first that $\mathcal{B}_\beta$ contains a quadrant $Q$, say
$Q=\mbox{\rm O}_{++}$ after reflecting coordinates, and choose $b$ in the interior
of $Q$. As $\mathcal{B}_\beta$ is pointed with angle less than $\pi$, its two
extreme rays lie in the two quadrants adjacent to $Q$ (possibly on the boundary
axes of $Q$): one, say $g^1$, has $g^1_1>0$ and $g^1_2\leq0$, and the other has
$g^2_1\leq0$ and $g^2_2>0$. Moving $b$ along $-\mathbf{e}_2$ stays in
$\mathcal{B}_\beta$ until the vertical line $y_1=b_1>0$ meets the ray of $g^1$,
which it does at a non-zero point because $g^1_1>0$; so that point is the
finite-bound point for index $2$, and it lies on the ray of $g^1$. Symmetrically,
moving $b$ along $-\mathbf{e}_1$ reaches a non-zero point on the ray of $g^2$.
Hence both extreme rays are hit, and $\mathcal{K}_\beta(b)=\mathcal{B}_\beta$ by
Lemma~\ref{lemma:main}.

Otherwise $\mathcal{B}_\beta$ contains no quadrant. As its angle exceeds $\pi/2$,
it contains at least one coordinate direction; it cannot contain two, as two
adjacent ones would give a quadrant and two opposite ones would contradict
pointedness. So $\mathcal{B}_\beta$ contains exactly one half-axis, say the
positive $\mathbf{e}_2$-axis in its interior after reflecting coordinates, and
its two extreme rays then lie in the open quadrants $\mbox{\rm O}_{++}$ and
$\mbox{\rm O}_{-+}$ adjacent to that half-axis, so that $g^1_2\,,g^2_2>0$.
Choose $b:=\mathbf{e}_2$\,. Moving $b$ along $\pm\mathbf{e}_1$ stays in
$\mathcal{B}_\beta$ until the horizontal line $y_2=1$ meets the ray of $g^1$ or of
$g^2$\,, which it does at non-zero points because $g^1_2\,,g^2_2>0$; these are the
two finite-bound points for index $1$. Hence both extreme rays are hit, and
$\mathcal{K}_\beta(b)=\mathcal{B}_\beta$ by Lemma~\ref{lemma:main}.\Halmos
\end{proof}

Theorem~\ref{thm:exist} guarantees some $b$ (possibly with a negative
entry) achieves $\mathcal{K}_\beta(b)=\mathcal{B}_\beta$\,, but says
nothing about whether a \emph{non-negative} $b$ can do so ---
which is exactly the regime of Theorems~\ref{thm:one} and
\ref{thm:two}. We now settle this, for the two-infinite-bounds case,
in terms of the sign pattern of $A_\beta$ alone; this also explains
the phenomenon illustrated by Figure~\ref{fig:four} algebraically, rather
than only pictorially.

Recall $A_\beta=[g^1\mid g^2]$, so that
$\det A_\beta = g^1_1g^2_2-g^2_1g^1_2$\,. Then
\[
h^1=\tfrac{1}{\det A_\beta}(g^2_2\,,-g^1_2)^\top\,,\qquad
h^2=\tfrac{1}{\det A_\beta}(-g^2_1\,,g^1_1)^\top\,.
\]
We are in the two-infinite-bounds case of Theorem~\ref{thm:two}
precisely when $h^1$ and $h^2$ each have uniform signs; the displayed
formulae then force $g^1_1\,,g^2_1$ to have opposite signs and
$g^1_2\,,g^2_2$ to have opposite signs. Equivalently, $g^1$ and $g^2$
lie in \emph{antipodal quadrants} of $\mathbb{R}^2$: either one of them lies
in $\mbox{\rm O}_{++}$ and the other in $\mbox{\rm O}_{--}$ (that is,
$g^1_1g^1_2>0$ and $g^2_1g^2_2>0$), or one of them lies in
$\mbox{\rm O}_{+-}$ and the other in $\mbox{\rm O}_{-+}$ (that is,
$g^1_1g^1_2<0$ and $g^2_1g^2_2<0$). To avoid
trivial degeneracies, we also assume that all four entries of
$A_\beta$ are non-zero.

For $b\geq\mathbf{0}$ with $\bar{x}_\beta:=A_\beta^{-1}b=(\bar{x}_{\beta_1}\,,\bar{x}_{\beta_2})^\top$ satisfying
$\bar{x}_{\beta_1}>0$\,, let
\[
t:=\bar{x}_{\beta_2}/\bar{x}_{\beta_1}\,\qquad
r_1:=\left|g^1_2/g^2_2\right|,\qquad
r_2:=\left|g^1_1/g^2_1\right|.
\]

\begin{lemma}\label{lemma:ray}
In the two-infinite-bounds case, $v^1$ lies on the ray of $g^1$ if
$t<r_1$ and on the ray of $g^2$ if $t>r_1$\,; symmetrically,
$v^2$ lies on the ray of $g^1$ if $t<r_2$ and on the ray of $g^2$
if $t>r_2$\,. Consequently, $v^1$ and $v^2$ lie on \emph{different}
extreme rays of $\mathcal{B}_\beta$ if and only if $t$ lies strictly
between $r_1$ and $r_2$\,.
\end{lemma}

\begin{proof}{Proof.}
In the two-infinite-bounds case each $h^i$ has uniform signs and, as all entries of
$A_\beta$ are non-zero, no zero entry. If $h^i\geq\mathbf{0}$, then
$L_i=\max_k\{-\bar{x}_{\beta_k}/h^i_k\}=-\min_k\{\bar{x}_{\beta_k}/|h^i_k|\}$, and if
$h^i\leq\mathbf{0}$, then $U_i=\min_k\{-\bar{x}_{\beta_k}/h^i_k\}=\min_k\{\bar{x}_{\beta_k}/|h^i_k|\}$.
In either case, as $\delta$ moves from $0$ to the finite bound $\theta$, the first
coordinate of $\bar{x}_\beta+\delta h^i$ to reach $0$ is
$k^*_i:=\operatorname*{argmin}_k\{\bar{x}_{\beta_k}/|h^i_k|\}$. So the $k^*_i$-th coordinate
of $\bar{x}_\beta+\theta h^i$ vanishes, the other coordinate $3-k^*_i$ is non-negative,
and $v^i$ lies on the ray of $g^{3-k^*_i}$ (and $v^i=\mathbf{0}$ exactly when both
coordinates vanish together). Now $k^*_i=1$ if and only if
$\bar{x}_{\beta_1}/|h^i_1|<\bar{x}_{\beta_2}/|h^i_2|$, that is, if and only if
$t>|h^i_2|/|h^i_1|$; and $|h^1_2|/|h^1_1|=r_1$\,, $|h^2_2|/|h^2_1|=r_2$ by the
formulae for $h^1,h^2$ above. This gives the stated threshold rule, and the
``different extreme rays'' statement follows by checking the cases.\Halmos
\end{proof}

By Lemma~\ref{lemma:ray} and Theorem~\ref{thm:two},
$t\in\left(\min\{r_1\,,r_2\},\max\{r_1\,,r_2\}\right)$ is exactly the
condition for $\mathcal{K}_\beta(b)=\mathcal{B}_\beta$\,.
So existence of a non-negative $b$ achieving equality reduces to
whether some \emph{achievable} $t$ --- that is, some $t\geq 0$ with
$b=g^1+tg^2\geq \mathbf{0}$ after normalizing $\bar{x}_{\beta_1}=1$ --- falls in that open
interval. Let
\[
T:=\{t\geq 0 ~:~ g^1+tg^2\geq \mathbf{0}\}.
\]

\begin{theorem}\label{thm:nonneg}
Suppose that $m=2$, exactly two of $L_1\,,L_2\,,U_1\,,U_2$ are infinite,
and all entries of $A_\beta$ are non-zero.
\begin{itemize}
\item[(a)] If one of $g^1,g^2$ lies in quadrant $\mbox{\rm O}_{++}$ and the
other in quadrant $\mbox{\rm O}_{--}$ (that is, $g^1_1g^1_2>0$ and
$g^2_1g^2_2>0$), then
$\mathcal{K}_\beta(b)\subsetneqq \mathcal{B}_\beta$ for \emph{every}
$b\geq\mathbf{0}$ with $\bar{x}_{\beta_1}\,,\bar{x}_{\beta_2}>0$: no non-negative interior right-hand
side achieves equality.
\item[(b)] If one of $g^1,g^2$ lies in quadrant $\mbox{\rm O}_{+-}$ and the
other in quadrant $\mbox{\rm O}_{-+}$ (that is, $g^1_1g^1_2<0$ and
$g^2_1g^2_2<0$), label them so that $g^1\in\mbox{\rm O}_{+-}$ and
$g^2\in\mbox{\rm O}_{-+}$\,. Then there exists $b\geq\mathbf{0}$ with
$\bar{x}_{\beta_1}\,,\bar{x}_{\beta_2}>0$ and $\mathcal{K}_\beta(b)=\mathcal{B}_\beta$ if and
only if
$
\det A_\beta = g^1_1g^2_2-g^2_1g^1_2 \; > \;0,
$
equivalently if and only if $r_1<r_2$\,, equivalently if and only if
$\mathcal{B}_\beta$ meets the open quadrant $\mathbb{R}^2_{++}$\,. If so, a
right-hand side $b\geq\mathbf{0}$ with $\bar{x}_{\beta_1}\,,\bar{x}_{\beta_2}>0$ has
$\mathcal{K}_\beta(b)=\mathcal{B}_\beta$ if and only if $b>\mathbf{0}$, that is,
if and only if $t=\bar{x}_{\beta_2}/\bar{x}_{\beta_1}\in(r_1\,,r_2)$.
\end{itemize}
\end{theorem}

\begin{proof}{Proof.}
(a) With $g^1\in \mbox{\rm O}_{++}$ (so $g^1_1\,,g^1_2>0$) and
$g^2\in \mbox{\rm O}_{--}$ (so $g^2_1\,,g^2_2<0$), the two defining
inequalities of $T$ both become upper bounds:
\[
g^1_1+tg^2_1\geq 0 \iff t\leq g^1_1/|g^2_1|=r_2\,,\qquad
g^1_2+tg^2_2\geq 0 \iff t\leq g^1_2/|g^2_2|=r_1\,,
\]
so $T=[0,\min\{r_1\,,r_2\}]$\,. This interval lies entirely at or
below $\min\{r_1\,,r_2\}$\,, so it cannot meet the open interval
$(\min\{r_1\,,r_2\},\max\{r_1\,,r_2\})$ required by Lemma~\ref{lemma:ray}. (Labeling instead $g^1\in \mbox{\rm O}_{--}$\,,
$g^2\in\mbox{\rm O}_{++}$ gives, by the same computation,
$T=[\max\{r_1\,,r_2\},\infty)$\,, which again only touches the
required open interval at its left endpoint, never entering it.)
Hence no admissible $t$ works, for either labeling.

(b) With $g^1\in\mbox{\rm O}_{+-}$ (so $g^1_1>0>g^1_2$) and
$g^2\in\mbox{\rm O}_{-+}$ (so $g^2_1<0<g^2_2$):
\[
g^1_1+tg^2_1\geq 0 \iff t\leq g^1_1/|g^2_1|=r_2\,,\qquad
g^1_2+tg^2_2\geq 0 \iff t\geq |g^1_2|/g^2_2=r_1\,,
\]
so $T=[r_1\,,r_2]$ (empty if $r_1>r_2$). Thus
$T\cap(r_1\,,r_2)\not=\varnothing$ if and only if
$r_1<r_2$\,, in which case
$T\cap(r_1\,,r_2)=(r_1\,,r_2)$ itself. Finally,
\[
r_1<r_2
\iff |g^1_2||g^2_1|<|g^1_1||g^2_2|
\iff g^1_2g^2_1<g^1_1g^2_2
\iff \det A_\beta>0\,,
\]
using that, under this labeling ($g^1\in\mbox{\rm O}_{+-}$ and
$g^2\in\mbox{\rm O}_{-+}$), we have $g^1_2\,,g^2_1<0$ and
$g^1_1\,,g^2_2>0$, so that $g^1_2g^2_1=|g^1_2||g^2_1|$ and
$g^1_1g^2_2=|g^1_1||g^2_2|$. For the geometric reading, normalize
$\bar{x}_{\beta_1}=1$, so $b=g^1+tg^2$ with $t>0$; the two inequalities above,
taken strictly, say that $b_2>0$ if and only if $t>r_1$ and $b_1>0$ if and only
if $t<r_2$\,. Hence $t\in(r_1\,,r_2)$ if and only if $b>\mathbf{0}$, and
$r_1<r_2$ if and only if some $b\in\mathcal{B}_\beta$ with
$\bar{x}_{\beta_1}\,,\bar{x}_{\beta_2}>0$ is strictly positive, that is, if and only
if $\mathcal{B}_\beta$ meets $\mathbb{R}^2_{++}$ (a point of $\mathbb{R}^2_{++}$ on
the boundary of $\mathcal{B}_\beta$ would lie on the ray of $g^1$ or $g^2$\,,
neither of which meets $\mathbb{R}^2_{++}$).\Halmos
\end{proof}

\begin{remark}\label{rem:m2boundary}
Part (a) complements Theorem~\ref{thm:exist}: for a whole family of
cones, no non-negative $b$ achieves $\mathcal{K}_\beta(b)=\mathcal{B}_\beta$\,, so the $b$ guaranteed by
Theorem~\ref{thm:exist} for such cones must have a negative entry. Part
(b) has a simple geometric reading: within the other family, a non-negative
interior $b$ achieves $\mathcal{K}_\beta(b)=\mathcal{B}_\beta$ exactly when $b$ is
strictly positive, and such a $b$ exists exactly when $\mathcal{B}_\beta$
contains a strictly positive vector; the sign of $\det A_\beta$ under the stated
labeling is just an algebraic test for this. (If $\det A_\beta<0$ in case (b),
then $\mathcal{B}_\beta\cap\mathbb{R}^2_+=\{\mathbf{0}\}$, so no non-zero
non-negative right-hand side is feasible at all.) Figure~\ref{fig:four} is not
an instance of part (b): its left-hand cone is an instance of part (a), and its
right-hand cone is discussed below.

The boundary cases excluded above behave quite differently from each other. If
$\bar{x}_\beta$ has a zero entry, then both finite bounds are $0$, because every
$h^i_k$ is non-zero; so $v^1=v^2=b$, and Theorem~\ref{thm:two} gives
$\mathcal{K}_\beta(b)\not=\mathcal{B}_\beta$\,. Such right-hand sides never achieve
equality, and requiring $\bar{x}_\beta>\mathbf{0}$ loses nothing.

A zero entry of $A_\beta$\,, by contrast, changes the picture. At $m=2$ it makes a
column a coordinate direction, $g^j=c\,\mathbf{e}_i$\,, and then, as in
Proposition~\ref{prp:m3axis}(a), $h^i=\mathbf{e}_j/c$: one bound for index $i$ is
infinite, and its bare generator spans the ray of $g^j$, while the finite bound
moves $b$ to $\bar{x}_{\beta_{j'}}g^{j'}$, which lies on the other ray ($j'\not=j$).
Hence $\mathcal{K}_\beta(b)=\mathcal{B}_\beta$ for every $b$ with
$\bar{x}_\beta>\mathbf{0}$. If instead $b\not=\mathbf{0}$ lies on the boundary of
$\mathcal{B}_\beta$\,, then $\mathcal{K}_\beta(b)=\mathcal{B}_\beta$ exactly when $b$
lies on the ray of the column that is not a coordinate direction. The right-hand cone of Figure~\ref{fig:four} is of this
kind: it is a limit of cones in case (a), for which no non-negative interior $b$
achieves equality, yet every interior $b$ achieves it there. So a zero entry of
$A_\beta$ can reverse the conclusion of part (a), and the outcome is not continuous
in $A_\beta$\,.
\end{remark}

%%%%%%%%%%%%%%%%%%%%%%%%%%%

\section{Transferring to $\bar{x}_\beta$-coordinates}\label{sec:transfer}

The results of this section hold for every $m$ ($m\geq3$ where stated), and the three sections that follow,
\S\ref{sec:general_m}, \S\ref{sec:m3} and \S\ref{subsec:tridiag}, all rest on them. We write $h^i_k$ for the $k$-th entry of the $i$-th column
of $H$, that is, $h^i_k=H_{ki}$\,.

Everything rests on transferring the defining relation into
$\bar{x}_\beta$-coordinates.

\begin{lemma}\label{lem:transfer}
If $b+\theta\mathbf{e}_i=\mu g^j$, then $\bar{x}_\beta+\theta h^i=\mu\mathbf{e}_j$\,.
\end{lemma}

\begin{proof}{Proof.}
Apply $H$ to both sides. The left side becomes $Hb+\theta H\mathbf{e}_i=
\bar{x}_\beta+\theta h^i$, and the right side becomes $\mu Hg^j=\mu\mathbf{e}_j$\,,
because $g^j=A_\beta \mathbf{e}_j$ gives $Hg^j=\mathbf{e}_j$\,.\Halmos
\end{proof}

Reading off coordinates, $\bar{x}_{\beta_k}=-\theta h^i_k$ for every $k\not=j$,
and $\mu=\bar{x}_{\beta_j}+\theta h^i_j$\,. So landing on the ray of $g^j$
requires the move along $\mathbf{e}_i$ to drive \emph{every} coordinate of
$\bar{x}_\beta$ other than the $j$-th to zero simultaneously; equivalently, the
vector $(\bar{x}_{\beta_k})_{k\not=j}$ must be parallel to $(h^i_k)_{k\not=j}$\,.

\begin{lemma}\label{lem:boundattained}
Let $\bar{x}_\beta\geq\mathbf{0}$, and suppose that $\bar{x}_\beta+\theta h^i=\mu\mathbf{e}_j$
with $\theta\not=0$, $\mu>0$, and $\bar{x}_{\beta_k}>0$ for some $k\not=j$.
Then $\theta=L_i$ if $\theta<0$, and $\theta=U_i$ if $\theta>0$; in particular
that bound is finite.
\end{lemma}

\begin{proof}{Proof.}
Because $\mu\mathbf{e}_j\geq\mathbf{0}$, we have $\bar{x}_\beta+\theta h^i\geq\mathbf{0}$, so
$\theta\in[L_i\,,U_i]$. Fix $k\not=j$ with $\bar{x}_{\beta_k}>0$; then
$\bar{x}_{\beta_k}=-\theta h^i_k>0$. If $\theta>0$, this forces $h^i_k<0$, so by
definition $U_i\leq-\bar{x}_{\beta_k}/h^i_k=\theta$; combined with $\theta\leq U_i$
this gives $\theta=U_i$\,. If $\theta<0$, then $h^i_k>0$, so
$L_i\geq-\bar{x}_{\beta_k}/h^i_k=\theta$, and $\theta=L_i$\,.\Halmos
\end{proof}

Lemma~\ref{lemma:main} says that $\mathcal{K}_\beta(b)=\mathcal{B}_\beta$ as soon as
each extreme ray $g^j$ is hit by a finite-bound point $b+\theta_i\mathbf{e}_i$\,.
The converse needs one proviso, because a bare generator $\pm\mathbf{e}_i$ of
$\mathcal{K}_\beta(b)$ can itself lie on an extreme ray of $\mathcal{B}_\beta$ when
that ray is a coordinate direction.

\begin{lemma}\label{lem:converse}
Suppose that no extreme ray of $\mathcal{B}_\beta$ is a coordinate direction (that
is, no column of $A_\beta$ has exactly one non-zero entry). If
$\mathcal{K}_\beta(b)=\mathcal{B}_\beta$\,, then every extreme ray $g^j$ is hit by a
finite-bound point $b+\theta_i\mathbf{e}_i$ with $\theta_i\in\{L_i\,,U_i\}$.
\end{lemma}

\begin{proof}{Proof.}
If $\mathcal{K}_\beta(b)=\mathcal{B}_\beta$\,, then each extreme ray of
$\mathcal{B}_\beta$ is an extreme ray of the finitely generated cone
$\mathcal{K}_\beta(b)$, hence is the ray of one of its generators. By hypothesis
that generator is not a bare $\pm\mathbf{e}_i$\,, so it is a finite-bound
point.\Halmos
\end{proof}

The hypothesis holds automatically when $H$ has no zero entry: if
$g^j=c\,\mathbf{e}_i$\,, then $h^i=H\mathbf{e}_i=\mathbf{e}_j/c$ has $m-1$ zero
entries. Under it, the two lemmas together say that
$\mathcal{K}_\beta(b)=\mathcal{B}_\beta$ exactly when each extreme ray $g^j$ is hit
by a finite-bound point; this assigns to each ray $j$ an index $\tau(j)$, so that
ray $j$ is hit by moving $b$ along $\mathbf{e}_{\tau(j)}$\,. For interior $b$ the
assignment $\tau$ cannot collapse.

\begin{lemma}\label{lem:bijection}
Let $m\geq3$, and suppose that no extreme ray of $\mathcal{B}_\beta$ is a coordinate
direction. If $\bar{x}_\beta>\mathbf{0}$ and $\mathcal{K}_\beta(b)=\mathcal{B}_\beta$\,,
then $\tau$ is a bijection, that is, a permutation of $\{1,\ldots,m\}$.
\end{lemma}

\begin{proof}{Proof.}
Suppose that $\tau(j)=\tau(j')=i$ with $j\not=j'$, say with bounds $\theta$ and $\theta'$.
By Lemma~\ref{lem:transfer}, $\bar{x}_\beta+\theta h^i=\mu_j\mathbf{e}_j$ and
$\bar{x}_\beta+\theta'h^i=\mu_{j'}\mathbf{e}_{j'}$\,. If $\theta=\theta'$ then
$\mu_j\mathbf{e}_j=\mu_{j'}\mathbf{e}_{j'}$\,, impossible for $j\not=j'$ with
$\mu_j\,,\mu_{j'}>0$; so $\theta\not=\theta'$. Subtracting,
$(\theta-\theta')h^i=\mu_j\mathbf{e}_j-\mu_{j'}\mathbf{e}_{j'}$\,. As $m\geq3$
there is an index $l\notin\{j,j'\}$, and the $l$-th coordinate of the right side
vanishes, so $h^i_l=0$. But then the $l$-th coordinate of
$\bar{x}_\beta+\theta h^i=\mu_j\mathbf{e}_j$ (note $l\not=j$) gives
$\bar{x}_{\beta_l}=-\theta h^i_l=0$, contradicting $\bar{x}_\beta>\mathbf{0}$.\Halmos
\end{proof}

The interiority hypothesis of Lemma~\ref{lem:bijection} is itself automatic as
soon as $H$ has no zero entry, which is the generic situation.

\begin{lemma}\label{lem:interiorforced}
Let $m\geq3$ and suppose that $H$ has no zero entry. If $b\in\mathcal{B}_\beta$
and $\mathcal{K}_\beta(b)=\mathcal{B}_\beta$\,, then $\bar{x}_\beta>\mathbf{0}$.
\end{lemma}

\begin{proof}{Proof.}
Suppose that $\bar{x}_{\beta_l}=0$ for some $l$. Choose any ray $j\not=l$; by
Lemma~\ref{lem:converse} (applicable, as $H$ has no zero entry) it is hit, say for index $i$, so
$\bar{x}_\beta+\theta_i h^i=\mu_j\mathbf{e}_j$ with $\mu_j>0$ by
Lemma~\ref{lem:transfer}. Its $l$-th coordinate (note $l\not=j$) reads
$\theta_ih^i_l=0$, and $h^i_l\not=0$, so $\theta_i=0$ and hence
$\bar{x}_\beta=\mu_j\mathbf{e}_j$\,. Now choose a ray $j'\not=j$ and, as
$m\geq3$, an index $l'\notin\{j,j'\}$; then $\bar{x}_{\beta_{l'}}=0$, and the
same argument applied to $j'$ gives $\bar{x}_\beta=\mu_{j'}\mathbf{e}_{j'}$\,,
which is incompatible with $\bar{x}_\beta=\mu_j\mathbf{e}_j$ and
$\mu_j\,,\mu_{j'}>0$. So no coordinate of $\bar{x}_\beta$ vanishes.\Halmos
\end{proof}

The lemmas combine into a characterization of tightness for interior
right-hand sides, in every dimension $m\geq3$. Here and in \S\ref{sec:m3}, when we
look for a witness we treat $\bar{x}_\beta$ as the unknown and recover the
right-hand side as $b:=A_\beta\bar{x}_\beta$\,, so that $\bar{x}_\beta=A_\beta^{-1}b$
as always.

\begin{theorem}\label{thm:generalchar}
Let $m\geq3$, and suppose that no column of $A_\beta$ has exactly one non-zero entry.
There exists $b$ with $\bar{x}_\beta>\mathbf{0}$ and
$\mathcal{K}_\beta(b)=\mathcal{B}_\beta$ if and only if there are a permutation
$\tau$ of $\{1,\ldots,m\}$, a vector $\bar{x}_\beta\in\mathbb{R}^m$ with
$\bar{x}_\beta>\mathbf{0}$, and real numbers $\theta_1\,,\ldots,\theta_m$ such that, for
every $j$,
\begin{equation}\label{eq:generalchar}
\bar{x}_{\beta_k}+\theta_{\tau(j)}h^{\tau(j)}_k=0\quad(k\not=j)
\qquad\text{and}\qquad
\mu_j:=\bar{x}_{\beta_j}+\theta_{\tau(j)}h^{\tau(j)}_j>0.
\end{equation}
Equivalently, $h^{\tau(j)}=(\mu_j\mathbf{e}_j-\bar{x}_\beta)/\theta_{\tau(j)}$ for every
$j$: up to the order of its columns and non-zero scalings of them, $H$ is the matrix
with columns $\mu_j\mathbf{e}_j-\bar{x}_\beta$ $(j=1,\ldots,m)$, for some
$\bar{x}_\beta>\mathbf{0}$ and $\mu>\mathbf{0}$. In that case $b:=A_\beta\bar{x}_\beta$ works,
ray $j$ is hit by moving $b$ along $\mathbf{e}_{\tau(j)}$\,, and, for each $\tau$, the
vector $\bar{x}_\beta$ is unique up to positive scaling.
\end{theorem}

\begin{proof}{Proof.}
($\Leftarrow$) Set $b:=A_\beta\bar{x}_\beta$\,, so that $A_\beta^{-1}b=\bar{x}_\beta>\mathbf{0}$.
For each $j$, \eqref{eq:generalchar} says that
$\bar{x}_\beta+\theta_{\tau(j)}h^{\tau(j)}=\mu_j\mathbf{e}_j$\,, and
$\theta_{\tau(j)}\not=0$ because $\bar{x}_{\beta_k}>0$ for $k\not=j$. Applying $A_\beta$
gives $b+\theta_{\tau(j)}\mathbf{e}_{\tau(j)}=\mu_jg^j$ with $\mu_j>0$, and by
Lemma~\ref{lem:boundattained}, $\theta_{\tau(j)}$ is the finite bound
$L_{\tau(j)}$ or $U_{\tau(j)}$\,. So every extreme ray of $\mathcal{B}_\beta$ carries
a finite-bound point, and $\mathcal{K}_\beta(b)=\mathcal{B}_\beta$ by
Lemma~\ref{lemma:main}.

($\Rightarrow$) By Lemma~\ref{lem:converse} every ray $j$ is hit by a finite-bound
point, and by Lemma~\ref{lem:bijection} the assignment $\tau$ is a permutation. By
Lemma~\ref{lem:transfer},
$\bar{x}_\beta+\theta_{\tau(j)}h^{\tau(j)}=\mu_j\mathbf{e}_j$ with $\mu_j>0$, which
is \eqref{eq:generalchar}.

The equivalent form is \eqref{eq:generalchar} solved for $h^{\tau(j)}$. For
uniqueness, fix $\tau$. The conditions for ray $1$ make $(\bar{x}_{\beta_k})_{k\not=1}$
proportional to $(h^{\tau(1)}_k)_{k\not=1}$\,, and those for ray $2$ fix
$\bar{x}_{\beta_1}/\bar{x}_{\beta_3}=h^{\tau(2)}_1/h^{\tau(2)}_3$ (here $m\geq3$ is used); together
they determine $\bar{x}_\beta$ up to scaling.\Halmos
\end{proof}

If $H$ has no zero entry, Lemma~\ref{lem:interiorforced} makes the interiority
hypothesis automatic, and Theorem~\ref{thm:generalchar} then describes every
$b\in\mathcal{B}_\beta$ with $\mathcal{K}_\beta(b)=\mathcal{B}_\beta$\,.

Theorem~\ref{thm:generalchar} is also an efficient test. As the uniqueness argument
shows, the candidate $\bar{x}_\beta$ is determined by $\tau(1)$ and $\tau(2)$ alone, so at
most $m(m-1)$ candidate directions arise, rather than one for each of the $m!$
permutations. For each, checking that $\bar{x}_\beta>\mathbf{0}$ and that every extreme
ray is hit --- that is, that each ray $j$ receives a finite-bound point
$\bar{x}_\beta+\theta h^i$ with a single non-zero entry, in position $j$ --- takes
$O(m^2)$ operations given $H$. So, once $H$ is known, whether some interior $b$
achieves equality is decided in $O(m^4)$ operations.

The equivalent form also measures how special tightness becomes as $m$ grows. The
bases admitting an interior witness are those whose inverses have the stated
columns, which depend on $3m-1$ parameters ($\bar{x}_\beta$ up to scaling, together with
$\mu$ and $\theta$) and the discrete choice of $\tau$, against the $m^2$ entries of
$H$. So one expects them to form a family of codimension $m^2-3m+1$: one at $m=3$,
where it is the single equation (ii) of Corollary~\ref{thm:m3char}, five at $m=4$,
and eleven at $m=5$. Tightness nevertheless occurs in every dimension:
Example~\ref{ex:two} below gives, for each $m\geq2$, a basis and an interior $b$
with $\mathcal{K}_\beta(b)=\mathcal{B}_\beta$\,.

None of the proofs above uses a particular value of $m$ beyond $m\geq3$ where
stated, and we will draw on them at that generality: Lemmas~\ref{lem:bijection}
and~\ref{lem:interiorforced} and Theorem~\ref{thm:generalchar} are used in
\S\ref{sec:general_m}, in \S\ref{sec:m3} and again in \S\ref{subsec:tridiag}.

\section{General dimension $m$}\label{sec:general_m}

Theorem~\ref{thm:exist} does not extend to $m\geq3$.
Theorem~\ref{thm:nonexist} exhibits, in every dimension $m\geq3$, cones for which
no $b$ works; Lemma~\ref{lem:zero} settles $b=\mathbf{0}$; and
Theorem~\ref{thm:notequal} demonstrates that, for every cone that is not an orthant,
equality fails for almost every $b$. For rational cones, Theorem~\ref{thm:rational}
gives a sufficient condition for failure stated in terms of $b$ and the bounds
$L_i\,,U_i$ alone. Example~\ref{ex:two} demonstrates the other side: at $b=\mathbf{0}$
the cone rule can capture only about a $2^{-(m-1)}$ fraction of
$\mathcal{B}_\beta$\,, yet an arbitrarily small move of $b$ restores equality.

\begin{theorem}\label{thm:nonexist}
For every $m\geq 3$, there exists a
(pointed simplicial) cone
$\mathcal{B}_\beta$\,, so that
$\mathcal{K}_\beta(b)\not=\mathcal{B}_\beta$\,,
for all $b\in\mathcal{B}_\beta$\,.
\end{theorem}

\begin{proof}{Proof.}
We define $\mathcal{B}_\beta$ via its
generators $g^1,g^2,\ldots,g^m$, as follows.
We choose algebraically independent
real numbers
$0<t_1<t_2<\cdots<t_m$\,.
Then we let
\[
g^k := (t_k^0=1,t_k^1,\ldots,t_k^{m-1})^\top,
\]
for $k=1,\ldots,m$.
The matrix $[g^1|g^2|\cdots| g^m]$ is
a (transposed) Vandermonde matrix, hence invertible, and so
$\mathcal{B}_\beta$ is well defined.
Ratios of components of $g^k$ are powers of
$t_k$ with integer exponents (in
$\{\pm 1,\pm 2,\ldots,\pm (m-1)\}$).

Moreover $H:=A_\beta^{-1}$ has no zero entry: each entry of $H$ is a cofactor of
a Vandermonde matrix divided by its determinant, and those cofactors are
non-zero polynomials in $t_1\,,\ldots,t_m$\,, so none can vanish at algebraically
independent arguments.

Now, suppose that for some $b\in \mathcal{B}_\beta$\, we have
$\mathcal{K}_\beta(b)=\mathcal{B}_\beta$\,.
By Lemma~\ref{lem:interiorforced}, $\bar{x}_\beta>\mathbf{0}$, and then by
Lemma~\ref{lem:bijection} the ray-to-index assignment supplied by
Lemma~\ref{lem:converse} (applicable, as $H$ has no zero entry) is a permutation
$\tau$ of $\{1,\ldots,m\}$: for each ray $j$ we have
$b+\theta_{\tau(j)}\mathbf{e}_{\tau(j)}=\mu_j g^j$ for some $\theta_{\tau(j)}\in\mathbb{R}$
and $\mu_j>0$.
Now, we can see that
$b$ and $\mu_jg^j$ are identical on
all components not indexed by $\tau(j)$, that is on
$m-1(\geq 2)$ components. If we consider the
sets of pairwise ratios between these
components, varying $j$, we will see a problem.
Just consider the rays $u,v,w$ hit along $\mathbf{e}_1\,,\mathbf{e}_2\,,\mathbf{e}_3$\,, that is,
$\tau(u)=1$, $\tau(v)=2$, $\tau(w)=3$. From ray $u$,
we can see that $b_2/b_3=
g^{u}_2/g^{u}_3$\,.
From ray $v$,
we can see that $b_3/b_1=
g^{v}_3/g^{v}_1$\,.
Multiplying, we can see that
$b_2/b_1= g^{u}_2/g^{u}_3 \times
g^{v}_3/g^{v}_1
$. But from ray $w$, we have
$b_2/b_1=g^{w}_2/g^{w}_1$\,.
So we conclude that
$g^{w}_2/g^{w}_1
= g^{u}_2/g^{u}_3 \times
g^{v}_3/g^{v}_1
$, which is a contradiction,
because ratios of $g^k$ are powers of $t_k$ with non-zero integer exponents, and the above equality yields a non-trivial polynomial relation with rational coefficients among distinct algebraically independent numbers.
Therefore $\mathcal{K}_\beta(b)\neq \mathcal{B}_{\beta}$\,.\Halmos
\end{proof}

\begin{lemma}\label{lem:zero}
If $b=\mathbf{0}$, then $\mathcal{K}_\beta(b)=\mathcal{B}_{\beta}$
if and only if $\mathcal{B}_\beta$ is an orthant.
\end{lemma}

\begin{proof}{Proof.}
For each $i=1,\ldots,m$, because $\bar{x}_\beta=\mathbf{0}$ when $b=\mathbf{0}$, every candidate value $-\bar{x}_{\beta_k}/h^i_k$ defining $L_i$ or $U_i$ equals $0$. Hence: if $h^i$ has uniform signs, exactly one of $L_i\,,U_i$ is infinite (the other equal to $0$), so index $i$ contributes the single generator $\pm\mathbf{e}_i$ (matching the sign of $h^i$) to $\mathcal{K}_\beta(\mathbf{0})$; if instead $h^i$ has mixed signs, both $L_i=U_i=0$, and index $i$ contributes nothing. Consequently
\[
\mathcal{K}_\beta(\mathbf{0})=\cone(\{\varepsilon_i\mathbf{e}_i ~:~ i\in S\})
\]
for some subset $S\subset\{1,\ldots,m\}$ and signs $\varepsilon_i\in\{\pm1\}$ ($i\in S$), where $S$ is the set of indices $i$ such that $h^i$ has uniform signs.

($\Rightarrow$) If $\mathcal{K}_\beta(\mathbf{0})=\mathcal{B}_\beta$\,, then because $\mathcal{B}_\beta$ is $m$-dimensional and $|S|\leq m$ coordinate vectors generate $\mathcal{K}_\beta(\mathbf{0})$, we must have $|S|=m$; hence $\mathcal{K}_\beta(\mathbf{0})=
\cone(\{\varepsilon_1\mathbf{e}_1\,,\ldots,\varepsilon_m\mathbf{e}_m\})=\mathcal{B}_\beta$ is an orthant.

($\Leftarrow$) Suppose that $\mathcal{B}_\beta$ is an orthant, so its $m$ generators (the columns of $A_\beta$) are, as rays, exactly $\varepsilon_1\mathbf{e}_1\,,\ldots,\varepsilon_m\mathbf{e}_m$ for some signs $\varepsilon_i$; equivalently, $A_\beta$ has exactly one non-zero entry in each row and column. Then $H=A_\beta^{-1}$ has the same structure, so each column $h^i$ has a single non-zero entry --- hence uniform signs --- and $i\in S$ for every $i$, giving $\mathcal{K}_\beta(\mathbf{0})=
\cone(\{\varepsilon_1\mathbf{e}_1\,,\ldots,\varepsilon_m\mathbf{e}_m\})=\mathcal{B}_\beta$\,.\Halmos
\end{proof}

Lemma~\ref{lem:zero} already yields, for every $\mathcal{B}_\beta$ that is not an
orthant, a right-hand side (namely $b=\mathbf{0}$) with
$\mathcal{K}_\beta(b)\not=\mathcal{B}_\beta$\,. The next result establishes that failure
is the typical behavior, for interior right-hand sides as well.

\begin{theorem}\label{thm:notequal}
Suppose that $m\geq 3$ and that $\mathcal{B}_\beta$ is not an orthant, and let
$g^j$ be a column of $A_\beta$ with at least two non-zero entries. If
$\mathcal{K}_\beta(b)=\mathcal{B}_\beta$\,, then
\[
b\in\bigcup_{i=1}^m \operatorname{span}\{g^j,\mathbf{e}_i\}.
\]
Consequently, the set of $b\in\mathbb{R}^m$ with
$\mathcal{K}_\beta(b)=\mathcal{B}_\beta$ is contained in a finite union of linear
subspaces of dimension $2<m$; in particular it is nowhere dense and has Lebesgue
measure zero, and $\mathcal{K}_\beta(b)\not=\mathcal{B}_\beta$ for almost every $b$
in the interior of $\mathcal{B}_\beta$\,.
\end{theorem}

\begin{proof}{Proof.}
Such a column exists: if every column of the non-singular matrix $A_\beta$ had
exactly one non-zero entry, then $\mathcal{B}_\beta$ would be an orthant.
Suppose that $\mathcal{K}_\beta(b)=\mathcal{B}_\beta$\,. The ray of $g^j$ is an
extreme ray of $\mathcal{B}_\beta=\mathcal{K}_\beta(b)$, and because
$\mathcal{K}_\beta(b)$ is generated by finitely many vectors, it is the ray of
one of those generators. That generator is not a bare $\pm\mathbf{e}_i$\,,
because $g^j$ has two non-zero entries; so it is $b+\theta\mathbf{e}_i$ for some
index $i$ and some finite $\theta\in\{L_i\,,U_i\}$, with
$b+\theta\mathbf{e}_i=\mu g^j$ for some $\mu>0$. Hence
$b=\mu g^j-\theta\mathbf{e}_i\in\operatorname{span}\{g^j,\mathbf{e}_i\}$. Each
of these $m$ subspaces has dimension $2$, as $g^j$ is not parallel to
$\mathbf{e}_i$\,, and $2<m$. A finite union of proper linear subspaces is closed,
nowhere dense and of measure zero, and the interior of $\mathcal{B}_\beta$ is a
non-empty open set, which gives the remaining claims.\Halmos
\end{proof}

Theorem~\ref{thm:notequal} uses no hypothesis on $A_\beta$ beyond
$\mathcal{B}_\beta$ not being an orthant. It is consistent with the positive
results that follow: for each assignment $\tau$ of rays to indices, an interior
witness is unique up to scaling (Theorem~\ref{thm:generalchar}), and the witness $b=\frac{\epsilon}{m-\epsilon}\mathbf{e}$ of
Example~\ref{ex:two} below lies in $\operatorname{span}\{g^i,\mathbf{e}_i\}$ for
every $i$.

\begin{theorem}\label{thm:rational}
Suppose that $m\geq 3$. Let $\mathcal{B}_\beta$
be a (pointed simplicial) rational cone that is not an orthant.
Let $b$ be an interior point of  $\mathcal{B}_\beta$\,, such that every point in
$
\{b+L_i\mathbf{e}_i~:~L_i>-\infty\}\cup\{b+U_i\mathbf{e}_i~:~U_i<+\infty\}
$
is not proportional to a rational vector. Then, $\mathcal{K}_\beta(b)\not=\mathcal{B}_\beta$\,.
\end{theorem}

\begin{proof}{Proof.}
 Assume, for contradiction, that $\mathcal{K}_\beta(b)=\mathcal{B}_{\beta}$\,.
 As $\mathcal{K}_\beta(b)=\mathcal{B}_{\beta}$\,, each extreme ray of $\mathcal{B}_\beta$ must be generated by one of the $2m$ vectors appearing in the representation of $\mathcal{K}_\beta(b)$\,. As $\mathcal{B}_{\beta}$ is not an orthant, vectors among $\{\pm \mathbf{e}_i\}_{i=1}^{m}$ cannot generate all its extreme rays. Hence a vector of the form $b+L_i\mathbf{e}_i$ or $b+U_i\mathbf{e}_i$ must generate at least one extreme ray of $\mathcal{B}_{\beta}$\,.
 By hypothesis, no such vector is proportional to a rational vector, whereas every extreme ray of the rational cone $\mathcal{B}_\beta$ is spanned by a rational vector. Thus, no vector of the form $b+L_i\mathbf{e}_i$ or $b+U_i\mathbf{e}_i$ can generate an extreme ray of $\mathcal{B}_{\beta}$\,, a contradiction. Therefore $\mathcal{K}_\beta(b)\neq \mathcal{B}_{\beta}$\,.\Halmos
\end{proof}

Motivated by the two drawings of Figure~\ref{fig:three},
we next see that $\mathcal{K}_\beta(b)$ can be a rather small fraction
of $\mathcal{B}_\beta$\,, but with an arbitrarily small perturbation of $b$,
we can have $\mathcal{K}_\beta(b)= \mathcal{B}_\beta$\,.

\begin{example}\label{ex:two}
For $m\geq 2$, consider the open half-space $\mathcal{H}:=\{y\in\mathbb{R}^m ~:~ \allowbreak \sum_{i=1}^m y_i> 0\}$.
For $0<\epsilon<m$, consider the linearly-independent points $g^{i}:=\mathbf{e}-(m-\epsilon)\mathbf{e}_i\in\mathcal{H}$, for
$i=1,\ldots,m$. It is easy to check that
these points (and hence also their conical hull, which we define to be $\mathcal{B}_\beta$) are arbitrarily close to $\mathcal{H}$
as $\epsilon\to0$. Here $H=A_\beta^{-1}=\bigl(\mathbf{e}\mathbf{e}^\top/\epsilon-I\bigr)/(m-\epsilon)$,
which is entrywise non-negative when $\epsilon\leq1$.
Now, for $b:=0$ and $0<\epsilon\leq1$, $\mathcal{K}_\beta(b)$ is precisely the non-negative orthant
(because $L_i=0$ and $U_i=+\infty$, for $i=1,\ldots,m$),
so we can make the ratio of the measure of $\mathcal{K}_\beta(b)$ to the measure
of $\mathcal{B}_\beta$ arbitrarily close to $1/2^{m-1}$ by choosing $\epsilon$ small.

But notice that this example is very sensitive to the choice of $b$.
If we choose instead $b:=\frac{\epsilon}{m-\epsilon}\mathbf{e}$
(for any $0<\epsilon<m$),
in the positive orthant, which we can make arbitrarily close to the origin
(by choosing $\epsilon$ small), we will have $L_i=-\epsilon$,
and in fact $b+L_i\mathbf{e}_i= \frac{\epsilon}{m-\epsilon}\mathbf{e}
-\epsilon \mathbf{e}_i=
\frac{\epsilon}{m-\epsilon}(\mathbf{e}-(m-\epsilon)\mathbf{e}_i)
=\frac{\epsilon}{m-\epsilon}g^i$;
as $\bar{x}_\beta=H b=\mathbf{e}/(m-\epsilon)>\mathbf{0}$, Lemma~\ref{lem:boundattained}
confirms that this step is indeed $L_i$\,, whatever the signs of $H$. Therefore $b+L_i\mathbf{e}_i$
is on the ray generated by $g^i$, for each generator $g^i$ of
$\mathcal{B}_\beta$\,, and hence $\mathcal{K}_\beta(b)= \mathcal{B}_\beta$\,.
\end{example}

%%%%%%%%%%%%%%%%%%%%%%%

\section{The case of dimension $m=3$}\label{sec:m3}

For $m=2$ and $b\geq\mathbf{0}$, Theorems~\ref{thm:one}, \ref{thm:two}
and~\ref{thm:nonneg}, together with Remark~\ref{rem:m2boundary} for bases with a zero
entry, settle completely when $\mathcal{K}_\beta(b)=\mathcal{B}_\beta$\,. We now do the same
for $m=3$. This is the first dimension in which the question has real content:
moving $b$ along a coordinate direction $\mathbf{e}_i$ changes only one entry of
$\bar{x}_\beta:=A_\beta^{-1}b$, so it generically drives $\bar{x}_\beta$ to the
relative interior of a \emph{facet} of $\mathcal{B}_\beta$ (where a single entry
of $\bar{x}_\beta$ vanishes), not to an \emph{extreme ray} (where $m-1\geq2$
entries must vanish simultaneously). For $m=2$ a facet already is an extreme ray,
which is exactly why Theorem~\ref{thm:exist} holds there; and
\S\ref{sec:general_m} demonstrated that for $m\geq3$ equality can fail for \emph{every}
$b$, and Theorem~\ref{thm:generalchar} says exactly when it holds for interior $b$.
For $m=3$ its conditions can be solved in closed form, which is what this section
does.

Throughout this section $m=3$ and $A_\beta$ is non-singular, and we use the
notation and the lemmas of \S\ref{sec:transfer}.

\subsection{The characterization}

As in \S\ref{sec:transfer}, $\tau$ denotes the assignment of rays to indices: ray
$j$ is hit by moving $b$ along $\mathbf{e}_{\tau(j)}$\,.

\begin{corollary}\label{thm:m3char}
Let $m=3$ with $A_\beta$ non-singular and no column of $A_\beta$ having exactly one
non-zero entry. There exists $b$ with $\bar{x}_\beta>\mathbf{0}$
and $\mathcal{K}_\beta(b)=\mathcal{B}_\beta$ if and only if there is a permutation
$\tau$ of $\{1,2,3\}$ such that
\begin{itemize}
\item[(i)] for each $j$, the two entries $h^{\tau(j)}_k$ $(k\not=j)$ are non-zero
and share a sign;
\item[(ii)] $h^{\tau(1)}_2h^{\tau(2)}_3h^{\tau(3)}_1=h^{\tau(1)}_3h^{\tau(2)}_1h^{\tau(3)}_2$\,;
\item[(iii)] $\mu_j:=\bar{x}_{\beta_j}+\theta_{\tau(j)}h^{\tau(j)}_j>0$ for each $j$,
where $\theta_{\tau(j)}:=-\bar{x}_{\beta_k}/h^{\tau(j)}_k$ for any $k\not=j$ and
\[
\bar{x}_\beta:=\bigl(h^{\tau(2)}_1/h^{\tau(2)}_3\,,\ h^{\tau(1)}_2/h^{\tau(1)}_3\,,\ 1\bigr)^\top.
\]
\end{itemize}
In that case $b:=A_\beta\bar{x}_\beta$ works, and $\bar{x}_\beta$ is unique up to positive
scaling.
\end{corollary}

\begin{proof}{Proof.}
This is Theorem~\ref{thm:generalchar} at $m=3$, with its conditions solved. If those
conditions hold for $\tau$, then for each $j$ the two entries
$h^{\tau(j)}_k=-\bar{x}_{\beta_k}/\theta_{\tau(j)}$ $(k\not=j)$ are non-zero and share a
sign, which is (i); eliminating $\theta_{\tau(j)}$ gives
$\bar{x}_{\beta_2}/\bar{x}_{\beta_3}=h^{\tau(1)}_2/h^{\tau(1)}_3$\,,
$\bar{x}_{\beta_1}/\bar{x}_{\beta_3}=h^{\tau(2)}_1/h^{\tau(2)}_3$ and
$\bar{x}_{\beta_1}/\bar{x}_{\beta_2}=h^{\tau(3)}_1/h^{\tau(3)}_2$\,; the first two, normalized by
$\bar{x}_{\beta_3}=1$, give the stated $\bar{x}_\beta$\,, the third is then (ii), and $\mu_j>0$ is
(iii). Conversely, by (i) the stated $\bar{x}_\beta$ is well defined and positive, and
$\theta_{\tau(j)}:=-\bar{x}_{\beta_k}/h^{\tau(j)}_k$ does not depend on $k\not=j$ --- for
$j=1,2$ by definition of $\bar{x}_\beta$\,, and for $j=3$ by (ii); so the conditions of
Theorem~\ref{thm:generalchar} hold, with $\mu_j>0$ by (iii). Uniqueness is part of
Theorem~\ref{thm:generalchar}.\Halmos
\end{proof}

Condition (i) fixes the common sign of the two entries $h^{\tau(j)}_k$ $(k\not=j)$,
so the only remaining freedom in the column $h^{\tau(j)}$ is the sign of its
$j$-th entry. This splits the four admissible orthants into two antipodal pairs,
and the split is exactly the one that decides whether (iii) is a real condition.

\begin{corollary}\label{crl:m3orthant}
Assume (i) and (ii) hold for $\tau$, and fix $j\in\{1,2,3\}$.
\begin{itemize}
\item[(a)] If the entry $h^{\tau(j)}_j$ has the \emph{opposite}
sign to $h^{\tau(j)}_k$ $(k\not=j)$, then $\mu_j>0$ holds automatically.
\item[(b)] Otherwise the column $h^{\tau(j)}$ has uniform signs, and $\mu_j>0$ is
equivalent to the magnitude bound
$\bigl|h^{\tau(j)}_j\bigr|<\bar{x}_{\beta_j}/\bigl|\theta_{\tau(j)}\bigr|$\,.
\end{itemize}
In particular, if (a) applies to every $j$, then (iii) holds automatically, so
$\mathcal{K}_\beta(b)=\mathcal{B}_\beta$ for $b=A_\beta\bar{x}_\beta$\,.
\end{corollary}

\begin{proof}{Proof.}
From $\bar{x}_{\beta_k}=-\theta_{\tau(j)}h^{\tau(j)}_k>0$ we get
$\operatorname{sign}\theta_{\tau(j)}=-\operatorname{sign}h^{\tau(j)}_k$\,. In case (a),
$\operatorname{sign}h^{\tau(j)}_j=-\operatorname{sign}h^{\tau(j)}_k=\operatorname{sign}\theta_{\tau(j)}$\,,
so $\theta_{\tau(j)}h^{\tau(j)}_j>0$ and
$\mu_j=\bar{x}_{\beta_j}+\theta_{\tau(j)}h^{\tau(j)}_j>\bar{x}_{\beta_j}>0$. In case (b) the
product is non-positive and $\mu_j>0$ reads
$\bar{x}_{\beta_j}>\bigl|\theta_{\tau(j)}h^{\tau(j)}_j\bigr|$\,.\Halmos
\end{proof}

The two branches have a familiar reading: $h^{\tau(j)}$ has mixed signs in case (a),
so both $L_{\tau(j)}$ and $U_{\tau(j)}$ are finite, while in case (b) it has
uniform signs, so exactly one of them is infinite. Indices with two finite
bounds therefore impose only a sign condition; an index with an infinite bound
imposes a magnitude condition as well.

Corollary~\ref{thm:m3char} excludes bases having a column with a single non-zero
entry. That case is not harder; it is easier, because such a column removes one
ray and one index from the problem at once.

\begin{proposition}\label{prp:m3axis}
Let $m=3$ with $A_\beta$ non-singular, and suppose that some column is a
coordinate direction, $g^j=c\,\mathbf{e}_i$ with $c\not=0$.
\begin{itemize}
\item[(a)] Then $h^i=\mathbf{e}_j/c$. For every $b$, the bare generator
$\operatorname{sign}(c)\,\mathbf{e}_i$ of $\mathcal{K}_\beta(b)$ spans the ray of
$g^j$; and, when $\bar{x}_\beta>\mathbf{0}$, the finite-bound point for index
$i$ is $\sum_{k\not=j}\bar{x}_{\beta_k}g^k$, which lies on no extreme ray.
\item[(b)] If all three columns are coordinate directions, then
$\mathcal{B}_\beta$ is an orthant, and therefore
$\mathcal{K}_\beta(b)=\mathcal{B}_\beta$ holds for every $b\in\mathcal{B}_\beta$\,.
\item[(c)] If exactly two columns are coordinate directions,
$g^j=c\,\mathbf{e}_i$ and $g^{j'}=c'\mathbf{e}_{i'}$\,, and $j''$, $i''$ are the
remaining indices, then there exists $b$ with $\bar{x}_\beta>\mathbf{0}$ and
$\mathcal{K}_\beta(b)=\mathcal{B}_\beta$ if and only if $h^{i''}_j$ and
$h^{i''}_{j'}$ are non-zero and share a sign --- equivalently, if and only if
$g^{j''}_i/c$ and $g^{j''}_{i'}/c'$ are non-zero and share a sign. In that case
every $\bar{x}_\beta>\mathbf{0}$ with
$\bar{x}_{\beta_j}/\bar{x}_{\beta_{j'}}=h^{i''}_j/h^{i''}_{j'}$ and
$\bar{x}_{\beta_{j''}}$ sufficiently large works.
\item[(d)] If exactly one column is a coordinate direction, $g^j=c\,\mathbf{e}_i$\,,
and $\{j',j''\}$, $\{i',i''\}$ are the remaining indices, then there exists $b$
with $\bar{x}_\beta>\mathbf{0}$ and $\mathcal{K}_\beta(b)=\mathcal{B}_\beta$ if and
only if, for one of the two bijections $\tau\colon\{j',j''\}\to\{i',i''\}$,
condition (i) of Corollary~\ref{thm:m3char} holds for $j'$ and $j''$, and condition
(iii) holds for $j'$ and $j''$ with
\[
\bar{x}_{\beta_j}:=1,\qquad
\bar{x}_{\beta_{j''}}:=h^{\tau(j')}_{j''}\big/h^{\tau(j')}_{j}\,,\qquad
\bar{x}_{\beta_{j'}}:=h^{\tau(j'')}_{j'}\big/h^{\tau(j'')}_{j}
\]
(positive by (i)). In that case $b:=A_\beta\bar{x}_\beta$ works, and $\bar{x}_\beta$ is unique
up to positive scaling for that $\tau$.
\end{itemize}
\end{proposition}

\begin{proof}{Proof.}
(a) From $Hg^j=\mathbf{e}_j$ and $g^j=c\,\mathbf{e}_i$ we get $h^i=\mathbf{e}_j/c$.
So moving along $\mathbf{e}_i$ changes only the $j$-th coordinate of $\bar{x}_\beta$\,: if $c>0$
then $L_i=-c\,\bar{x}_{\beta_j}$ and $U_i=+\infty$, and if $c<0$ then
$L_i=-\infty$ and $U_i=-c\,\bar{x}_{\beta_j}$\,. In either case the bare generator
is $\operatorname{sign}(c)\,\mathbf{e}_i$\,, the direction of $g^j$, and the finite
bound gives $\bar{x}_\beta-\bar{x}_{\beta_j}\mathbf{e}_j$\,, whose image under
$A_\beta$ is $\sum_{k\not=j}\bar{x}_{\beta_k}g^k$; with two positive coefficients
this is in the relative interior of a facet.

(b) This is the remark preceding Theorem~\ref{thm:one}.

(c), (d) Write $R$ for the rays that are not coordinate directions and $I$ for
the indices $i$ with $\mathbf{e}_i$ parallel to no column; by
non-singularity the axis columns use distinct unit vectors, so $|R|=|I|$. By
(a), every ray outside $R$ is spanned by a bare generator for every $b$, and no
index outside $I$ contributes a point on any extreme ray. Suppose
that $\mathcal{K}_\beta(b)=\mathcal{B}_\beta$ with $\bar{x}_\beta>\mathbf{0}$. Each ray
in $R$ is an extreme ray of $\mathcal{K}_\beta(b)$ that is not a coordinate
direction, so as in Lemma~\ref{lem:converse} it is hit by a finite-bound point,
necessarily for an index in $I$; and the argument of
Lemma~\ref{lem:bijection}, which uses only $\bar{x}_\beta>\mathbf{0}$ and $m\geq3$,
establishes that one index hits at most one ray. Hence the hits define a bijection
$\tau\colon R\to I$, and for each $j\in R$ the computation in the proof of
Corollary~\ref{thm:m3char} gives condition (i) and the parallelism
$(\bar{x}_{\beta_k})_{k\not=j}\parallel(h^{\tau(j)}_k)_{k\not=j}$\,, together with
$\mu_j>0$, which is (iii). Conversely, given such $\tau$ and $\bar{x}_\beta$\,, the
finite-bound points hit every ray in $R$ (Lemma~\ref{lem:boundattained}), the bare
generators span the rays outside $R$, and the argument of Lemma~\ref{lemma:main}
gives $\mathcal{K}_\beta(b)=\mathcal{B}_\beta$\,.

In (c), $R=\{j''\}$ and $I=\{i''\}$: the single parallelism fixes only
$\bar{x}_{\beta_j}/\bar{x}_{\beta_{j'}}$\,, positive by (i); and
$\theta_{i''}=-\bar{x}_{\beta_j}/h^{i''}_j$ does not involve
$\bar{x}_{\beta_{j''}}$\,, so $\mu_{j''}=\bar{x}_{\beta_{j''}}+\theta_{i''}h^{i''}_{j''}>0$
holds once $\bar{x}_{\beta_{j''}}$ is large enough. For the reformulation, take
$j=1$, $i=1$, $j'=2$, $i'=2$ after relabeling, so
$A_\beta=\begin{pmatrix}c&0&g^3_1\\0&c'&g^3_2\\0&0&g^3_3\end{pmatrix}$ with
$g^3_3\not=0$; solving $A_\beta h^3=\mathbf{e}_3$ gives
$h^3_1=-g^3_1/(c\,g^3_3)$ and $h^3_2=-g^3_2/(c'g^3_3)$, whose product has the sign
of $(g^3_1/c)(g^3_2/c')$. In (d), $R=\{j',j''\}$: the two parallelisms fix
$\bar{x}_{\beta_{j''}}/\bar{x}_{\beta_j}$ and $\bar{x}_{\beta_{j'}}/\bar{x}_{\beta_j}$\,,
hence $\bar{x}_\beta$ up to scale, and no compatibility condition like (ii) arises.\Halmos
\end{proof}

Together, Corollary~\ref{thm:m3char} and Proposition~\ref{prp:m3axis} decide, for
every non-singular $A_\beta$ at $m=3$, whether some interior $b$ achieves
$\mathcal{K}_\beta(b)=\mathcal{B}_\beta$\,.

Requiring $b$ itself to be non-negative costs only a further sign condition on the
entries of $A_\beta$\,, and no condition on magnitudes.

\begin{theorem}\label{thm:m3nonneg}
In the setting of Corollary~\ref{thm:m3char}, there exists $b\geq\mathbf{0}$ with
$\bar{x}_\beta>\mathbf{0}$ and $\mathcal{K}_\beta(b)=\mathcal{B}_\beta$ if and only if some
$\tau$ satisfies (i), (ii), (iii) and
\begin{itemize}
\item[(iv)] $g^j_k\geq0$ for every $j$ and every $k\not=\tau(j)$.
\end{itemize}
Equivalently, every negative entry of $A_\beta$ lies on the transversal
$\{(\tau(j),j):j=1,2,3\}$.
\end{theorem}

\begin{proof}{Proof.}
From $b+\theta_{\tau(j)}\mathbf{e}_{\tau(j)}=\mu_jg^j$ we get
$b_k=\mu_jg^j_k$ for every $k\not=\tau(j)$, and $\mu_j>0$, so
$b_k\geq0\iff g^j_k\geq0$. Every coordinate $k$ is covered, because there are two
rays $j$ with $\tau(j)\not=k$. This is (iv). As $\tau$ is a permutation, the
exempted positions $(\tau(j),j)$ form a transversal of $A_\beta$\,.\Halmos
\end{proof}

Condition (iv) is easy to test, because containment in a transversal is a purely
local condition on the sign pattern of $A_\beta$\,.

\begin{corollary}\label{crl:twoneg}
If some row or some column of $A_\beta$ contains two or more negative entries,
then there is no $b\geq\mathbf{0}$ with $\bar{x}_\beta>\mathbf{0}$ and
$\mathcal{K}_\beta(b)=\mathcal{B}_\beta$\,.
\end{corollary}

\begin{proof}{Proof.}
A set $S$ of cells of an $m\times m$ array is contained in a transversal
$\{(\tau(j),j):j=1,\ldots,m\}$ for some permutation $\tau$ if and only if $S$ has
at most one cell in each row and at most one in each column. The forward
implication is immediate, as $\tau$ is a bijection. Conversely, such an $S$ is
$\{(\sigma(j),j):j\in J\}$ for some $J\subset\{1,\ldots,m\}$ and some injective
$\sigma\colon J\to\{1,\ldots,m\}$; as $\sigma$ is injective, the unused columns
$\{1,\ldots,m\}\setminus J$ and the unused rows
$\{1,\ldots,m\}\setminus\sigma(J)$ are equinumerous, so any bijection between
them extends $\sigma$ to a permutation $\tau$ of $\{1,\ldots,m\}$.

Now suppose that some $b\geq\mathbf{0}$ as in the statement exists. By
Theorem~\ref{thm:m3nonneg} there is a permutation $\tau$ for which every negative
entry of $A_\beta$ lies on the transversal $\{(\tau(j),j)\}$, and hence, by the
above, $A_\beta$ has at most one negative entry in each row and in each column.\Halmos
\end{proof}

The converse of Corollary~\ref{crl:twoneg} fails, because a $\tau$ witnessing the
sign pattern must also satisfy (i)--(iii). Nonetheless the corollary rules out
non-negative right-hand sides at a glance, without examining any $\tau$.

For $m=2$ the transversal test is consistent with Theorem~\ref{thm:nonneg} but,
being only necessary, does not reproduce it. In case~(a) there, both negative
entries of $A_\beta$ lie in one column, which is not a transversal, and indeed no
non-negative interior witness exists. In case~(b) they lie on the anti-diagonal, a
transversal, so the test is passed; but witnesses exist only when
$\det A_\beta>0$, that is, when $\mathcal{B}_\beta$ meets $\mathbb{R}^2_{++}$\,. When
$\det A_\beta<0$ the test is passed although no non-zero non-negative right-hand
side is feasible at all.

\medskip

\begin{remark}\label{rem:m3boundary}
Corollary~\ref{thm:m3char} assumes $\bar{x}_\beta>\mathbf{0}$, that is, $b$ interior to
$\mathcal{B}_\beta$\,. On the boundary the combinatorics genuinely change. If
$\bar{x}_{\beta_l}=0$, then for any ray $j\not=l$ the relation
$\bar{x}_{\beta_l}=-\theta_i h^i_l$ forces $h^i_l=0$ or $\theta_i=0$, and, as in the
proof of Lemma~\ref{lem:interiorforced}, $H$ must then have a zero in row $l$; and Lemma~\ref{lem:bijection} can fail, because a single index $i$
with $h^i_l=0$ may then hit two rays, once via $L_i$ and once via $U_i$\,. For
instance, with
\[
H=\begin{pmatrix}1&2&0\\-1&3&0\\0&-1&1\end{pmatrix},\qquad
b=(0,1,1)^\top,\qquad \bar{x}_\beta=(2,3,0)^\top,
\]
index $1$ hits the ray of $g^1$ (with $\theta=3$) and the ray of $g^2$ (with
$\theta=-2$), while index $2$ hits the ray of $g^3$; so
$\mathcal{K}_\beta(b)=\mathcal{B}_\beta$ although $b$ lies on the boundary. Such
behavior is confined to the measure-zero set of bases whose inverse $H$ has a
zero entry.
\end{remark}

\subsection{Reading the criterion in $A_\beta$\,}

It is useful to record the same criterion in terms of $A_\beta$ rather than $H$.
For a permutation $\tau$ of $\{1,2,3\}$, let $u,v,w$ be the rays with
$\tau(u)=1$, $\tau(v)=2$, $\tau(w)=3$, and ask when a single $b$ can be moved along
$\mathbf{e}_{\tau(j)}$ to land on the ray of $g^j$, for each $j$, that is, when
$b+\theta_{\tau(j)}\mathbf{e}_{\tau(j)}=\mu_jg^j$ has a solution with $\mu_j>0$.
Comparing, for each $i$, the two components of $b$ not indexed by $i$ with those of
the ray hit along $\mathbf{e}_i$ gives three homogeneous linear equations in $b$,
collected as $M(\tau)\,b=\mathbf{0}$, where
\[
M(\tau):=
\begin{pmatrix}
0 & g^u_3 & -g^u_2\\
g^v_3 & 0 & -g^v_1\\
g^w_2 & -g^w_1 & 0
\end{pmatrix}.
\]

\begin{lemma}\label{lemma:detM}
$\det M(\tau) = g^u_2 g^v_3 g^w_1 - g^u_3 g^v_1 g^w_2$\,. The system
$M(\tau)b=\mathbf{0}$, which records the proportionality conditions for the pattern
$\tau$, has a solution $b\not=\mathbf{0}$ if and only if $\det M(\tau) = 0$\,. Such a
solution realizes $\tau$, with every $\mu_j>0$, only under further sign
conditions, which Corollary~\ref{thm:m3char} makes explicit.
\end{lemma}

\begin{proof}{Proof.}
Expanding $\det M(\tau)$ along its first row gives
$-g^u_3(g^v_1 g^w_2) - g^u_2(g^v_3g^w_1)\cdot(-1)= g^u_2 g^v_3 g^w_1 - g^u_3 g^v_1
g^w_2$\,, as claimed. The system $M(\tau)b=\mathbf{0}$ has a non-trivial solution if and only
if $\det M(\tau)=0$\,.\Halmos
\end{proof}

Eliminating $\theta$ instead of $\mu$ turns the same system into three
homogeneous equations in $\bar{x}_\beta$\,, with matrix
\[
N(\tau):=
\begin{pmatrix}
0 & h^{\tau(1)}_3 & -h^{\tau(1)}_2\\
h^{\tau(2)}_3 & 0 & -h^{\tau(2)}_1\\
h^{\tau(3)}_2 & -h^{\tau(3)}_1 & 0
\end{pmatrix},
\]
which has the form of $M(\tau)$ with $g$ replaced by $h$ and $(u,v,w)$ by
$(\tau(1),\tau(2),\tau(3))$; expanding,
$\det N(\tau)$ is the difference of the two sides of condition (ii). 

\begin{proposition}\label{prp:duality}
We have
$N(\tau)=-\frac{1}{\det A_\beta}\,D\,\Pi_\tau\,D\,M(\tau)\,A_\beta$\,, where
$D:=\operatorname{diag}(1,-1,1)$ and $\Pi_\tau$ is the permutation matrix with
$(\Pi_\tau v)_j=v_{\tau(j)}$\,. Consequently
\[
\det N(\tau)=-\operatorname{sign}(\tau)\,\frac{\det M(\tau)}{(\det A_\beta)^2}\,,
\]
so condition (ii) holds if and only if $\det M(\tau)=0$\,.
\end{proposition}

\begin{proof}{Proof.}
Row $i$ of $M(\tau)$ applied to $b$ is $D_{ii}(b\times g^{j})_i$ for the ray $j$ with
$\tau(j)=i$, and row
$j$ of $N(\tau)$ applied to $\bar{x}_\beta$ is
$D_{jj}(\bar{x}_\beta\times h^{\tau(j)})_j$\,, by inspection. Using
$(Ax)\times(Ay)=\det(A)A^{-\top}(x\times y)$ with $A=H$, and $\bar{x}_\beta=Hb$,
$h^{\tau(j)}=H\mathbf{e}_{\tau(j)}$\,,
\[
\bar{x}_\beta\times h^{\tau(j)}
=\det(H)\,H^{-\top}\bigl(b\times\mathbf{e}_{\tau(j)}\bigr)
=\tfrac{1}{\det A_\beta}\,A_\beta^\top\bigl(b\times\mathbf{e}_{\tau(j)}\bigr).
\]
Row $j$ of $A_\beta^\top$ is $g^j$, so the $j$-th coordinate is
\[
\tfrac{1}{\det A_\beta}\,g^j\cdot(b\times\mathbf{e}_{\tau(j)})
=\tfrac{1}{\det A_\beta}\det[g^j,b,\mathbf{e}_{\tau(j)}]
=-\tfrac{1}{\det A_\beta}(b\times g^j)_{\tau(j)}\,.
\]
With $i=\tau(j)$ this is
$-\frac{1}{\det A_\beta}$ times the corresponding entry for $M(\tau)$, which is the
stated matrix identity. Taking determinants of these $3\times3$ matrices gives
\[
\det N(\tau)=\Bigl(-\frac{1}{\det A_\beta}\Bigr)^{3}(\det D)^{2}\,\det\Pi_\tau\,
\det M(\tau)\,\det A_\beta\,,
\]
and with $(\det D)^2=1$ and $\det\Pi_\tau=\operatorname{sign}(\tau)$
this is the displayed formula; as $\det A_\beta\not=0$, the two conditions
coincide.\Halmos
\end{proof}

Corollary~\ref{thm:m3char} thus refines Lemma~\ref{lemma:detM}: the vanishing of
$\det M(\tau)$ is exactly condition (ii), and (i) and (iii) are the positivity
requirements that decide whether the resulting $b$ is usable. Non-generic bases
can force $\det M(\tau)=0$, and the most natural such family is the symmetric one.

\begin{proposition}\label{prop:symmetric}
Suppose that $m=3$ and $A_\beta$ is symmetric, with positive off-diagonal entries
$p:=A_{12}=A_{21}$\,, $q:=A_{13}=A_{31}$\,, and $r:=A_{23}=A_{32}$\,. Let
$
b:=(pq,\ pr,\ qr)^\top.
$
Then $b+\theta_i\mathbf{e}_i = \mu_i g^i$ has a solution with $\mu_i>0$, for
each $i=1,2,3$ (specifically, $\mu_1=r$, $\mu_2=q$, $\mu_3=p$). Hence,
by Lemmas~\ref{lem:boundattained} and~\ref{lemma:main}, if in addition $\bar{x}_\beta:=A_\beta^{-1}b>\mathbf{0}$ (so that $b$ is an
interior point of $\mathcal{B}_\beta$), then $\mathcal{K}_\beta(b)=\mathcal{B}_\beta$\,.
\end{proposition}

\begin{proof}{Proof.}
Take $\tau=\mathrm{id}$, so that $u=1,v=2,w=3$ in Lemma~\ref{lemma:detM}. By
symmetry of $A_\beta$\,,
\[
\det M(\mathrm{id}) = g^1_2g^2_3g^3_1 - g^1_3g^2_1g^3_2 = p\,r\,q - q\,p\,r = 0
\]
identically, with no genericity assumption on $p,q,r$ needed. Solving $M(\mathrm{id})b=\mathbf{0}$
directly gives, up to a positive scaling, $b=(pq,pr,qr)^\top$: indeed, the three rows of
$M(\mathrm{id})$ become $q\,b_2 - p\,b_3=0$, $r\,b_1-p\,b_3=0$, and $r\,b_1-q\,b_2=0$
(the third being a consequence of the first two), and $b=(pq,pr,qr)^\top$ satisfies all
three. Recovering the scalars: from $b_2=p\,r=\mu_1 g^1_2=\mu_1 p$ we get
$\mu_1=r$; from $b_3=q\,r=\mu_1g^1_3=\mu_1 q$ we get $\mu_1=r$ again
(consistently); similarly $\mu_2=q$ and $\mu_3=p$. As $p,q,r>0$ by
hypothesis, all three $\mu_i>0$\,. The points $b+\theta_i\mathbf{e}_i$
($i=1,2,3$) then generate positive multiples of $g^1,g^2,g^3$ respectively, that is,
one non-zero point on each extreme ray of $\mathcal{B}_\beta$\,. If
$\bar{x}_\beta>\mathbf{0}$, then each $\theta_i\not=0$ (as $b$ itself lies on no
extreme ray), so by Lemma~\ref{lem:boundattained} each $\theta_i$ is the finite bound
$L_i$ or $U_i$\,; these points are therefore generators of $\mathcal{K}_\beta(b)$, and
the conclusion follows from Lemma~\ref{lemma:main}.\Halmos
\end{proof}

In the language of Corollary~\ref{thm:m3char} this is the pattern $\tau=\mathrm{id}$,
for which (ii) holds identically; the auxiliary hypothesis $\bar{x}_\beta>\mathbf{0}$ of
Proposition~\ref{prop:symmetric} is exactly condition (i), and unwinding it gives
an explicit test: writing $X:=d_1r-pq$, $Y:=d_2q-pr$ and $Z:=d_3p-qr$ for the
diagonal entries $d_1\,,d_2\,,d_3$ of $A_\beta$\,, Cramer's rule gives
$\bar{x}_\beta=(YZ,XZ,XY)^\top/\det A_\beta$\,, so $\bar{x}_\beta>\mathbf{0}$ holds precisely
when $X,Y,Z$ share a sign and $\det A_\beta>0$.

\begin{example}\label{exa:sym3}
Let
\[
A_\beta:=
\begin{pmatrix}
10 & 1 & 2\\
1 & 10 & 3\\
2 & 3 & 10
\end{pmatrix},
\]
so $p=1$, $q=2$, $r=3$, and Proposition~\ref{prop:symmetric} gives
$b:=(pq,pr,qr)^\top=(2,3,6)^\top$. One computes
$\bar{x}_\beta=A_\beta^{-1}b=\tfrac{1}{218}(17,28,119)^\top>\mathbf{0}$\,, so $\beta$ is primal feasible
and $b$ is an interior point of $\mathcal{B}_\beta$\,, as required. Directly:
moving along $\mathbf{e}_1$\,, components $2,3$ of $b$ are $(3,6)^\top$, proportional to
$(g^1_2\,,g^1_3)^\top=(1,2)^\top$ with ratio $\mu_1=3=r$, giving $\theta_1=\mu_1
A_{11}-b_1=3\cdot 10-2=28$, and indeed $A_\beta^{-1}(b+28\,\mathbf{e}_1)=3\,\mathbf{e}_1$\,,
landing exactly on the extreme ray of $g^1$ (both other entries of $\bar{x}_\beta$ vanishing
simultaneously, not merely one). Similarly, index $2$ gives $\mu_2=2=q$,
$\theta_2=17$, landing on $2\,g^2$; and index $3$ gives $\mu_3=1=p$,
$\theta_3=4$, landing exactly on $g^3$\,. Thus $\mathcal{K}_\beta(b)=\mathcal{B}_\beta$
for this genuinely three-dimensional, non-orthant cone.
\end{example}

%%%%%%%%%%%%%%%%%%%%%%%%%%%

\section{Tridiagonal bases}\label{subsec:tridiag}

The negative results of \S\ref{sec:general_m} leave open whether some natural
family of bases might do better than the fully generic ones used there. A
natural candidate is bandwidth-1 (\emph{tridiagonal}) $A_\beta$: such bases
arise, for instance, from path-structured or interval-indexed LPs, and from
discretized second-order operators. We assume throughout this section that
$A_\beta$ is tridiagonal with all $2(m-1)$ off-diagonal entries \emph{and}
all $m$ diagonal entries non-zero (so that the sparsity pattern is not
itself degenerate), and we write
\[
W_j := \{j-1,j,j+1\}\cap\{1,\ldots,m\}
\]
for the (row) support of column $g^j$; note $|W_j|=3$ for $2\leq j\leq m-1$
and $|W_j|=2$ for $j\in\{1,m\}$.

\begin{lemma}\label{lemma:windowint}
For non-empty $J\subset\{1,\ldots,m\}$,
\[
\bigcap_{j\in J} W_j = [\max\{J\}-1,\,\min\{J\}+1]\cap\{1,\ldots,m\}.
\]
\end{lemma}

\begin{proof}{Proof.}
A point $k$ lies in $W_j$ for every $j\in J$ if and only if $|k-j|\leq1$
for every $j\in J$, that is, $k\geq j-1$ and $k\leq j+1$ for every $j\in J$;
equivalently $k\geq\max\{J\}-1$ and $k\leq\min\{J\}+1$.\Halmos
\end{proof}

\begin{theorem}\label{thm:tridiagchar}
Suppose that $A_\beta$ is tridiagonal with all diagonal and off-diagonal entries
non-zero.
\begin{itemize}
\item[(a)] If $m\geq4$, then $\mathcal{K}_\beta(b)\not=\mathcal{B}_\beta$ for every
$b\in\mathcal{B}_\beta$\,.
\item[(b)] If $m=3$, write
$A_\beta=\begin{pmatrix}d_1&p&0\\q&d_2&r\\0&s&d_3\end{pmatrix}$. There exists
$b$ with $\bar{x}_\beta>\mathbf{0}$ and $\mathcal{K}_\beta(b)=\mathcal{B}_\beta$ if
and only if
\begin{itemize}
\item[(i)] $pq\,d_3=rs\,d_1$\,;
\item[(ii)] $p/d_1>0$ and $s/d_3>0$\,;
\item[(iii)] $d_1d_2/(pq)<1$\,.
\end{itemize}
In that case $b$ is unique up to positive scaling, namely
$b=(p,\,pq/d_1\,,\,s)^\top$, and the extreme rays of $g^1,g^2,g^3$ are hit by
moving $b$ along $\mathbf{e}_3\,,\mathbf{e}_2\,,\mathbf{e}_1$\,, respectively.
\item[(c)] If $m=2$, there exists $b\in\mathcal{B}_\beta$ with
$\mathcal{K}_\beta(b)=\mathcal{B}_\beta$\,.
\end{itemize}
In particular, some tridiagonal $A_\beta$ of order $m$ admits a $b$ with
$\mathcal{K}_\beta(b)=\mathcal{B}_\beta$ if and only if $m\leq3$; but at $m=3$
whether a given $A_\beta$ does depends on $A_\beta$\,, through (i)--(iii).
\end{theorem}

\begin{proof}{Proof.}
Part (c) is Theorem~\ref{thm:exist}. For (a), let $m\geq4$ and suppose, for contradiction, that $\mathcal{K}_\beta(b)=\mathcal{B}_\beta$ for
some $b$; necessarily $b\not=\mathbf{0}$, because $\mathcal{B}_\beta$ is not an orthant
(each column has $|W_j|\geq2$ non-zero entries) and Lemma~\ref{lem:zero}
applies to $b=\mathbf{0}$. As no extreme ray is a coordinate direction (again because $|W_j|\geq2$),
Lemma~\ref{lem:converse} establishes that every extreme ray
$g^j$ is hit by a finite-bound point for some index $\tau(j)$; comparing
coordinates $k\not=\tau(j)$ forces $b_k=0$ for every $k\notin W_j\cup\{\tau(j)\}$.
We say that ray $j$ is \emph{own-hit} if $\tau(j)=j$, and \emph{cross-hit}
if $\tau(j)\not=j$. (We do not need $j\mapsto\tau(j)$ to be injective here; the
argument below uses only the support constraints it produces. When
$\bar{x}_\beta>\mathbf{0}$ it is in fact a bijection, by Lemma~\ref{lem:bijection}.)

Let $C:=\{j : \tau(j)\not=j\}$ (the cross-hit rays) and
$J:=\{1,\ldots,m\}\setminus C$ (the own-hit rays). For $j\in J$ we have $\tau(j)=j\in W_j$\,, so the
displayed constraint reads $b_k=0$ for $k\notin W_j$\,. For $j\in C$, because
$\tau(j)\not=j$, the coordinate $k=j$ is among those compared, giving
$b_j=\mu_j (A_\beta)_{jj}$ with $\mu_j>0$ and $(A_\beta)_{jj}\not=0$,
so $b_j\not=0$ (we refer to this last deduction as the \emph{cross-hit
observation}). Combining,
\[
C \ \subset\ \operatorname{supp}(b)\ \subset\ \bigcap_{j\in J} W_j\,,
\]
with the right side understood as $\{1,\ldots,m\}$ if $J=\emptyset$.

\emph{Case $J\not=\emptyset$.} Let $a:=\min\{J\}$, $c:=\max\{J\}$. By Lemma~\ref{lemma:windowint}, $\bigcap_{j\in J}W_j=[c-1,a+1]\cap\{1,\ldots,m\}$, so
$C\subset[c-1,a+1]$ forces
\[
J\ \supset\ \{1,\ldots,m\}\setminus[c-1,a+1]\ =\ [1,c-2]\cup[a+2,m].
\]
If $c\geq3$, then $1\in[1,c-2]\subset J$, forcing $a=1$; equivalently
(contrapositive), $a\not=1 \Rightarrow c\leq2$. Symmetrically, if $a\leq
m-2$, then $c=m$; equivalently, $c\not=m \Rightarrow a\geq m-1$. We check
the four cases for $(a,c)$ against these two implications: if $a=1$ and
$c\not=m$, the second implication forces $a\geq m-1$, so $1\geq m-1$,
i.e., $m\leq2$; if $a\not=1$ and $c=m$, the first forces $c\leq2$, so
$m\leq2$; if $a\not=1$ and $c\not=m$, both force $c\leq2$ and $a\geq m-1$,
so (as $a\leq c$) $m-1\leq2$, i.e., $m\leq3$. Each of these three cases
contradicts $m\geq4$, leaving only $a=1$ and $c=m$. Then
$J\supset[1,m-2]\cup[3,m]$, and because $m\geq4$, these two intervals intersect, covering $\{1,\ldots,m\}$ entirely, forcing $J=\{1,\ldots,m\}$,
i.e., $C=\emptyset$. But then $\operatorname{supp}(b)\subset
\bigcap_{j=1}^m W_j=[m-1,2]=\emptyset$ (as $m\geq4>3$), so $b=\mathbf{0}$, a
contradiction.

\emph{Case $J=\emptyset$} (every ray cross-hit): then
$\operatorname{supp}(b)=\{1,\ldots,m\}$ entirely, by the cross-hit
observation applied to every $j$. But we also need
$\{1,\ldots,m\}\subset W_j\cup\{\tau(j)\}$ for every $j$, that is,
$m-|W_j|\leq1$; for boundary $j\in\{1,m\}$, $|W_j|=2$, forcing $m\leq3$,
contradicting $m\geq4$.

Both cases are contradictory, so no such $b$ exists when $m\geq4$, which is (a).

For (b), let $m=3$, so $g^1=(d_1\,,q,0)^\top$, $g^2=(p,d_2\,,s)^\top$,
$g^3=(0,r,d_3)^\top$, and no column of $A_\beta$ is a coordinate direction.
Suppose first that $\bar{x}_\beta>\mathbf{0}$ and
$\mathcal{K}_\beta(b)=\mathcal{B}_\beta$\,. By Lemmas~\ref{lem:converse}
and~\ref{lem:bijection}, ray $j$ is hit by a finite-bound point
$b+\theta_{\tau(j)}\mathbf{e}_{\tau(j)}=\mu_jg^j$ with $\mu_j>0$, for a
permutation $j\mapsto\tau(j)$\,; comparing coordinates gives $b_k=\mu_jg^j_k$
for every $k\not=\tau(j)$\,. If $\tau(1)\not=3$, then $b_3=\mu_1g^1_3=0$; hence
$\tau(3)=3$ (otherwise $b_3=\mu_3d_3\not=0$), so $\tau(3)\not=1$ and
$b_1=\mu_3g^3_1=0$; by the same argument $\tau(1)=1$, so $\tau(2)=2$ and
$b_1=\mu_2p\not=0$, a contradiction. So $\tau(1)=3$; symmetrically $\tau(3)=1$; and
therefore $\tau(2)=2$. The comparisons now read
\[
b_1=\mu_1d_1=\mu_2p\,,\qquad b_2=\mu_1q=\mu_3r\,,\qquad
b_3=\mu_2s=\mu_3d_3\,.
\]
The first and third give $\mu_1/\mu_2=p/d_1$ and
$\mu_3/\mu_2=s/d_3$\,, which must be positive: this is (ii).
Substituting into the second gives $\mu_2pq/d_1=\mu_2rs/d_3$\,, which is
(i), and then $b=\mu_2(p,\,pq/d_1\,,\,s)^\top$. Write $u:=d_1d_2/(pq)$. Under
(i) we have $\det A_\beta=d_3\,pq\,(u-2)$, so $u\not=2$, and a direct computation
gives
\begin{equation}\label{eq:tridiagxbar}
\bar{x}_\beta=A_\beta^{-1}b=\frac{\mu_2}{2-u}\Bigl(\frac{p}{d_1}(1-u),\;1,\;
\frac{s}{d_3}(1-u)\Bigr)^{\!\top}.
\end{equation}
Given (ii), $\bar{x}_\beta>\mathbf{0}$ holds if and only if $2-u>0$ and $1-u>0$,
that is, $u<1$, which is (iii). This proves necessity, and uniqueness of $b$ up
to positive scaling.

Conversely, assume (i)--(iii) and let $b:=(p,\,pq/d_1\,,\,s)^\top$. By
\eqref{eq:tridiagxbar} with $\mu_2=1$, $\bar{x}_\beta>\mathbf{0}$. Moreover,
using $rs/d_3=pq/d_1$ (from (i)) in the last identity,
\[
b-s\,\mathbf{e}_3=\frac{p}{d_1}\,g^1\,,\qquad
b+\Bigl(d_2-\frac{pq}{d_1}\Bigr)\mathbf{e}_2=g^2\,,\qquad
b-p\,\mathbf{e}_1=\frac{s}{d_3}\,g^3\,,
\]
with positive multipliers by (ii), and with non-zero steps $-s$,
$d_2-pq/d_1=(pq/d_1)(u-1)$ and $-p$. By Lemma~\ref{lem:boundattained}, each
step is the finite bound $L_i$ or $U_i$ for its index, so every extreme ray of
$\mathcal{B}_\beta$ carries a finite-bound point, and
$\mathcal{K}_\beta(b)=\mathcal{B}_\beta$ by Lemma~\ref{lemma:main}.\Halmos
\end{proof}

In the language of \S\ref{sec:m3}, (b) uses the pairing $\tau(1)=3$,
$\tau(2)=2$, $\tau(3)=1$, and (i) is condition (ii) of Corollary~\ref{thm:m3char}
for it, equivalently $\det M(\tau)=0$
(Lemma~\ref{lemma:detM}). Condition (i) holds automatically when $A_\beta$ is
persymmetric ($d_1=d_3$\,, $p=r$, $q=s$), but it is a genuine equation on the
entries in general, so a generic tridiagonal $3\times3$ basis admits no interior
$b$ with $\mathcal{K}_\beta(b)=\mathcal{B}_\beta$\,. The proof also demonstrates why
Proposition~\ref{prop:symmetric}'s identity pairing cannot be used here: the
corner zeros of $A_\beta$ force each end ray to be hit along the
\emph{opposite} end coordinate. By (i), condition (iii) can equally be written
$d_2d_3/(rs)<1$, which makes the symmetry under reversal of the coordinates
visible.

\begin{example}\label{exa:tridiag3}
Let
\[
A_\beta:=\begin{pmatrix}10&2&0\\3&-10&2\\0&3&10\end{pmatrix},\qquad
b:=(10,3,15)^\top.
\]
One computes $\bar{x}_\beta=A_\beta^{-1}b=(53/56,\,15/56,\,159/112)^\top>0$, so $b$ is an
interior point of $\mathcal{B}_\beta$\,. Direct calculation gives
$L_1=-10$, landing on the ray of $g^3$; $L_2=-53$, landing on the ray of
$g^2$; and $L_3=-15$, landing on the ray of $g^1$ --- each finite bound
lands exactly on a different extreme ray (using only the $L_i$'s), so
$\mathcal{K}_\beta(b)=\mathcal{B}_\beta$ by Lemma~\ref{lemma:main}. This agrees
with Theorem~\ref{thm:tridiagchar}(b): here $pq\,d_3=rs\,d_1=60$, $p/d_1\,,s/d_3>0$,
$d_1d_2/(pq)=-50/3<1$, and $b=5\,(p,\,pq/d_1\,,\,s)^\top$. A diagonal of mixed
signs is not needed: for instance,
$A_\beta=\begin{pmatrix}1&1&0\\2&1&2\\0&1&1\end{pmatrix}$ and $b=(1,2,1)^\top$
satisfy (i)--(iii), with $d_1d_2/(pq)=1/2$ and
$\bar{x}_\beta=\frac13(1,2,1)^\top>\mathbf{0}$.
\end{example}

\begin{remark}\label{rem:tridiagpatterns}
Theorem~\ref{thm:tridiagchar} settles the tridiagonal case: equality fails for
every $b$ once $m\geq4$, and at $m=3$ it holds for some interior $b$ exactly
under (i)--(iii). For $b$ on the boundary of $\mathcal{B}_\beta$ at $m=3$, note
that every entry of $H$ other than $H_{11}=(d_2d_3-rs)/\det A_\beta$ and
$H_{33}=(d_1d_2-pq)/\det A_\beta$ is automatically non-zero; so whenever these two
are non-zero, Lemma~\ref{lem:interiorforced} shows that any $b$ with
$\mathcal{K}_\beta(b)=\mathcal{B}_\beta$ is interior, and (b) covers every $b$.
In the language of the proof of (a), an exhaustive check (over all $2^m$ subsets
$C$) demonstrates that, for $m=3$, the support condition
$C\subset\operatorname{supp}(b)\subset\bigcap_{j\in J}W_j$ is satisfied by four
patterns, $C=\emptyset$, $\{2\}$, $\{1,3\}$ and $\{1,2,3\}$, and for $m\geq4$ by
none. Comparing \emph{every} coordinate $k\not=\tau(j)$ in
$b+\theta_{\tau(j)}\mathbf{e}_{\tau(j)}=\mu_jg^j$, zero and non-zero alike, gives the
sharper condition $\operatorname{supp}(b)\setminus\{\tau(j)\}=W_j\setminus\{\tau(j)\}$ for
every $j$, which rules out $C=\emptyset$ and $C=\{2\}$ and leaves $C=\{1,3\}$ and
$C=\{1,2,3\}$. For interior $b$, the assignment $j\mapsto\tau(j)$ is a bijection, and
the proof of (b) shows that only $C=\{1,3\}$ survives, with
$(\tau(1)\,,\tau(2)\,,\tau(3))=(3,2,1)$: both derangements realizing $C=\{1,2,3\}$ fail by
tridiagonal sparsity alone.
\end{remark}

We can identify a natural, easily verified sufficient condition for a
tridiagonal basis to have an entrywise-positive inverse, by turning to a
classical class of matrices \citep*{BP1994}. Recall that a \emph{Z-matrix} is a square
matrix with non-positive off-diagonal entries, and that a (non-singular)
\emph{M-matrix} is a Z-matrix whose inverse exists and is entrywise
non-negative. A standard sufficient
condition for a Z-matrix to be an M-matrix is that it
be strictly diagonally dominant with positive diagonal. A classical
consequence of Perron--Frobenius theory is that an \emph{irreducible}
M-matrix --- one whose associated directed graph is strongly connected ---
has an entrywise \emph{strictly} positive inverse,
not merely a non-negative one; because a tridiagonal matrix with non-zero
off-diagonal entries is automatically irreducible (its associated graph is
a path), every tridiagonal M-matrix therefore has an entrywise-positive
inverse.

\begin{theorem}\label{thm:posinv}
Suppose that $A_\beta$ is tridiagonal, with all diagonal and off-diagonal
entries non-zero, and $H:=A_\beta^{-1}$ is entrywise positive. Then
$\mathcal{K}_\beta(b)\not=\mathcal{B}_\beta$ for every $b\in\mathcal{B}_\beta$
and every $m\geq3$.
\end{theorem}

\begin{proof}{Proof.}
For $m\geq4$, this is the special case of Theorem~\ref{thm:tridiagchar}(a)
where $A_\beta$ happens to have an entrywise-positive inverse; the
conclusion there does not use any sign information about $H$. It remains to
rule out $m=3$; write $A_\beta=\begin{pmatrix}d_1&p&0\\q&d_2&r\\0&s&d_3
\end{pmatrix}$, so $g^1=(d_1\,,q,0)^\top$, $g^2=(p,d_2\,,s)^\top$, $g^3=(0,r,d_3)^\top$. (At
$m=3$, non-zero diagonal entries are in fact automatic: e.g., if $d_1=0$,
the identity $A_\beta H=I$ restricted to row $1$, column $2$ reads
$pH_{22}=0$, forcing $p=0$ because $H_{22}>0$, contradicting $p\not=0$; so
the hypothesis is redundant here, although we state it for uniformity with
the general-$m$ case.) As $H$ is entrywise positive it has no zero entry, so
Lemma~\ref{lem:interiorforced} shows that any $b$ with
$\mathcal{K}_\beta(b)=\mathcal{B}_\beta$ has $\bar{x}_\beta>\mathbf{0}$, and then
Theorem~\ref{thm:tridiagchar}(b) requires $p/d_1>0$. But the $(1,2)$ entry of
$A_\beta H=I$ reads $d_1H_{12}+pH_{22}=0$ (row $1$ of $A_\beta$ is
$(d_1\,,p,0)$), so $p/d_1=-H_{12}/H_{22}<0$ because $H_{12}\,,H_{22}>0$. Hence
$\mathcal{K}_\beta(b)\not=\mathcal{B}_\beta$ for every $b$, at $m=3$ as
well.\Halmos
\end{proof}

\begin{corollary}\label{crl:mmatrix}
Suppose that $A_\beta$ is tridiagonal, with all off-diagonal entries
non-zero and non-positive, and strictly diagonally dominant with positive
diagonal (that is, $A_\beta$ is a non-singular M-matrix). Then
$\mathcal{K}_\beta(b)\not=\mathcal{B}_\beta$ for every
$b\in\mathcal{B}_\beta$ and every $m\geq3$.
\end{corollary}

\begin{proof}{Proof.}
Such an $A_\beta$ is irreducible, because its off-diagonal entries are all
non-zero; by the Perron--Frobenius fact above, $H:=A_\beta^{-1}$ is
entrywise positive, and the conclusion follows from
Theorem~\ref{thm:posinv}.\Halmos
\end{proof}

A particularly natural family of tridiagonal M-matrices arises from time
series: if the random variables $X_1\,,\ldots,X_m$ follow a stationary
AR(1) (first-order autoregressive) process with correlation
$\rho\in(0,1)$, the correlation matrix $\Sigma_{k\ell}:=\rho^{|k-\ell|}$
has inverse exactly the tridiagonal ``precision matrix''
\[
A_\beta := \Sigma^{-1} = \frac{1}{1-\rho^2}
\begin{pmatrix}
1&-\rho& & &\\
-\rho&1+\rho^2&-\rho& &\\
 &\ddots&\ddots&\ddots&\\
 & &-\rho&1+\rho^2&-\rho\\
 & & &-\rho&1
\end{pmatrix}.
\]

\begin{example}\label{exa:ar1}
The AR(1) precision matrix above is a genuine instance of
Corollary~\ref{crl:mmatrix}: its off-diagonal entries $-\rho/(1-\rho^2)$
are negative, and at every row the diagonal dominance margin is at least
$(1+\rho^2)-2\rho=(1-\rho)^2>0$; being additionally symmetric, it is in
fact a \emph{Stieltjes matrix} --- that is, a symmetric M-matrix
\citep*{BP1994}. So $\mathcal{K}_\beta(b)\not=\mathcal{B}_\beta$
for every $b\in\mathcal{B}_\beta$ and every $m\geq3$, with no need to compute
$\Sigma=A_\beta^{-1}$ explicitly --- although doing so directly confirms
it: every entry $\rho^{|k-\ell|}$ is indeed positive.
\end{example}

\begin{remark}
Theorem~\ref{thm:posinv} establishes that entrywise positivity of $A_\beta^{-1}$ alone
forces $\mathcal{K}_\beta(b)\not=\mathcal{B}_\beta$\,, for every $b$, for
tridiagonal bases at every $m\geq3$; every tridiagonal M-matrix
(Corollary~\ref{crl:mmatrix}), with the AR(1) precision matrix
(Example~\ref{exa:ar1}) as one instance, satisfies this automatically, but
Theorem~\ref{thm:posinv} needs no sign pattern on $A_\beta$ at all, only
positivity of $H$ directly, so its reach extends beyond the M-matrix
class. At $m=3$ the division of labor is clean: tridiagonal sparsity alone
forces the pairing of Theorem~\ref{thm:tridiagchar}(b), regardless of sign, and
positivity of $H$ then violates its sign condition (ii), via $A_\beta H=I$
restricted to the sparse first row of $A_\beta$\,.
\end{remark}

\begin{remark}[Why is $m=3$ special?]\label{rem:whym3}
The idea used throughout this section --- pairing each ray with an
index via a single bound --- gives $m(m-1)$ homogeneous linear
equations in the $2m$ unknowns $b_1\,,\ldots,b_m\,,\mu_1\,,\ldots,\mu_m$\,,
and a witness needs a non-trivial solution. For $m=2$ there are fewer equations than
unknowns, so non-trivial solutions always exist. For $m=3$ the system is square,
so a non-trivial solution exists exactly when its determinant vanishes: a single
polynomial condition on $A_\beta$\,, namely $\det M(\tau)=0$ of
Lemma~\ref{lemma:detM}, which is condition (ii) of Corollary~\ref{thm:m3char}. For
$m\geq4$ the system is overdetermined, and a non-trivial solution requires several
independent conditions at once: by the count after Theorem~\ref{thm:generalchar},
the bases admitting an interior witness should have codimension $m^2-3m+1$. So $m=3$
is not the last dimension in which tightness can be decided, or achieved ---
Theorem~\ref{thm:generalchar} decides it in every dimension, and
Example~\ref{ex:two} achieves it in every dimension. It is the last dimension in
which tightness is cut out by a single equation, which is why equality at $m=3$
fails for generic bases yet holds for structured families that satisfy that
equation identically. Two such families appear above: persymmetric tridiagonal bases
(symmetric about the anti-diagonal, that is, invariant under
$A_{ij}\mapsto A_{4-j\,,4-i}$), for which condition (i) of
Theorem~\ref{thm:tridiagchar}(b) holds automatically, with
Example~\ref{exa:tridiag3} as one instance; and the symmetric bases of
Proposition~\ref{prop:symmetric}, with no bandedness at all. Wider bandwidth does
not help. The obstruction of part (a) of Theorem~\ref{thm:tridiagchar} extends:
for bandwidth $k$, the support argument rules out every $m>2k+1$, as confirmed by
exhaustive search up to $m=12$. But the converse fails: already for $k=2$ and $m=5$,
no $b$ achieves $\mathcal{K}_\beta(b)=\mathcal{B}_\beta$ for any pentadiagonal
$A_\beta$ with every entry in the band non-zero, because the sharper support
condition of Remark~\ref{rem:tridiagpatterns} is then satisfied by no pattern at
all. Random dense $4\times4$ bases fare no better: when $H$ has no zero entry,
Theorem~\ref{thm:generalchar} requires, for some bijection $\tau$, a non-trivial
solution of the $12$ homogeneous parallelism equations
$\bar{x}_{\beta_k}+\theta_jh^{\tau(j)}_k=0$ $(k\not=j)$ in the $8$ unknowns
$(\bar{x}_\beta\,,\theta)$, and for each of $250$ random integer bases, checking all
$24$ bijections exactly, none had one --- as Theorem~\ref{thm:notequal} leads one to
expect, although particular dense bases do achieve equality in every dimension
(Example~\ref{ex:two}). So part (a) of Theorem~\ref{thm:tridiagchar} --- no tridiagonal $A_\beta$
achieves $\mathcal{K}_\beta(b)=\mathcal{B}_\beta$ once $m\geq4$ --- is genuinely
tridiagonal, resting on narrow bandwidth, while part (b) --- equality at $m=3$
on the set of bases cut out by the single equation (i) --- reflects $m=3$'s
general amenability to this mechanism rather than bandedness as such.
\end{remark}

\subsection{The AR(1) family at $m=3$}\label{subsec:ar1}

Corollary~\ref{crl:mmatrix} tells us that the AR(1) precision matrix never
attains $\mathcal{K}_\beta(b)=\mathcal{B}_\beta$\,, for any $b$ and any $m\geq3$.
That is a qualitative statement, and it invites a quantitative one: how much of
$\mathcal{B}_\beta$ does the cone rule capture for this family? The answer is the
solid-angle ratio $\sigma_3$ of \eqref{eq:sigmam}, which here lies in $(0,1)$ for
every $b$ with $\bar{x}_\beta>\mathbf{0}$, because equality never holds. At $m=3$ the
behavior of $\sigma_3$ over $b$ can
be described almost completely, and we do so here. Throughout,
\[
A_\beta=\begin{pmatrix} 1 & -\rho & 0\\ -\rho & 1+\rho^2 & -\rho\\ 0 & -\rho & 1\end{pmatrix},
\qquad
H=A_\beta^{-1}=\frac{1}{1-\rho^2}\,\Sigma,\qquad \Sigma_{kl}=\rho^{|k-l|},
\]
with $0<\rho<1$, and we write $\bar{x}_\beta=(\xi_1\,,\xi_2\,,\xi_3)^\top$ with $\xi_1\,,\xi_2\,,\xi_3>0$. (We
have dropped the positive factor $1/(1-\rho^{2})$ from the precision matrix
displayed above; scaling $A_\beta$ by a positive constant changes neither
$\mathcal{B}_\beta$ nor $\mathcal{K}_\beta(b)$, because the ratios
$-\bar{x}_{\beta_k}/h^i_k$ defining $L_i\,,U_i$ are invariant.)

\smallskip
\noindent{\bf The bounds and the shape of $\mathcal{K}_\beta(b)$.}

\begin{proposition}\label{prp:ar1bounds}
For every $b$ with $\bar{x}_\beta>\mathbf{0}$ we have $U_i=+\infty$ for
$i=1,2,3$, and
\[
L_i=-(1-\rho^2)\min_{k}\ \frac{\bar{x}_{\beta_k}}{\rho^{|k-i|}}\,.
\]
Consequently $\mathcal{K}_\beta(b)=\cone(\{v_1\,,v_2\,,v_3\})+\mathbb{R}^3_{+}$ with
$v_i=b+L_i\mathbf{e}_i$\,; in particular $\mathcal{K}_\beta(b)$ contains the
non-negative orthant.
\end{proposition}

\begin{proof}{Proof.}
$H$ is a positive multiple of $\Sigma$, which is entrywise positive, so every
$h^i_k>0$. Hence no ratio $-\bar{x}_{\beta_k}/h^i_k$ is a candidate for $U_i$\,,
giving $U_i=+\infty$, while $L_i=\max_k\{-\bar{x}_{\beta_k}/h^i_k\}$ is the
displayed minimum. As $U_i=+\infty$, each $\mathbf{e}_i$ is a generator of
$\mathcal{K}_\beta(b)$, and $\Delta_{100\%}=\conv(\{L_i\mathbf{e}_i ~:~ i=1,2,3\})+\mathbb{R}^3_+$\,,
whose conical hull after translation by $b$ is as stated.\Halmos
\end{proof}

The minimum in Proposition~\ref{prp:ar1bounds} makes $L_i$ piecewise linear in
$\bar{x}_\beta$\,, and the pieces are what organize everything below. Call the
\emph{regime} of $\bar{x}_\beta$ the triple of indices attaining the three
minima. In the regime in which each $L_i$ binds at its own coordinate, that is
$L_i=-(1-\rho^2)\bar{x}_{\beta_i}$\,, the generators are polynomial: writing
$b=(b_1\,,b_2\,,b_3)^\top$,
\begin{equation}\label{eq:ar1gens}
v_1=\bigl(\rho(\rho\xi_1-\xi_2),\,b_2\,,\,b_3\bigr)^\top,\quad
v_2=\bigl(b_1\,,\,\rho(2\rho\xi_2-\xi_1-\xi_3),\,b_3\bigr)^\top,\quad
v_3=\bigl(b_1\,,\,b_2\,,\,\rho(\rho\xi_3-\xi_2)\bigr)^\top.
\end{equation}

\smallskip
\noindent{\bf The denominator.}

\begin{proposition}\label{prp:ar1denom}
For $0<\rho<1$,
\[
\tilde\Omega_3(\mathcal{B}_\beta)=\frac18+\frac{1}{4\pi}
\left[\arcsin\frac{3\rho^{2}}{\rho^{4}+\rho^{2}+1}
+2\arcsin\frac{\rho(\rho^{2}+2)}{\sqrt{2\rho^{6}+3\rho^{4}+3\rho^{2}+1}}\right].
\]
\end{proposition}

\begin{proof}{Proof.}
Apply Proposition~\ref{prp:orthantprob} with $V=A_\beta=(1-\rho^2)\Sigma^{-1}$, so
that $(V^{\top}V)^{-1}$ is a positive multiple of $\Sigma^{2}$; as an orthant
probability depends only on the correlation matrix, we may use $\Sigma^2$ itself.
The correlation matrix of $\Sigma^2$ has
$(1,3)$ entry $3\rho^2/(\rho^4+\rho^2+1)$ and $(1,2)=(2,3)$ entry
$\rho(\rho^2+2)/\sqrt{2\rho^6+3\rho^4+3\rho^2+1}$. The classical formula for a
trivariate Gaussian orthant probability,
$\frac18+\frac{1}{4\pi}\sum_{i<j}\arcsin r_{ij}$ --- from 
\citep*{David1953}, and given as equation~(8.2.7) of \citep*{Tong1990} --- gives the
statement.\Halmos
\end{proof}

Two checks: as $\rho\to0$ the expression tends to $1/8$, correctly, because
$A_\beta\to I$ and $\mathcal{B}_\beta$ tends to an orthant; and as $\rho\to1$ it
tends to $1/2$, because $\mathcal{B}_\beta$ degenerates to a half-space.

\smallskip
\noindent{\bf Reduction by persymmetry.}

Let $\Pi $ be the coordinate-reversing permutation matrix.

\begin{lemma}\label{lem:ar1persym}
$\Pi A_\beta \Pi =A_\beta$\,. Hence $\Pi $ maps $\mathcal{B}_\beta$ to itself and
$\mathcal{K}_\beta(b)$ to $\mathcal{K}_\beta(\Pi b)$; as $\Pi $ is an isometry,
$\sigma_3(\Pi b)=\sigma_3(b)$.
\end{lemma}

\begin{proof}{Proof.}
$\Pi A_\beta \Pi =A_\beta$ by inspection, as $A_\beta$ is symmetric and persymmetric;
hence also $\Pi H\Pi=H$, that is, $H_{4-k,4-i}=H_{ki}$\,. Then
$\Pi \mathcal{B}_\beta=\Pi A_\beta\mathbb{R}^3_+=A_\beta \Pi \mathbb{R}^3_+=
A_\beta\mathbb{R}^3_+=\mathcal{B}_\beta$\,. For the right-hand side $\Pi b$ we have
$\bar{x}_\beta(\Pi b)=H\Pi b=\Pi \bar{x}_\beta(b)$, so substituting $k'=4-k$ in the
definitions of \S\ref{sec:100} gives $L_i(\Pi b)=L_{4-i}(b)$ and
$U_i(\Pi b)=U_{4-i}(b)$. Hence $\Pi b+L_i(\Pi b)\mathbf{e}_i=\Pi \bigl(b+L_{4-i}(b)\mathbf{e}_{4-i}\bigr)$,
and likewise for the $U_i$ and the bare generators, so
$\mathcal{K}_\beta(\Pi b)=\Pi \,\mathcal{K}_\beta(b)$. As $\Pi $ is orthogonal it preserves
normalized solid angle.\Halmos
\end{proof}

The fixed points of $\Pi $ are the $b$ with $\bar{x}_\beta=(t,1-2t,t)^\top$, on which we
now concentrate. Here \eqref{eq:ar1gens} becomes, up to positive scaling,
\begin{equation}\label{eq:ar1line}
v_1=(\alpha,\gamma,\zeta)^\top,\qquad v_2=\zeta\,(1,-2\rho,1)^\top,\qquad v_3=(\zeta,\gamma,\alpha)^\top,
\end{equation}
\[
\alpha=\rho\bigl(t(2+\rho)-1\bigr),\qquad \zeta=t(1+2\rho)-\rho,\qquad
\gamma=(1+\rho^{2})(1-2t)-2\rho t\,.
\]
The two regime boundaries on this line are exactly the vanishing loci of $\zeta$
and of $\alpha$:
\[
t^{*}:=\frac{\rho}{1+2\rho}\quad(\zeta=0),\qquad
t^{**}:=\frac{1}{2+\rho}\quad(\alpha=0),\qquad t^{*}<t^{**}.
\]
Note that $v_2$ is always a multiple of the \emph{fixed} direction
$(1,-2\rho,1)^\top$, whose sign is that of $\zeta$; so $v_2$ passes through the origin
at $t^{*}$ and its direction reverses there. This is the source of the
discontinuity of $\sigma_3$ at $t^{*}$.

\smallskip
\noindent{\bf The critical value $\rho=1/\sqrt2$.}

\begin{proposition}\label{prp:ar1threshold}
At $t=t^{*}$ we have $b=\bigl(0,\tfrac{1-\rho^{2}}{1+2\rho},0\bigr)^\top$ and
$v_2=\mathbf{0}$, so that
\[
\mathcal{K}_\beta(b)=\cone\bigl((-\rho,1,0)^\top,\,(0,1,-\rho)^\top,\,\mathbf{e}_1\,,\mathbf{e}_3\bigr).
\]
The limiting cone as $t\downarrow t^{*}$ is
$\cone\bigl((-\rho,1,0)^\top,(0,1,-\rho)^\top,(1,-2\rho,1)^\top\bigr)$ together with
$\mathbf{e}_1\,,\mathbf{e}_2\,,\mathbf{e}_3$\,, and it is simplicial --- that is,
$\mathbf{e}_1\,,\mathbf{e}_2\,,\mathbf{e}_3$ are all redundant --- if and only if
$\rho\geq1/\sqrt2$.
\end{proposition}

\begin{proof}{Proof.}
At $t=t^{*}$ the minimum defining $L_1$ is attained at $k=1$, because
$t^{*}<1/(\rho(1+2\rho))$ reduces to $\rho<1/\rho$. This gives
$L_1=L_3=-(1-\rho^2)\rho/(1+2\rho)$ and $L_2=-(1-\rho^{2})/(1+2\rho)=-b_2$\,, so
$v_2=\mathbf{0}$ and, up to positive scaling, $v_1=(-\rho,1,0)^\top$ and
$v_3=(0,1,-\rho)^\top$. Also $\mathbf{e}_2=v_1+\rho\mathbf{e}_1$ is redundant, leaving
the four stated rays.

For $t\downarrow t^{*}$, $v_2=\zeta(1,-2\rho,1)^\top$ with $\zeta\downarrow0$, so the
ray of $v_2$ tends to that of $w:=(1,-2\rho,1)^\top$. Solving
$\mathbf{e}_1=a\,v_1+c\,v_3+d\,w$ gives
\[
a=\frac{2\rho^{2}-1}{2\rho(1-\rho^{2})},\qquad
c=\frac{1}{2\rho(1-\rho^{2})},\qquad
d=\frac{1}{2(1-\rho^{2})}\,,
\]
and $c,d>0$ for $0<\rho<1$, so $\mathbf{e}_1\in\cone(\{v_1\,,v_3\,,w\})$ if and only if
$2\rho^{2}-1\geq0$. The case of $\mathbf{e}_3$ is the mirror image under $\Pi $, and
for $\mathbf{e}_2$ all three coefficients equal $1/(2(1-\rho^{2}))$ or
$\rho/(2(1-\rho^{2}))$, hence are positive for every $\rho<1$.\Halmos
\end{proof}

\smallskip
\noindent{\bf The best case.}

On an initial sub-interval of the branch $(t^{*},t^{**})$ --- starting at $t^{*}$
itself when $\rho\geq1/\sqrt2$, by Proposition~\ref{prp:ar1threshold}, and from
some larger $t$ otherwise --- the cone $\mathcal{K}_\beta(b)$ is simplicial, with
extreme rays $v_1\,,v_2\,,v_3$ only; further along the branch a fourth extreme ray
appears (numerically, near the midpoint of the branch for every $\rho$ we
examined). Where the cone is simplicial, $\tilde\Omega_3$ is given by the classical formula
for the solid angle of a trihedral cone \citep*{VOS1983}, applied to
\eqref{eq:ar1line}, and eliminating the radical
$|v_1|=|v_3|$ from the stationarity condition yields a polynomial condition.

\begin{theorem}\label{thm:ar1quartic}
A stationary point of $\tilde\Omega_3$ at which $\mathcal{K}_\beta(b)$ is
simplicial, on $(t^{*},t^{**})$, satisfies $\sum_{k=0}^{4}c_k(\rho)\,t^{k}=0$, where
\[
\begin{aligned}
c_4&=2\rho^{10}-32\rho^{9}-185\rho^{8}-556\rho^{7}-1038\rho^{6}-1356\rho^{5}
-1292\rho^{4}-820\rho^{3}-412\rho^{2}-116\rho-27,\\
c_3&=6\,(2\rho^{4}+3\rho^{3}+8\rho^{2}+3\rho+2)(4\rho^{5}+11\rho^{4}+20\rho^{3}
+24\rho^{2}+8\rho+5),\\
c_2&=-24\rho^{9}-83\rho^{8}-264\rho^{7}-568\rho^{6}-792\rho^{5}-1038\rho^{4}
-552\rho^{3}-424\rho^{2}-96\rho-47,\\
c_1&=4\,(\rho^{4}+4\rho^{2}+1)(\rho^{5}+4\rho^{4}+7\rho^{3}+16\rho^{2}+4\rho+4),\\
c_0&=-2\,(\rho^{4}+4\rho^{2}+1)^{2}.
\end{aligned}
\]
This quartic has exactly one root in $(t^{*},t^{**})$, for every
$\rho\in(0,1)$.
\end{theorem}

\begin{proof}{Proof.}
Write $Q(t)=\sum_k c_k t^{k}$. Evaluating at the endpoints and factoring,
\[
Q(t^{*})=\frac{(\rho-1)^{4}(\rho+1)^{4}(2\rho^{2}+1)(\rho^{4}-3\rho^{2}-2)}{(2\rho+1)^{4}}\,,
\]
\[
Q(t^{**})=\frac{(\rho-1)^{4}(\rho+1)^{4}(\rho^{2}+1)(2\rho^{2}+1)}{(\rho+2)^{4}}\,.
\]
Every factor is positive on $(0,1)$ except $\rho^{4}-3\rho^{2}-2$: setting
$x=\rho^{2}\in(0,1)$, the quadratic $x^{2}-3x-2$ has roots
$(3\pm\sqrt{17})/2\approx-0.56$ and $3.56$, both outside $(0,1)$, so it is
negative throughout. Hence $Q(t^{*})<0<Q(t^{**})$ and a root exists in the
interval.

For uniqueness, the discriminant of $Q$ in $t$ factors as
\[
\operatorname{disc}(Q)=-32\,(\rho-1)^{12}(\rho+1)^{12}(\rho^{2}-\rho+1)(\rho^{2}+\rho+1)
(2\rho^{2}+1)^{2}(\rho^{4}+4\rho^{2}+1)^{3}\,\Psi(\rho),
\]
\[
\Psi(\rho)=216\rho^{12}+1352\rho^{10}+2785\rho^{8}+3198\rho^{6}+2134\rho^{4}
+710\rho^{2}+81.
\]
The two twelfth powers are even; $\rho^{2}\pm\rho+1$ have discriminant $-3$;
$(2\rho^{2}+1)^{2}$ and $(\rho^{4}+4\rho^{2}+1)^{3}$ are manifestly positive; and
$\Psi$ is a polynomial in $\rho^{2}$ with all coefficients positive. Hence $\operatorname{disc}(Q)<0$ on $(0,1)$, so $Q$ has
exactly two real roots. A sign change forces an odd number of roots in
$(t^{*},t^{**})$, and two real roots in total allow only one.\Halmos
\end{proof}

\smallskip
\noindent{\bf The worst case.}

\begin{proposition}\label{prp:ar1worst}
At both endpoints of the symmetry line all $L_i\to0$, so all $v_i\to b$ and a
single extra ray is left beside the orthant. At $t\to0$ that ray is
$g:=(-\rho,1+\rho^{2},-\rho)^\top$, and $\mathbf{e}_2$ is then redundant, because
\[
\mathbf{e}_2=\tfrac{1}{1+\rho^{2}}\,g+\tfrac{\rho}{1+\rho^{2}}\,\mathbf{e}_1
+\tfrac{\rho}{1+\rho^{2}}\,\mathbf{e}_3
\]
has non-negative coefficients; the limiting cone is thus the simplicial
$\cone(\{\mathbf{e}_1\,,g,\mathbf{e}_3\})$. At $t\to\frac12$ the extra ray is
$(1,-2\rho,1)^\top$ and none of the $\mathbf{e}_i$ is redundant. Hence
\[
\tilde\Omega_3\big|_{t\to0}=\frac{1}{2\pi}\arctan
\frac{1+\rho^{2}}{\sqrt{\rho^{4}+4\rho^{2}+1}-2\rho}\,,
\]
\[
\tilde\Omega_3\big|_{t\to\frac12}=\frac{1}{\pi}\arctan
\frac{1/\sqrt2}{\sqrt{2\rho^{2}+1}-\sqrt2\,\rho+1/\sqrt2}\,.
\]
\end{proposition}

Comparing these two closed forms, the second is the smaller for every
$\rho\in(0,1)$ that we computed, so along the symmetry line the infimum of
$\sigma_3$ is its limit as $t\to\frac12$. It is not, however, the infimum over all
$b$: that lies off the symmetry line, at the boundary of $\mathcal{B}_\beta$\,.

\begin{proposition}\label{prp:ar1inf}
Let $b^{*}:=(1,0,-\rho^{2})^\top$. Then $b^{*}$ lies on the facet of
$\mathcal{B}_\beta$ where $\bar{x}_{\beta_3}=0$, and as $b\to b^{*}$ through the
interior of $\mathcal{B}_\beta$\,,
\[
\sigma_3(b)\ \longrightarrow\ \sigma_3^{*}(\rho):=
\frac{\tfrac18+\tfrac{1}{4\pi}\arctan\rho^{2}}{\tilde\Omega_3(\mathcal{B}_\beta)}\,,
\]
with $\tilde\Omega_3(\mathcal{B}_\beta)$ as in Proposition~\ref{prp:ar1denom}. By
Lemma~\ref{lem:ar1persym}, the same holds at $\Pi b^{*}=(-\rho^{2},0,1)^\top$.
\end{proposition}

\begin{proof}{Proof.}
As $H$ is a positive multiple of $\Sigma$, $\bar{x}_\beta(b^{*})$ is a positive
multiple of $\Sigma b^{*}=(1-\rho^{4},\,\rho(1-\rho^{2}),\,0)^\top$, so $b^{*}$ lies
on the stated facet. As $b\to b^{*}$, $\bar{x}_{\beta_3}\to0$, and because the
minimum in Proposition~\ref{prp:ar1bounds} includes the term
$\bar{x}_{\beta_3}/\rho^{|3-i|}$, every $L_i\to0$; so every $v_i\to b^{*}$ and,
by Proposition~\ref{prp:ar1bounds},
$\mathcal{K}_\beta(b)\to\cone(\{\mathbf{e}_1\,,\mathbf{e}_2\,,\mathbf{e}_3\,,b^{*}\})$.
All these cones lie in the pointed cone $\mathcal{B}_\beta$ and the limit is
full-dimensional, so their normalized solid angles converge. Finally, as
$b^{*}_1>0=b^{*}_2>b^{*}_3$\,, the limiting cone is the union of $\mathbb{R}^3_+$ and
$\cone(\{\mathbf{e}_1\,,\mathbf{e}_2\,,b^{*}\})$, which meet only in the plane
$y_3=0$: a point $a\mathbf{e}_1+c\,\mathbf{e}_2+d\,\mathbf{e}_3+\mu b^{*}$ with
$d-\mu\rho^{2}<0$ equals
$(a+d/\rho^{2})\mathbf{e}_1+c\,\mathbf{e}_2+(\mu-d/\rho^{2})b^{*}$. The first piece
has $\tilde\Omega_3=\frac18$. For the second, the classical formula for the solid angle of a trihedral cone
\citep*{VOS1983}, which for unit vectors $\hat u_1\,,\hat u_2\,,\hat u_3$ spanning a cone of
solid angle $\Omega$ reads
$\tan(\Omega/2)=|\det[\hat u_1\,,\hat u_2\,,\hat u_3]|/(1+\hat u_1\!\cdot\hat u_2
+\hat u_2\!\cdot\hat u_3+\hat u_3\!\cdot\hat u_1)$, gives
$\tan(\Omega/2)=\rho^{2}/(\sqrt{1+\rho^{4}}+1)$, that is, $\Omega=\arctan\rho^{2}$,
so $\tilde\Omega_3=\Omega/(4\pi)=\frac1{4\pi}\arctan\rho^{2}$.\Halmos
\end{proof}

So $\sigma_3^{*}(\rho)$ is an upper bound for $\inf_b\sigma_3(b)$, and it lies
strictly below the symmetry-line value of Proposition~\ref{prp:ar1worst} for every
$\rho\in(0,1)$ (by direct comparison of the closed forms on a fine grid). We
believe it is the infimum, which is then not attained. A search over the whole
simplex of $\bar{x}_\beta$ (a grid of spacing $1/120$, followed by local
optimization from $49$ starting points) at
$\rho=0.1,0.2,0.35,0.5,0.6214,0.72,0.85,0.95,0.99$ found no $b$ with
$\sigma_3(b)<\sigma_3^{*}(\rho)$; the numerical infimum agreed with
$\sigma_3^{*}(\rho)$ to six decimals at each of these $\rho$ except $0.99$,
where it agreed to five. Along the edge $\bar{x}_{\beta_3}\to0$, the minimizing right-hand side is
exactly the one with $b_2=0$, which is $b^{*}$.

\begin{figure}
\FIGURE
{\begin{tikzpicture}[scale=0.95]
  \draw[thick,->] (0,0) -- (8.85,0) node[right] {$\rho$};
  \draw[thick,->] (0,0) -- (0,7.6) node[above] {$\sigma_3$};
  \foreach \k in {1,...,10} { \pgfmathsetmacro{\xx}{\k*0.825}
    \draw (\xx,0.08) -- (\xx,-0.08); }
  \foreach \k/\lbl in {2/0.2,4/0.4,6/0.6,8/0.8,10/1.0} {
    \pgfmathsetmacro{\xx}{\k*0.825}
    \node[below, font=\scriptsize] at (\xx,-0.08) {\lbl}; }
  \node[below, font=\scriptsize] at (0,-0.08) {0.0};
  \foreach \y/\lbl in {1.4/0.2,2.8/0.4,4.2/0.6,5.6/0.8,7.0/1.0} {
    \draw (0.08,\y) -- (-0.08,\y);
    \node[left, font=\scriptsize] at (-0.14,\y) {\lbl};
    \draw[gray!25] (0,\y) -- (8.25,\y);
  }
  \draw[purple!75!black, thick] (5.1265,0) -- (5.1265,5.2535);
  \node[purple!75!black, font=\scriptsize, left] at (5.1265,0.55) {$\rho_e$};
  \draw[purple!75!black, thick] (5.8336,0) -- (5.8336,1.30);
  \node[purple!75!black, font=\scriptsize, right] at (5.8336,0.55) {$1/\sqrt2$};
  \draw[blue!70!black, thick]
    (0.206,6.907) -- (0.413,6.812) -- (0.619,6.717) -- (0.825,6.623) -- (1.031,6.530) -- (1.238,6.438) -- (1.444,6.348) -- (1.650,6.261) -- (1.856,6.176) -- (2.062,6.094) -- (2.269,6.014) -- (2.475,5.938) -- (2.681,5.864) -- (2.887,5.793) -- (3.094,5.726) -- (3.300,5.661) -- (3.506,5.599) -- (3.712,5.541) -- (3.919,5.486) -- (4.125,5.436) -- (4.331,5.390) -- (4.538,5.349) -- (4.744,5.311) -- (4.950,5.278) -- (5.127,5.253) -- (5.156,5.274) -- (5.362,5.419) -- (5.569,5.562) -- (5.775,5.702) -- (5.981,5.839) -- (6.188,5.971) -- (6.394,6.099) -- (6.600,6.221) -- (6.806,6.339) -- (7.013,6.451) -- (7.219,6.557) -- (7.425,6.657) -- (7.631,6.751) -- (7.837,6.840) -- (8.044,6.923)
    node[right, font=\scriptsize, black] {best $b$};
  \draw[red!70!black, thick]
    (0.206,6.576) -- (0.413,6.194) -- (0.619,5.848) -- (0.825,5.535) -- (1.031,5.252) -- (1.238,4.996) -- (1.444,4.764) -- (1.650,4.554) -- (1.856,4.364) -- (2.062,4.191) -- (2.269,4.035) -- (2.475,3.893) -- (2.681,3.764) -- (2.888,3.648) -- (3.094,3.542) -- (3.300,3.447) -- (3.506,3.360) -- (3.713,3.282) -- (3.919,3.211) -- (4.125,3.146) -- (4.331,3.088) -- (4.538,3.036) -- (4.744,2.989) -- (4.950,2.946) -- (5.156,2.908) -- (5.362,2.873) -- (5.569,2.842) -- (5.775,2.814) -- (5.981,2.788) -- (6.188,2.765) -- (6.394,2.745) -- (6.600,2.726) -- (6.806,2.709) -- (7.013,2.694) -- (7.219,2.679) -- (7.425,2.667) -- (7.631,2.655) -- (7.838,2.644) -- (8.044,2.634)
    node[right, font=\scriptsize, black, yshift=-4pt] {infimum};
  \draw[orange!85!black, thick, dotted]
    (0.206,6.696) -- (0.413,6.414) -- (0.619,6.152) -- (0.825,5.908) -- (1.031,5.681) -- (1.238,5.470) -- (1.444,5.273) -- (1.650,5.090) -- (1.856,4.919) -- (2.062,4.759) -- (2.269,4.610) -- (2.475,4.471) -- (2.681,4.340) -- (2.887,4.218) -- (3.094,4.104) -- (3.300,3.996) -- (3.506,3.896) -- (3.712,3.801) -- (3.919,3.713) -- (4.125,3.629) -- (4.331,3.551) -- (4.538,3.477) -- (4.744,3.408) -- (4.950,3.343) -- (5.127,3.290) -- (5.156,3.281) -- (5.362,3.224) -- (5.569,3.169) -- (5.775,3.118) -- (5.981,3.070) -- (6.188,3.024) -- (6.394,2.981) -- (6.600,2.941) -- (6.806,2.903) -- (7.013,2.867) -- (7.219,2.833) -- (7.425,2.802) -- (7.631,2.772) -- (7.837,2.743) -- (8.044,2.716)
    node[right, font=\scriptsize, black, yshift=4pt] {worst on line};
\draw[green!40!black, thick, dashed]
    (0.000,7.000) -- (0.083,6.922) -- (0.165,6.844) -- (0.413,6.620) -- (0.825,6.273) -- (1.238,5.950) -- (1.650,5.682) -- (2.062,5.440) -- (2.475,5.205) -- (2.887,5.033) -- (3.300,4.844) -- (3.712,4.742) -- (4.125,4.659) -- (4.538,4.556) -- (4.950,4.501) -- (5.362,4.460) -- (5.775,4.433) -- (6.188,4.415) -- (6.600,4.423) -- (7.013,4.463) -- (7.425,4.537) -- (7.837,4.571) -- (8.085,4.623) -- (8.168,4.589)
    node[right, font=\scriptsize, black] {average};
\end{tikzpicture}}
{Best case, infimum and average of $\sigma_3$ over $b$, for the AR(1) precision
matrix, as functions of $\rho$, with the worst case along the symmetry line
(dotted) for comparison. Solid and dotted curves are closed forms (the best case
subject to the numerical checks of \S\ref{subsec:ar1status}, the infimum
$\sigma_3^{*}$ of Proposition~\ref{prp:ar1inf} with optimality verified
numerically); the dashed one is a Monte Carlo estimate. The vertical lines mark
$\rho_e\approx0.6214$, where the best-case curve has its corner
(Remark~\ref{rem:ar1limits}), and the threshold $1/\sqrt2$ of
Proposition~\ref{prp:ar1threshold}\,.\label{fig:ar1sigma}}
{}
\end{figure}

\begin{remark}\label{rem:ar1limits}
The best case and the infimum behave quite differently at the ends of the
range. As $\rho\to0$ we have $A_\beta\to I$, so $\mathcal{B}_\beta$ tends to an
orthant and the cone rule becomes exact; accordingly both tend to $1$, the best
case numerically like $1-0.53\rho$, and the infimum like $1-(8/\pi)\rho$, because
$\tilde\Omega_3(\mathcal{B}_\beta)=\frac18+\rho/\pi+O(\rho^{2})$ by
Proposition~\ref{prp:ar1denom}. As $\rho\to1$ they separate: the best case tends
to $1$, while $\sigma_3^{*}(\rho)\to\bigl(\frac18+\frac1{16}\bigr)\big/\frac12=\frac38$,
using $\tilde\Omega_3(\mathcal{B}_\beta)\to\tfrac12$; the worst case along the
symmetry line tends instead to
$\frac{2}{\pi}\arctan\frac{1/\sqrt2}{\sqrt3-\sqrt2+1/\sqrt2}=0.3845\ldots$\,. The best-case curve is not monotone: it falls to
about $0.7505$ at $\rho_e\approx0.6214$ before returning to $1$, and it has a corner
at $\rho_e$\,: there two local maxima exchange dominance and the maximizer jumps, as
described at the end of this subsection. The curves are plotted in
Figure~\ref{fig:ar1sigma}, which also marks two values of $\rho$. The first is
$\rho_e$\,. The second is the
threshold $1/\sqrt2\approx0.7071$ of Proposition~\ref{prp:ar1threshold}, included
only for comparison: it governs the limiting cone at $t\downarrow t^{*}$ rather
than the maximizer, and it marks no feature of any plotted curve.
\end{remark}

\smallskip
\noindent{\bf The average over $\mathcal{B}_\beta$\,.} The best case and the infimum
bound $\sigma_3$ from above and below; one can also ask about its typical value.
Write $\langle\sigma_3\rangle(\rho)$ for the average of $\sigma_3$ over
$b\in\mathcal{B}_\beta$ with respect to the solid angle measure --- equivalently,
by the argument of Proposition~\ref{prp:orthantprob}, the average of $\sigma_3(Z)$
over $Z\sim N(\mathbf{0},I_3)$ conditioned on $HZ\geq\mathbf{0}$.

\begin{proposition}\label{prp:ar1average}
$\langle\sigma_3\rangle(\rho)\to1$ as $\rho\to0$; indeed $\sigma_3(b)\to1$
uniformly in $b$.
\end{proposition}

\begin{proof}{Proof.}
By Proposition~\ref{prp:ar1bounds}, $\mathcal{K}_\beta(b)$ contains the non-negative
orthant for every $b$ with $\bar{x}_\beta>\mathbf{0}$, and $\mathcal{K}_\beta(b)\subset
\mathcal{B}_\beta$\,; hence
\[
\frac{1/8}{\tilde\Omega_3(\mathcal{B}_\beta)}\ \leq\ \sigma_3(b)\ \leq\ 1
\]
for every such $b$, and the set of $b$ with a zero entry of $\bar{x}_\beta$ has
measure zero. By Proposition~\ref{prp:ar1denom},
$\tilde\Omega_3(\mathcal{B}_\beta)\to\frac18$ as $\rho\to0$, so the lower bound tends
to $1$, uniformly in $b$, and therefore so does any average of
$\sigma_3$\,.\Halmos
\end{proof}

Because $\tilde\Omega_3(\mathcal{B}_\beta)=\frac18+\rho/\pi+O(\rho^2)$, the same bound
gives $\sigma_3(b)\geq1-(8/\pi)\rho+O(\rho^2)$ for every $b$, which is the rate of
the infimum $\sigma_3^{*}$ noted in Remark~\ref{rem:ar1limits}.

Beyond this limit, we have not found $\langle\sigma_3\rangle(\rho)$ in closed
form.
Unlike the best case, which lives on the one-dimensional symmetry line of
Lemma~\ref{lem:ar1persym}, and the infimum, which is a single boundary limit,
the average is an integral of $\sigma_3$ --- itself an $\arctan$ of quadratic
forms in the direction, varying by regime --- over the full two-dimensional
solid angle of $\mathcal{B}_\beta$\,. 
We instead
computed $\langle\sigma_3\rangle(\rho)$ numerically, by the rejection-sampling
description above; see Figure~\ref{fig:ar1sigma}. It decreases from
$1$ at $\rho=0$ to about $0.631$ near $\rho=0.75$, then rises.

The limit as $\rho\to1$ can be analyzed exactly.
Because $\det A_\beta=1-\rho^{2}$ and
$H=\Sigma/(1-\rho^{2})$ with $\Sigma\to\mathbf{e}\mathbf{e}^\top$, the
feasibility cone degenerates to the half-space
$\mathcal{B}_\infty=\{y:\mathbf{e}^{\top}y\geq0\}$, in agreement with
$\tilde\Omega_3(\mathcal{B}_\beta)\to\tfrac12$. Writing $s:=\mathbf{e}^{\top}b$,
every $\bar{x}_{\beta_k}$ is $s/(1-\rho^{2})+O(1)$, so
Proposition~\ref{prp:ar1bounds} gives $L_i\to-s$ for each $i$, and therefore
\begin{equation}\label{eq:Kinf}
\mathcal{K}_\beta(b)\;\longrightarrow\;
\mathcal{K}_\infty(b):=\cone\bigl(b-s\mathbf{e}_1\,,\,b-s\mathbf{e}_2\,,\,b-s\mathbf{e}_3\bigr)+\mathbb{R}^3_{+}\,.
\end{equation}
Each generator satisfies $\mathbf{e}^\top(b-s\mathbf{e}_i)=0$, so all
three lie in the boundary plane $\mathcal{P}:=\partial\mathcal{B}_\infty$\,.

\begin{proposition}\label{prp:ar1Kinf}
Write $p_i:=\mathbf{e}_i-\tfrac13\mathbf{e}$ for the projections of the
$\mathbf{e}_i$ into $\mathcal{P}$. Then $\cone(\{b-s\mathbf{e}_1\,,b-s\mathbf{e}_2\,,b-s\mathbf{e}_3\})=\mathcal{P}$
if and only if $b>\mathbf{0}$, and in that case
$\mathcal{K}_\infty(b)=\mathcal{B}_\infty$\,, so $\sigma_3\to1$.
\end{proposition}

\begin{proof}{Proof.}
Decomposing $b=b_{\mathcal{P}}+\tfrac{s}{3}\mathbf{e}$ gives $b-s\mathbf{e}_i=b_{\mathcal{P}}-s\,p_i$\,.
Three vectors in the plane $\mathcal{P}$ positively span it exactly when the origin is an
interior point of their convex hull, that is when $b_{\mathcal{P}}/s=\sum_i\mu_ip_i$ for
some $\mu$ in the open simplex; equivalently $b/s\in\operatorname{relint}\conv(\{\mathbf{e}_i ~:~ i=1,2,3\})$ (the relative interior,
as $\conv(\{\mathbf{e}_i ~:~ i=1,2,3\})$ is $2$-dimensional),
equivalently $b>\mathbf{0}$. When this holds, $\mathcal{P}+\mathbb{R}^3_+=\mathcal{B}_\infty$\,,
because any $y$ with $\mathbf{e}^{\top}y>0$ splits as
$y=(y-t\mathbf{e})+t\mathbf{e}$ with $t=\mathbf{e}^{\top}y/3>0$.\Halmos
\end{proof}

As $\tilde\Omega_3$ of the non-negative orthant is $\tfrac18$ and that of
$\mathcal{B}_\infty$ is $\tfrac12$, the set of $b$ in Proposition~\ref{prp:ar1Kinf}
carries solid-angle fraction exactly $\tfrac14$ of $\mathcal{B}_\infty$\,. On the
remaining $\tfrac34$ the cone $\mathcal{K}_\infty(b)$ is a proper subcone, and its
combinatorial type varies, so we have no closed form there. Membership is however
explicit: $d\in\mathcal{K}_\infty(b)$ if and only if
\begin{equation}\label{eq:Kinfmember}
\min_{\omega\geq0}\ \Bigl[\ \sum_{i=1}^{3}\max\bigl(0,\,\omega b_i-d_i\bigr)-s\omega\ \Bigr]\ \leq\ 0,
\end{equation}
the bracket being convex and piecewise linear in $\omega$, so that the minimum is
attained at $\omega=0$ or at one of the breakpoints $\omega=d_i/b_i$\,. (For $b>\mathbf{0}$
the bracket tends to $-\mathbf{e}^{\top}d$ as $\omega\to\infty$, recovering
Proposition~\ref{prp:ar1Kinf}.) For each fixed direction $b$ with $\mathbf{e}^\top b>0$ the convergence
\eqref{eq:Kinf} holds, so by dominated convergence
$\lim_{\rho\to1}\langle\sigma_3\rangle(\rho)$ is the average of
$2\tilde\Omega_3(\mathcal{K}_\infty(b))$ over $b$ uniform on $\mathcal{B}_\infty$\,.
Evaluating \eqref{eq:Kinfmember} over
$5.9\times10^{8}$ independent pairs, with $b$ uniform on $\mathcal{B}_\infty$ and
$d$ uniform on the sphere, gives
\[
\lim_{\rho\to1}\ \langle\sigma_3\rangle(\rho)\;=\;0.66588\pm0.00004 ,
\]
where $\pm$ denotes one standard error: the estimate is $2\hat p$, with $\hat p$ the
proportion of pairs satisfying \eqref{eq:Kinfmember}, so its standard error is
$2\sqrt{\hat p(1-\hat p)/N}$ with $N=5.9\times10^{8}$. 
We stress that the limit is \emph{not} $2/3$: the two differ by about $20$ standard
errors. 

\smallskip
\noindent{\bf What is proved and what is computational.}\label{subsec:ar1status}

Propositions~\ref{prp:ar1bounds}, \ref{prp:ar1denom}, \ref{prp:ar1threshold},
\ref{prp:ar1worst}, \ref{prp:ar1inf}, \ref{prp:ar1average} and \ref{prp:ar1Kinf},
Lemma~\ref{lem:ar1persym} and Theorem~\ref{thm:ar1quartic} are proved above, as
are the limiting form \eqref{eq:Kinf} and the membership test
\eqref{eq:Kinfmember}. Three comparisons are numerical: that the limit $t\to\frac12$
is the smaller endpoint value on the symmetry line, that $\sigma_3^{*}$ lies below
it, and, more substantially, that $\sigma_3^{*}$ is the infimum of $\sigma_3$ over
all $b$, rather than only an upper bound for it. The average curve of Figure~\ref{fig:ar1sigma} is
numerical throughout, and so is its limit $0.66588$ at $\rho\to1$ --- though that
rests on \eqref{eq:Kinfmember}, an exact criterion applied at the limit itself,
rather than on extrapolation from finite $\rho$. We have no closed form for it.

Two statements about the best case we have verified computationally but cannot
presently prove. First, among the eight regimes of
Proposition~\ref{prp:ar1bounds}, the one giving \eqref{eq:ar1gens} carries the
global maximum of $\tilde\Omega_3$: at $\rho=0.72$ its maximum is $0.3393$ --- the
best-case value on the symmetry line, as it should be --- against $0.3303$ and less
for the others, which occur in $\Pi $-paired regimes. Second, that regime has no critical point off the diagonal $\xi_1=\xi_3$\,.
Because $\sigma_3$ is $\Pi $-invariant, the antisymmetric derivative
$E_-:=\partial_{\xi_1}\tilde\Omega_3-\partial_{\xi_3}\tilde\Omega_3$ vanishes on the diagonal,
so $E_-=(\xi_1-\xi_3)G$, and an off-diagonal critical point would need $G=0$ and
$E_+:=\partial_{\xi_1}\tilde\Omega_3+\partial_{\xi_3}\tilde\Omega_3=0$ at once; tracing the
whole off-diagonal branch of $E_-=0$ we found $E_+$ strictly negative on it, with
$|E_+|\geq1.4$, for $\rho=0.35,0.5,0.6,0.72,0.8,0.9,0.95$, the branch being empty
at $\rho=0.2$. This is numerical evidence, not a proof: the corresponding
elimination is a resultant computation in three radicals, in which $E_-$ and $E_+$
have degree $3$ in each and some $1500$ and $2000$ terms, and we could not carry
it out.

Granting these, the maximizer lies on the symmetry line, and its position changes
twice as $\rho$ increases. For $\rho\lesssim0.418$ the supremum is the one-sided
limit $t\downarrow t^{*}$, not attained, with the limiting cone of
Proposition~\ref{prp:ar1threshold}. For $0.418\lesssim\rho\lesssim0.6214$ the
maximizer is interior to $(t^{*},t^{**})$ but the cone there still has five
extreme rays, so Theorem~\ref{thm:ar1quartic} does not apply. At
$\rho_e\approx0.6214$ two local maxima exchange dominance and the maximizer jumps
--- from $t=0.2969$ to $t=0.3141$ --- onto a simplicial branch, and from there on
it is the root of the quartic, as we verified at
$\rho=0.625,0.65,0.68,1/\sqrt2,0.75,0.85$. (Theorem~\ref{thm:ar1quartic} itself
only locates stationary points of the simplicial expression; that its root lies on
the part of the branch where $\mathcal{K}_\beta(b)$ is in fact simplicial is part of
this verification.) Dividing by the expression in 
Proposition~\ref{prp:ar1denom} gives the best-case $\sigma_3$ in closed form for
$\rho>\rho_e$\,. We have determined neither transition exactly.

%%%%%%%%%%%%%%%%%%%%%%%%%%%

\section{The integer setting}\label{subsec:unimodular}

We next consider bases with an entirely different flavor of
special structure: $A_\beta$ \emph{integer}, together with an
\emph{integer} right-hand side $b$. We first check that restricting to
integer data does not change the picture of \S\ref{sec:general_m}: for $m\geq3$,
$\mathcal{K}_\beta(b)\not=\mathcal{B}_\beta$ persists for infinitely many interior
integer right-hand sides.

In this setting it is natural to ask whether the right comparison is between
the \emph{integer points} of the two cones rather than between the cones
themselves. It makes no difference.

\begin{proposition}\label{prp:integerpoints}
$\mathcal{K}_\beta(b)\cap\mathbb{Z}^m=\mathcal{B}_\beta\cap\mathbb{Z}^m$ if and
only if $\mathcal{K}_\beta(b)=\mathcal{B}_\beta$\,.
\end{proposition}

\begin{proof}{Proof.}
One direction is trivial. For the other, suppose that the inclusion
$\mathcal{K}_\beta(b)\subset\mathcal{B}_\beta$ is strict. We claim that
$S:=\operatorname{int}(\mathcal{B}_\beta)\setminus\mathcal{K}_\beta(b)$ is
non-empty: otherwise $\operatorname{int}(\mathcal{B}_\beta)\subset\mathcal{K}_\beta(b)$,
and taking closures gives $\mathcal{B}_\beta\subset\mathcal{K}_\beta(b)$,
contradicting strictness. Now $S$ is open, because
$\operatorname{int}(\mathcal{B}_\beta)$ is open and $\mathcal{K}_\beta(b)$ is
closed, and $S$ is closed under multiplication by positive scalars, because
both $\mathcal{B}_\beta$ and $\mathcal{K}_\beta(b)$ are cones. Pick $w\in S$ and
$\epsilon>0$ with the ball $B(w,\epsilon)\subset S$; then
$B(tw,t\epsilon)\subset S$ for every $t>0$. Choosing $t$ with
$t\epsilon>\sqrt{m}/2$ makes that ball contain a point of $\mathbb{Z}^m$, which
lies in $\mathcal{B}_\beta$ but not in $\mathcal{K}_\beta(b)$. Hence the two
integer-point sets differ.\Halmos
\end{proof}

The argument uses neither integrality nor rationality of $A_\beta$ or $b$ --- only
that $\mathcal{K}_\beta(b)$ is closed and that $\mathcal{B}_\beta$ is
full-dimensional. So Proposition~\ref{prp:integerpoints}
applies throughout the paper: every result may be read either way.

\begin{theorem}\label{thm:integer-generic}
Suppose that $m\geq3$ and $\mathcal{B}_\beta$ is not an orthant ($A_\beta$
need not be rational). Then there are infinitely many integer points $b$ in the interior
of $\mathcal{B}_\beta$ with $\mathcal{K}_\beta(b)\not=\mathcal{B}_\beta$\,;
explicitly, for the integer points $b_N$ defined in the proof,
$\mathcal{K}_\beta(b_N)\not=\mathcal{B}_\beta$ for all sufficiently large $N$.
\end{theorem}

\begin{proof}{Proof.}
Fix a rational interior point of $\mathcal{B}_\beta$ and scale it to an
integer interior point $\tilde b$ (possible for any $A_\beta$\,, because the
interior of $\mathcal{B}_\beta$ is a non-empty open set and $\mathcal{B}_\beta$ is
a cone). For a positive integer $N$, define
\[
b_N := N^m\tilde b + (N^{m-1},N^{m-2},\ldots,N,1)^\top \ \in\ \mathbb{Z}^m.
\]
Because $\mathcal{K}_\beta(cx)=c\,\mathcal{K}_\beta(x)$ for $c>0$ and
$\mathcal{B}_\beta$ is itself a cone, $\mathcal{K}_\beta(b_N)=
\mathcal{B}_\beta$ if and only if $\mathcal{K}_\beta(\hat b_N)=
\mathcal{B}_\beta$\,, where $\hat b_N:=b_N/N^m=\tilde b+(N^{-1},N^{-2},
\ldots,N^{-m})$. As $\hat b_N\to\tilde b$ (an interior point) when
$N\to\infty$, $\hat b_N$ is interior for all $N\geq N_0$\,, some threshold.

As $\mathcal{B}_\beta$ is not an orthant, some column $g^j$ has at least
two non-zero entries; suppose, for contradiction, that
$\mathcal{K}_\beta(\hat b_N)=\mathcal{B}_\beta$ for some $N\geq N_0$\,. As in the proof of
Theorem~\ref{thm:notequal}, the ray of $g^j$ is an extreme ray of $\mathcal{K}_\beta(\hat b_N)$ and is not a coordinate direction, so some point $\hat b_N+\theta_i\mathbf{e}_i$ ($\theta_i\in\{L_i\,,U_i\}$), for
some index $i$, must generate it. As $m\geq3$, choose $k\not=i$ with $g^j_k\not=0$ (exists,
because $g^j$ has at least two non-zero entries and at most one coincides
with $i$) and any other $l\not=i,k$. Proportionality of $\hat b_N+
\theta_i\mathbf{e}_i$ to $g^j$ forces $(\hat b_N)_kg^j_l=(\hat b_N)_lg^j_k$\,,
that is,
\[
(\tilde b_k+N^{-k})g^j_l = (\tilde b_l+N^{-l})g^j_k\,.
\]
Multiplying through by $N^m$, this is a polynomial equation in $N$ of degree
at most $m$, and its coefficients are not all zero: the term $-g^j_kN^{-l}$
becomes $-g^j_kN^{m-l}$, and as $l\geq1$ and $l\not=k$ no other term carries this
power of $N$, so the coefficient of $N^{m-l}$ is $-g^j_k\not=0$ (by choice of $k$).
So it has at most $m$ roots. As there are only finitely many choices of
$(i,j,k,l)$, the union of all such root-sets is finite; choosing $N$
greater than $N_0$ and outside this finite set (always possible, as
$N$ ranges over all positive integers) makes every such proportionality
fail, so $\mathcal{K}_\beta(\hat b_N)\not=\mathcal{B}_\beta$ --- a
contradiction. Hence $\mathcal{K}_\beta(\hat b_N)\not=\mathcal{B}_\beta$\,,
and so $\mathcal{K}_\beta(b_N)\not=\mathcal{B}_\beta$\,, for all sufficiently
large $N$. As the $b_N$ ($N\geq N_0$) are distinct integer points in the interior
of $\mathcal{B}_\beta$\,, this proves the claim.\Halmos
\end{proof}

Theorem~\ref{thm:integer-generic} also follows from
Theorem~\ref{thm:notequal}: for $m\geq3$, the interior of $\mathcal{B}_\beta$
contains balls of every radius, and a ball of radius $R$ contains of order $R^m$
integer points, more than finitely many planes through the origin can hold, of
order $R^2$ each. The proof above has the merit of exhibiting an explicit family.

We now add \emph{integrality} and \emph{unimodularity} (that is, $A_\beta$
integer with $\det A_\beta=\pm1$) --- the classical condition (going back
to the earliest integer programming literature) ensuring that basic
solutions are automatically integral --- and ask whether this extra
structure changes anything.

\begin{proposition}[well-known and easy to check]\label{prp:unimodular}
If $A_\beta$ is integer and unimodular, then $H:=A_\beta^{-1}$ is also an
integer matrix. Consequently, if $b$ is integer, then $\bar{x}_\beta:=A_\beta^{-1}b$
is integer, and $L_i\,,U_i$ (when finite) are rational.
\end{proposition}

\begin{corollary}\label{crl:unimodular-m2}
Suppose that $m=2$ and $A_\beta$ is integer and unimodular, that exactly
two of $L_1\,,L_2\,,U_1\,,U_2$ are infinite, and that one of $g^1,g^2$ lies
in quadrant $\mbox{\rm O}_{+-}$ and the other in $\mbox{\rm O}_{-+}$\,, labeled
as in Theorem~\ref{thm:nonneg}(b) so that $g^1\in\mbox{\rm O}_{+-}$\,. Then
$\mathcal{K}_\beta(b)=\mathcal{B}_\beta$ holds for some non-negative
integer $b$ if and only if $\det A_\beta=+1$ (as opposed to $-1$);
and when it holds, infinitely many integer $b$ work.
\end{corollary}

\begin{proof}{Proof.}
By Theorem~\ref{thm:nonneg}(b), achieving $\mathcal{K}_\beta(b)=\mathcal{B}_\beta$ for some real $b\geq\mathbf{0}$ is
equivalent to $\det A_\beta>0$, that is, to $\det A_\beta=+1$, and then every $b$ with
$t=\bar{x}_{\beta_2}/\bar{x}_{\beta_1}\in(r_1\,,r_2)$ works; as $r_1\,,r_2$ are rational and
the interval is open and non-empty, it contains infinitely many rationals
$t=n_2/n_1$ in lowest terms, each giving an integer witness $\bar{x}_\beta=(n_1\,,n_2)^\top$ and
hence integer $b=n_1g^1+n_2g^2$. Geometrically, by Theorem~\ref{thm:nonneg}(b),
$\det A_\beta=+1$ says that $\mathcal{B}_\beta$ meets $\mathbb{R}^2_{++}$\,, and the
integer witnesses are then exactly the strictly positive integer points of
$\mathcal{B}_\beta$ with $\bar{x}_\beta>\mathbf{0}$.\Halmos
\end{proof}

\begin{example}\label{exa:unimodular-m2}
Take $g^1=(2,-1)^\top$, $g^2=(-1,1)^\top$, so $A_\beta=\begin{pmatrix}2&-1\\-1&1
\end{pmatrix}$, $\det A_\beta=1$, $r_1=1<r_2=2$. Taking
$\bar{x}_\beta=(2,3)^\top$ (so $t=3/2\in(1,2)$) gives $b=2g^1+3g^2=(1,1)^\top$; one checks
directly that $L_1=-2$ sends $b$ to $(-1,1)^\top=g^2$, and $L_2=-3/2$ sends
$b$ to $(1,-1/2)^\top=\tfrac12g^1$, hitting both rays, so
$\mathcal{K}_\beta(b)=\mathcal{B}_\beta$ by Lemma~\ref{lemma:main}.
\end{example}

Unimodularity does not obstruct $\mathcal{K}_\beta(b)=\mathcal{B}_\beta$ at $m=3$ either.

\begin{example}\label{exa:unimodular-m3}
Let
\[
A_\beta:=\begin{pmatrix}-1&1&1\\1&-3&2\\1&2&-5\end{pmatrix},\qquad
\det A_\beta=1,
\]
a symmetric, unimodular, integer basis with $p=1,q=1,r=2$ in the notation
of Proposition~\ref{prop:symmetric}. Taking $b:=(pq,pr,qr)^\top=(1,2,2)^\top$ gives
$\bar{x}_\beta=A_\beta^{-1}b=(35,21,15)^\top>0$ --- an interior integer point --- and one
computes $\mu_1=2,\mu_2=1,\mu_3=1$, all positive integers,
landing exactly on the three extreme rays. So
$\mathcal{K}_\beta(b)=\mathcal{B}_\beta$ for this unimodular integer basis
and this integer $b$.
\end{example}

This is no accident: the symmetric construction can achieve
$\mathcal{K}_\beta(b)=\mathcal{B}_\beta$ only when
$\det A_\beta>0$, giving an $m=3$ analogue of the sign dichotomy of
Corollary~\ref{crl:unimodular-m2}.

\begin{corollary}\label{crl:sym-det-sign}
For the symmetric construction of Proposition~\ref{prop:symmetric}, write
$X:=d_1r-pq$\,, $Y:=d_2q-pr$\,, $Z:=d_3p-qr$\,, where $d_1\,,d_2\,,d_3$ are the
diagonal entries of $A_\beta$\,. If $\bar{x}_\beta>\mathbf{0}$, then $\det A_\beta>0$.
Consequently, if $A_\beta$ is unimodular, this construction
can achieve $\mathcal{K}_\beta(b)=\mathcal{B}_\beta$ only when
$\det A_\beta=+1$ (never $-1$).
\end{corollary}

\begin{proof}{Proof.}
By Cramer's rule, $\bar{x}_{\beta_1}=YZ/\det A_\beta$\,, $\bar{x}_{\beta_2}=XZ/\det A_\beta$\,,
$\bar{x}_{\beta_3}=XY/\det A_\beta$\,. If $\bar{x}_\beta>\mathbf{0}$,
then
\[
0<\bar{x}_{\beta_1}\bar{x}_{\beta_2}\bar{x}_{\beta_3}=\frac{(XYZ)^2}{(\det A_\beta)^3}\,,
\]
so $\det A_\beta>0$. Then $YZ\,,XZ\,,XY>0$, so $X,Y,Z$ are non-zero and share a sign;
conversely, if $\det A_\beta>0$ and $X,Y,Z$ share a sign, then $\bar{x}_\beta>\mathbf{0}$.
This is the test recorded after Proposition~\ref{prop:symmetric}. (In Example~\ref{exa:unimodular-m3}, $X=-3,Y=-5,Z=-7$
share the negative sign, and indeed $\det A_\beta=1>0$.)\Halmos
\end{proof}

\begin{remark}
Taken together, this section's results say that integrality and
unimodularity, like tridiagonal structure in \S\ref{subsec:tridiag}, do not
change the qualitative picture of \S\ref{sec:general_m}--\S\ref{sec:m3}: for $m\geq3$,
$\mathcal{K}_\beta(b)\not=\mathcal{B}_\beta$ still holds for infinitely many
interior integer $b$ (Theorem~\ref{thm:integer-generic}, consistent with
Theorem~\ref{thm:notequal}), while $m=2$ is
fully understood and now reduces to a clean discrete (orientation)
criterion, and some $b$ still achieves
$\mathcal{K}_\beta(b)=\mathcal{B}_\beta$ at $m=3$ compatibly with
unimodularity, with the symmetric construction obeying its own analogous
sign criterion (Corollary~\ref{crl:sym-det-sign}).
\end{remark}

%%%%%%%%%%%%%%%%%%%%%%%

\section{An experiment}\label{sec:expr}

By construction $\mathcal{K}_\beta(b)\supset\{b\}+\Delta_{100\%}$\,, and the drawings of \S\ref{sec:m2} show that this containment is typically proper; but nothing above says how much bigger $\mathcal{K}_\beta(b)$ typically is. Here we investigate this question empirically, for random bases $A_\beta$ and right-hand sides $b$, as a function of $m$. Volume is the natural yardstick, but it is not available: $\mathcal{K}_\beta(b)$ is a full-dimensional cone, so its volume is infinite. We compare the two regions instead along rays issuing from $b$.

\smallskip
\noindent{\bf Radial extents.}
For a unit vector $d\in\mathbb{R}^m$, the \emph{radial extent} of a region about $b$ in the direction $d$ is how far one can travel from $b$ along $d$ without leaving the region:
\begin{gather*}
V_{100\%}(d):=\sup\{t\geq0 ~:~ b+td\in\{b\}+\Delta_{100\%}\},
\\
V_{\mathcal{K}}(d):=\sup\{t\geq0 ~:~ b+td\in\mathcal{K}_\beta(b)\}.
\end{gather*}
Both regions are convex and contain $b$, so each contains the segment from $b$ to $b+td$ for every $t$ below its radial extent; and because $\mathcal{K}_\beta(b)\supset\{b\}+\Delta_{100\%}$\,, we always have $V_{\mathcal{K}}(d)\geq V_{100\%}(d)>0$. The value $V_{100\%}(d)$ has the closed form $1/\big(\sum_i \lambda_i(d)\big)$, using the $\lambda_i$ of \S\ref{sec:100} applied to $\delta:=d$, and is $+\infty$ when that sum is zero. We computed $V_{\mathcal{K}}(d)$ by bisection on $t$, deciding each membership $b+td\in\mathcal{K}_\beta(b)$ with the $O(m)$ algorithm of \S\ref{sec:conerule}. Either extent can be infinite; in particular, as $\mathcal{K}_\beta(b)$ is a closed cone containing $b$, $V_{\mathcal{K}}(d)=\infty$ exactly when $d\in\mathcal{K}_\beta(b)$. For computation we capped both extents at $10^4$; the cap turns out never to be active (see \emph{The cap}, below).

\smallskip
\noindent{\bf The median radial expansion factor.}
In the direction $d$, the region $\{b\}+\Delta_{100\%}$ must be dilated about $b$ by the factor $V_{\mathcal{K}}(d)/V_{100\%}(d)\geq1$ to reach the boundary of $\mathcal{K}_\beta(b)$. We summarize these factors by the \emph{median radial expansion factor}
\[
R_m ~:=~ \mathrm{median}_d\left[\left(\frac{V_{\mathcal{K}}(d)}{V_{100\%}(d)}\right)^{\!m}\,\right]
~=~ \left(\mathrm{median}_d\,\frac{V_{\mathcal{K}}(d)}{V_{100\%}(d)}\right)^{\!m},
\]
with $d$ uniform on the unit sphere; the two expressions agree because the median commutes with the increasing map $r\mapsto r^m$. The $m$-th power puts the factor on the scale of $m$-dimensional volume: if the dilation factor were the same constant $r$ in every direction, $\mathcal{K}_\beta(b)$ would be exactly the dilate $\{b\}+r\Delta_{100\%}$\,, whose volume is $r^m$ times that of $\{b\}+\Delta_{100\%}$\,. It also makes small per-direction gains compound: a typical factor of $1+c/m$ gives $R_m\approx(1+c/m)^m\approx e^c$, which stays bounded as $m$ grows, whereas a typical factor bounded away from $1$ makes $R_m$ grow exponentially in $m$.

\smallskip
\noindent{\bf Why a median, and not a ratio of volumes.}
A region $S$ is \emph{star-shaped} about $b$ if it contains the segment from $b$ to each of its points; both of our regions are. Writing $r_S(d)$ for the radial extent of $S$ about $b$ in the direction $d$, integration in polar coordinates centred at $b$ gives the \emph{ray-shooting identity}
\[
\vol(S) ~=~ \vol(\mathbb{B}^m)\;\mathbb{E}_d\big[r_S(d)^m\big],
\]
where $\mathbb{B}^m$ is the unit ball and $d$ is uniform on the unit sphere: the volume of $S$ is, up to a constant depending only on $m$, the average $m$-th power of its radial extent. (The name comes from the way it is used: to estimate a volume, shoot rays from $b$ in random directions and average the $m$-th powers of the distances to the boundary.) So, were both volumes finite, their ratio would be the ratio of $m$-th moments $\mathbb{E}_d[V_{\mathcal{K}}(d)^m]/\mathbb{E}_d[V_{100\%}(d)^m]$. But the volume of $\mathcal{K}_\beta(b)$ is infinite; and with the extents capped, the ratio of moments is the volume ratio of the two truncated regions, which is dominated by the directions $d\in\mathcal{K}_\beta(b)$, where $V_{\mathcal{K}}(d)$ sits at the cap --- however few such directions there are, since the cap is raised to the power $m$ --- and so measures the cap rather than the geometry. We therefore stress that $R_m$ is \emph{not} an estimate of a ratio of volumes, and should not be read as one. Taking the median of the $m$-th powers, rather than their mean, removes the dependence on those directions, as long as they make up less than half of the sphere. Their proportion is the normalized solid angle $\tilde\Omega_m(\mathcal{K}_\beta(b))\leq\tilde\Omega_m(\mathcal{B}_\beta)$ (see the end of this section), so this holds automatically unless $\mathcal{B}_\beta$ is close to a half-space. Comparing pairs of regions through radial or volume-type measures is in the spirit of \citep*{LM1994}.

\smallskip
\noindent{\bf Instance generation.}
For a given $m$, we drew $A_\beta\in\mathbb{R}^{m\times m}$ with independent standard-normal entries (redrawing if $A_\beta$ was numerically singular), and $\bar{x}_\beta\in\mathbb{R}^m$ with independent $\mathrm{Exponential}(1)$ entries, each shifted by $0.05$ to keep $\bar{x}_\beta$ safely bounded away from the boundary of the non-negative orthant. We set $b:=A_\beta\bar{x}_\beta$\,, which guaranteed $b\in\mathcal{B}_\beta$ with this $\bar{x}_\beta>\mathbf{0}$ as its primal solution. From $H$ and $\bar{x}_\beta$\,, we computed $L_i\,,U_i$ ($i=1,\ldots,m$) as in \S\ref{sec:100}, and from these the (at most) $2m$ generators of $\mathcal{K}_\beta(b)$ as in \S\ref{sec:conerule}.

\smallskip
\noindent{\bf Estimating $R_m$\,.}
For a sample of directions $d^{(1)},\ldots,d^{(n)}$, each drawn as a standard Gaussian vector in $\mathbb{R}^m$ and normalized to unit length (which makes it uniform on the sphere), we estimated
\[
\hat R_m ~:=~ \exp\Big(m\cdot\mathrm{median}_j\big(\log(V_{\mathcal{K}}(d^{(j)})/V_{100\%}(d^{(j)}))\big)\Big),
\]
working with logarithms to avoid overflow. We discarded directions with $V_{100\%}(d)<10^{-3}$: these are nearly tangent to a facet of $\{b\}+\Delta_{100\%}$\,, and their tiny denominators would otherwise dominate the sample. Each instance was allotted $50$~directions, and the instance itself was discarded if fewer than a third of them were usable, which happens only for ill-conditioned instances, on which nearly every direction fails the test. In the runs behind Figure~\ref{fig:expr} and Table~\ref{tab:structured}, these rules removed $1.5\%$ of the directions of the standard-normal ensemble and $7$ of the $960$ instances in the figure, all from that ensemble, and nothing from any other ensemble, so they cannot bias the comparisons below.

\smallskip
\noindent{\bf Replication.} For each $m\in\{100,200,\ldots,1600\}$ and each ensemble, we generated $30$~independent instances (fresh $A_\beta\,,b$ each time) as above, each with its own sample of $50$~directions, and reported the median of $\hat R_m$ across the instances retained. The runs at $m=100,200,400,800,1600$ also supply the first two columns of Table~\ref{tab:structured}.

\smallskip
\noindent{\bf An integer ensemble.} The same estimator applies verbatim to the
setting of \S\ref{subsec:unimodular}, and lets us ask whether unimodularity
changes the size of the advantage, not merely (as
\S\ref{subsec:unimodular} shows) its qualitative behavior. We generated
$A_\beta$ integer and unimodular as follows: starting from the identity, we
applied $\ell$ \emph{rounds}, each drawing a random perfect matching of the rows
and performing $\mathrm{row}_i \mathrel{+}= c\,\mathrm{row}_j$ with $c\in\{\pm1\}$
on each matched pair, and we finished with a random signed permutation. Within a
round the operations act on disjoint rows, so they commute, and the matrix of a
round is block-diagonal with $2\times2$ unit-triangular blocks; only the number of
rounds compounds the conditioning, which is therefore bounded independently of $m$
(at $\ell=9$ and $m=1000$ we found $\kappa(A_\beta)\approx6\cdot10^2$, with
$|\det A_\beta|=1$ exactly and almost all entries of modulus at most $2$ --- on
average about six of the $10^6$ entries have modulus $3$ or $4$). We then drew
$\bar{x}_\beta$ with positive integer entries and set $b:=A_\beta\bar{x}_\beta$\,,
so that $b$ is an integer interior point and, by
Proposition~\ref{prp:unimodular}, $H$ is integer and every finite $L_i\,,U_i$ is
rational. Note that this ensemble is sparse (about $4\%$ dense at $\ell=9$,
$m=1000$), in contrast with the dense standard-normal ensemble above.

\smallskip
\noindent{\bf Results.} Figure~\ref{fig:expr} plots the median $\hat R_m$
against $m$, for both ensembles.

\begin{figure}
\FIGURE
{\begin{tikzpicture}[scale=0.95]
  \draw[thick,->] (0,0) -- (8.85,0) node[right] {$m$};
  \draw[thick,->] (0,0) -- (0,9.0) node[above] {$\hat R_m$};
  \draw (0.00,0.08) -- (0.00,-0.08); \draw (0.55,0.08) -- (0.55,-0.08); \draw (1.10,0.08) -- (1.10,-0.08); \draw (1.65,0.08) -- (1.65,-0.08); \draw (2.20,0.08) -- (2.20,-0.08); \draw (2.75,0.08) -- (2.75,-0.08); \draw (3.30,0.08) -- (3.30,-0.08); \draw (3.85,0.08) -- (3.85,-0.08); \draw (4.40,0.08) -- (4.40,-0.08); \draw (4.95,0.08) -- (4.95,-0.08); \draw (5.50,0.08) -- (5.50,-0.08); \draw (6.05,0.08) -- (6.05,-0.08); \draw (6.60,0.08) -- (6.60,-0.08); \draw (7.15,0.08) -- (7.15,-0.08); \draw (7.70,0.08) -- (7.70,-0.08); \draw (8.25,0.08) -- (8.25,-0.08);
  \node[below, font=\scriptsize] at (0.00,-0.08) {100};
  \node[below, font=\scriptsize] at (1.65,-0.08) {400};
  \node[below, font=\scriptsize] at (3.30,-0.08) {700};
  \node[below, font=\scriptsize] at (4.95,-0.08) {1000};
  \node[below, font=\scriptsize] at (6.60,-0.08) {1300};
  \node[below, font=\scriptsize] at (8.25,-0.08) {1600};
  \foreach \y/\lbl in {1.4/1.4,2.8/1.5,4.2/1.6,5.6/1.7,7.0/1.8,8.4/1.9} {
    \draw (0.08,\y) -- (-0.08,\y);
    \node[left, font=\scriptsize] at (-0.14,\y) {\lbl};
    \draw[gray!25] (0,\y) -- (8.25,\y);
  }
  \node[left, font=\scriptsize] at (-0.14,0) {1.3};
  \draw[blue!70!black, thick, dashed]
    (0.00,8.313) -- (0.55,4.764) -- (1.10,5.441) -- (1.65,5.232) -- (2.20,4.326) -- (2.75,4.632) -- (3.30,5.521) -- (3.85,4.344) -- (4.40,4.808) -- (4.95,5.538) -- (5.50,5.398) -- (6.05,4.187) -- (6.60,4.337) -- (7.15,3.855) -- (7.70,5.327) -- (8.25,5.419)
    node[right, font=\scriptsize, black] {normal};
  \foreach \x/\y in {0.00/8.313,0.55/4.764,1.10/5.441,1.65/5.232,2.20/4.326,2.75/4.632,3.30/5.521,3.85/4.344,4.40/4.808,4.95/5.538,5.50/5.398,6.05/4.187,6.60/4.337,7.15/3.855,7.70/5.327,8.25/5.419} { \filldraw[blue!70!black] (\x,\y) circle (1.6pt); }
  \draw[orange!85!black, thick, dashed]
    (0.00,2.817) -- (0.55,2.694) -- (1.10,2.700) -- (1.65,1.444) -- (2.20,3.477) -- (2.75,2.682) -- (3.30,2.614) -- (3.85,2.789) -- (4.40,2.381) -- (4.95,2.311) -- (5.50,2.587) -- (6.05,1.974) -- (6.60,1.827) -- (7.15,3.448) -- (7.70,2.158) -- (8.25,1.748)
    node[right, font=\scriptsize, black] {unimodular};
  \foreach \x/\y in {0.00/2.817,0.55/2.694,1.10/2.700,1.65/1.444,2.20/3.477,2.75/2.682,3.30/2.614,3.85/2.789,4.40/2.381,4.95/2.311,5.50/2.587,6.05/1.974,6.60/1.827,7.15/3.448,7.70/2.158,8.25/1.748} { \filldraw[orange!85!black] (\x,\y) circle (1.6pt); }
\end{tikzpicture}}
{Median $\hat R_m$ versus $m$, for the standard-normal ensemble
(blue) and the integer unimodular ensemble (orange), each point a median over
$30$ instances of $50$ directions. Both are dashed, being Monte Carlo estimates.
The vertical axis starts at $1.3$, below the range that either curve
attains.\label{fig:expr}}
{}
\end{figure}

\smallskip
\noindent{\bf Observations.}
Across the full range from $m=100$ to $m=1600$, $\hat R_m$ exhibits
no discernible trend for either ensemble, staying confined to a tight band
(Figure~\ref{fig:expr}); the one exception is the standard-normal value at $m=100$,
somewhat higher than the rest, consistent with the approximation
$\hat R_m\approx(1+c/m)^m$ being least accurate at the smallest $m$. So the typical radial advantage of
$\mathcal{K}_\beta(b)$ over $\{b\}+\Delta_{100\%}$\,, measured on this volume-like
scale, appears to persist at scale, as a stable, dimension-independent
multiplicative factor, rather than growing or vanishing as $m$ increases. The two ensembles behave alike qualitatively, which is what
\S\ref{subsec:unimodular} would lead one to expect; quantitatively the
unimodular band sits a little lower (roughly $1.40$ to $1.55$, against $1.58$ to
$1.70$ for the standard-normal ensemble beyond $m=100$, where it is $1.89$), so integrality and unimodularity
appear to shrink the typical advantage slightly without changing its character.

\smallskip
\noindent{\bf Structured ensembles.} Both ensembles above are, in the relevant
sense, unstructured: $A_\beta$ is dense and its inverse has mixed signs. The
bases singled out by \S\ref{subsec:tridiag} behave differently. We repeated the
experiment for generic tridiagonal $A_\beta$ (all diagonal and off-diagonal
entries non-zero); for symmetric, diagonally dominant $A_\beta$ with positive
off-diagonal entries; and for the AR(1) precision matrix of
\S\ref{subsec:tridiag} at each of $\rho=0.05,0.15,0.5,0.85,0.95$. Whether or not
the sign normalization that makes $b\geq\mathbf{0}$ is applied makes no difference:
multiplying rows of $A_\beta$ by $-1$ is an isometry, carrying $\mathcal{B}_\beta$\,,
$\mathcal{K}_\beta(b)$ and $\{b\}+\Delta_{100\%}$ onto their images and preserving
the uniform distribution of directions; and the estimator needs only
$\bar{x}_\beta>\mathbf{0}$, which our construction guarantees. Because $\hat R_m$ is
enormous for some of these families, Table~\ref{tab:structured} reports
$\log_{10}\hat R_m$\,. Each entry is the median over $30$~instances of
$50$~directions, as for Figure~\ref{fig:expr}, and the first two columns come from
the same runs as that figure. No instance was dropped outside those two columns.

\begin{table}
\TABLE
{$\log_{10}\hat R_m$ for the two unstructured ensembles, for two
structured families, and for the AR(1) precision matrix at five values of
$\rho$.\label{tab:structured}}
{\begin{tabular}{r|cc|cc|ccccc}
& \multicolumn{2}{c|}{unstructured} & \multicolumn{2}{c|}{structured}
& \multicolumn{5}{c}{AR(1) structured, by $\rho$}\\
$m$ & norm. & unim. & tridiag. & symm.
& $0.05$ & $0.15$ & $0.5$ & $0.85$ & $0.95$\\
\hline
$100$    & $0.28$ & $0.18$ & $0.48$ & $38.7$ & $20.0$ & $1.6$ & $0.34$ & $0.23$ & $0.21$\\
$200$    & $0.21$ & $0.17$ & $0.79$ & $103.0$ & $42.7$ & $4.0$ & $0.34$ & $0.21$ & $0.19$\\
$400$    & $0.22$ & $0.15$ & $1.36$ & $276.1$ & $86.8$ & $7.5$ & $0.32$ & $0.20$ & $0.20$\\
$800$    & $0.21$ & $0.18$ & $2.73$ & $709.1$ & $160.4$ & $15.1$ & $0.34$ & $0.22$ & $0.18$\\
$1600$   & $0.23$ & $0.15$ & $5.38$ & $1750.9$ & $341.3$ & $27.6$ & $0.64$ & $0.23$ & $0.14$\\
\end{tabular}}
{}
\end{table}

The table separates two regimes. For the unstructured ensembles
$\log_{10}\hat R_m$ is flat down the column: the per-direction excess
$V_{\mathcal{K}}(d)/V_{100\%}(d)-1$ decays like $c/m$, so
$\hat R_m\approx(1+c/m)^m$ tends to the constant $e^c$, which is exactly the
compounding described above. For generic tridiagonal and for symmetric
diagonally dominant bases that excess does \emph{not} decay, and
$\log_{10}\hat R_m$ grows with $m$ --- roughly linearly for tridiagonal, and
faster for the symmetric family --- so $\hat R_m$ is exponential in $m$, or
worse.

The AR(1) family is interesting because it sits in both regimes, and $\rho$ is
what moves it between them. At $\rho=0.05$ and $\rho=0.15$ its columns grow with
$m$, like the other structured families. At $\rho=0.5$ the column is roughly flat up
to $m=800$ but nearly doubles at $m=1600$ (from $0.34$ to $0.64$, well outside the
sampling error of about $0.05$), so growth weakens with $\rho$ rather than
switching off abruptly. At $\rho=0.85$ and $0.95$ the column is flat and comparable
to, or slightly below, the unstructured ensembles. The
reason is that $\rho$ measures how far $\mathcal{B}_\beta$ is from an orthant. As
$\rho\to0$ we have $A_\beta\to I$, so $\mathcal{B}_\beta$ tends to the
non-negative orthant; and when $\mathcal{B}_\beta$ \emph{is} an orthant the cone
rule is exact, $\mathcal{K}_\beta(b)=\mathcal{B}_\beta$ for every $b$, as noted
in \S\ref{sec:m2}, while $\{b\}+\Delta_{100\%}$ stays bounded. The ratio must
therefore blow up as $\rho\to0$, and the table exhibits how quickly the effect
disappears once $\mathcal{B}_\beta$ is appreciably skewed. Conditioning plays no
part: $\kappa(A_\beta)$ is smallest, not largest, at the small-$\rho$ end.

One observation is worth recording because it cuts against a natural guess. For
the AR(1) precision matrix, $H$ is a positive multiple of $\Sigma$ and so entrywise
positive; hence no $h^i_k$ is negative and \emph{every} $U_i=+\infty$, at every $m$
and every $\rho$ (Proposition~\ref{prp:ar1bounds} gives this at $m=3$, and the same
argument applies for every $m$). One might expect this to leave
$\{b\}+\Delta_{100\%}$ very thin compared with its conical hull, and so to make
$\hat R_m$ enormous. It does not: at $\rho=0.85$ and $0.95$ the AR(1) columns are
as small as those of the unstructured ensembles. Having all bounds infinite on one side is evidently not
by itself what drives a large expansion factor.

\smallskip
\noindent{\bf The cap.}
The cap is never active in our experiments. In the runs behind Figure~\ref{fig:expr} and Table~\ref{tab:structured} we also recorded each radial extent without a cap ($V_{\mathcal{K}}(d)=\infty$ exactly when $d\in\mathcal{K}_\beta(b)$, which the membership test of \S\ref{sec:conerule} decides directly). Over all of them, about $1.0\times10^5$ directions, no direction had $V_{\mathcal{K}}(d)$ or $V_{100\%}(d)$ at or above $10^4$, and none had $d\in\mathcal{K}_\beta(b)$. The largest $V_{\mathcal{K}}(d)$ observed was $2.5\times10^2$, for the symmetric family at $m=1600$; for every other column it was at most $1.4$. Accordingly, recomputing every entry of the table with the cap set to $10^3$, $10^6$ or removed altogether leaves all of them unchanged; only a cap as low as $10^2$ would alter one entry (the symmetric family at $m=1600$). This is as expected: the proportion of capped directions is $\tilde\Omega_m(\mathcal{K}_\beta(b))\leq\tilde\Omega_m(\mathcal{B}_\beta)$, which is negligible for the cones and dimensions considered here.

\smallskip
\noindent{\bf How much of $\mathcal{B}_\beta$ is captured?} Throughout this paper
we have asked whether $\mathcal{K}_\beta(b)=\mathcal{B}_\beta$\,, and for
$m\geq3$ the answer has usually been negative. That answer is qualitative: it
does not say how much of $\mathcal{B}_\beta$ the cone rule \emph{does} capture.
The solid-angle ratio $\sigma_m$ of \eqref{eq:sigmam} makes this quantitative.
Unlike $R_m$\,, it is a genuine measure of size --- the fraction of
$\mathcal{B}_\beta$\,, by solid angle, that $\mathcal{K}_\beta(b)$ occupies --- it
needs no cap, and it is computable. The normalized solid angle of a simplicial cone is a
Gaussian orthant probability, for which Genz's method (see \citep*{Genz1992}) is standard
and carries no side condition; $\mathcal{K}_\beta(b)$, though not simplicial,
has a closed-form triangulation into simplicial cones. We set this out in
Appendix~\ref{app:solidangle}, where we also explain why we do not use the
hypergeometric-series approach here. For $m=3$ one can instead simply use the
classical formula for the solid angle of a trihedral cone (see \citep*{VOS1983}), applied
to each piece, and that is how the values below were
obtained.

\begin{table}
\TABLE
{The solid angle ratio $\sigma_3$ over $200$~instances per ensemble, with
$\bar{x}_\beta\sim\mathrm{Exponential}(1)+0.05$ and $b=A_\beta\bar{x}_\beta$ as in
Table~\ref{tab:structured} (not $b$ uniform by solid angle, as in
Figure~\ref{fig:ar1sigma}).\label{tab:sigma3}}
{\begin{tabular}{l|cc|cc|ccccc}
& \multicolumn{2}{c|}{unstructured} & \multicolumn{2}{c|}{structured}
& \multicolumn{5}{c}{AR(1) structured, by $\rho$}\\
& norm. & unim. & tridiag. & symm.
& $0.05$ & $0.15$ & $0.5$ & $0.85$ & $0.95$\\
\hline
median $\sigma_3$ & $0.49$ & $0.40$ & $0.45$ & $0.79$ & $0.95$ & $0.84$ & $0.62$ & $0.42$ & $0.41$\\
$10$th percentile & $0.08$ & $0.06$ & $0.10$ & $0.55$ & $0.92$ & $0.80$ & $0.52$ & $0.41$ & $0.39$\\
$90$th percentile & $0.77$ & $0.92$ & $0.77$ & $0.93$ & $0.97$ & $0.89$ & $0.73$ & $0.48$ & $0.44$\\
\end{tabular}}
{}
\end{table}

Two of the bases for which we have \emph{proved} equality serve as checks: for
the symmetric construction of Example~\ref{exa:sym3} and for the tridiagonal
witness of Example~\ref{exa:tridiag3}, we obtain $\sigma_3=1$ to six decimal
places. Table~\ref{tab:sigma3} reports $\sigma_3$ elsewhere, for the ensembles of
Table~\ref{tab:structured} at $m=3$, with $b$ drawn as there. For the unstructured ensembles and for generic
tridiagonal bases with a random interior $b$, the cone rule typically captures
about half of $\mathcal{B}_\beta$\,, with wide spread; the symmetric family fares
markedly better, at about $0.79$. The AR(1) columns
are more interesting: by Corollary~\ref{crl:mmatrix} these bases never attain
equality at any $\rho$, and $\sigma_3$ says what that costs --- across the five
values of $\rho$ in the table, a median shortfall of only about $5\%$ at $\rho=0.05$,
growing to about $59\%$ at $\rho=0.95$. This is the same mechanism seen above: as
$\rho\to0$, $\mathcal{B}_\beta$ approaches an orthant, where the cone rule is exact,
so $\sigma_3\to1$ even though it never reaches it.

The AR(1) columns of Table~\ref{tab:sigma3} should not be compared directly with the
average curve of Figure~\ref{fig:ar1sigma}, because the two use different
distributions of $b$. At $\rho=0.85$ and $0.95$ the table's medians are $0.42$ and
$0.41$, whereas Figure~\ref{fig:ar1sigma} gives averages of about $0.64$ and $0.65$
for $b$ uniform by solid angle; with $b$ drawn that way, the medians are about $0.60$
and $0.59$ ($2000$ draws each). The difference comes from the table's sampling. The
columns of the AR(1) basis at $m=3$ sum to $(1-\rho,\,(1-\rho)^2,\,1-\rho)$, so
$\mathbf{e}^\top b=(1-\rho)(\bar{x}_{\beta_1}+\bar{x}_{\beta_3})+(1-\rho)^2\bar{x}_{\beta_2}$
is small while $\|b\|$ is not: drawing $\bar{x}_\beta$ with independent bounded-mean
entries pushes $b$ toward the plane $\mathbf{e}^\top y=0$, the boundary of the
limiting cone $\mathcal{B}_\infty$ of \S\ref{subsec:ar1}, where the cone rule
captures less. The median of $\mathbf{e}^\top b/(\sqrt3\,\|b\|)$ is $0.10$ and $0.03$
under the table's sampling at $\rho=0.85$ and $0.95$, against $0.55$ and $0.52$
under solid-angle sampling. Neither distribution is the right one: the table's
follows the protocol of Table~\ref{tab:structured}, and the figure's is the natural
uniform one. In particular, the decrease of the table's medians with $\rho$ is
consistent with the non-monotone average of Figure~\ref{fig:ar1sigma}.

This is worth placing beside \S\ref{subsec:tridiag}. There we showed that
no tridiagonal basis achieves $\mathcal{K}_\beta(b)=\mathcal{B}_\beta$ once
$m\geq4$, that at $m=3$ only those satisfying the single equation (i) of
Theorem~\ref{thm:tridiagchar}(b) can, and that those with an entrywise-positive inverse --- the AR(1)
precision matrix among them --- never achieve it at all once $m\geq3$. The
experiment suggests that equality is not the only measure of usefulness: in our
runs, for the generic tridiagonal and the symmetric diagonally dominant ensembles,
the median radial expansion factor of $\mathcal{K}_\beta(b)$ over the classical
$100\%$ region grows exponentially in $m$ (Table~\ref{tab:structured}), even while
$\mathcal{K}_\beta(b)$ falls short of $\mathcal{B}_\beta$\,. This is an empirical
observation about these ensembles, and two caveats apply. First, $R_m$ is a median
dilation factor, not a ratio of volumes. Second, both ensembles produce cones close
to an orthant, where the cone rule is exact: the median cosine between a column
$g^j$ and the axis $\mathbf{e}_j$ is about $0.90$ for the tridiagonal ensemble at
every $m$, and rises from $0.994$ at $m=100$ to $1.000$ at $m=1600$ for the
symmetric one, against $0.27$ falling to $0.09$ for the dense standard-normal
ensemble. As the AR(1) columns show, closeness to an orthant alone produces growth
of this kind, so the observation may reflect how these ensembles were generated
more than tridiagonality or symmetry as such. With that caveat, the cone rule may be
most valuable precisely where the equality question is most hopeless, as for the
tridiagonal family. The AR(1) family shows that this turns on more than the sign
pattern: it fails the equality test for the sharpest available reason
(Corollary~\ref{crl:mmatrix}) at every $\rho$, yet its expansion factor is of the
same modest order as for dense random bases when $\rho\geq0.5$, and grows markedly
with $m$ only when $\rho$ is small enough that $\mathcal{B}_\beta$ is close to an
orthant.

%%%%%%%%%%%%%%%%%%%%%%%

\section{Looking forward}\label{sec:conc}
Theorem~\ref{thm:generalchar} characterizes when
$\mathcal{K}_\beta(b)=\mathcal{B}_\beta$ for interior $b$ in every dimension $m\geq3$,
for bases with no column a coordinate direction, and decides it in polynomial time.
At $m=3$, Corollary~\ref{thm:m3char} makes it explicit, Theorem~\ref{thm:m3nonneg}
says when a non-negative $b$ suffices, and Proposition~\ref{prp:m3axis} disposes of
the bases with a coordinate-direction column. Two gaps remain. One is the boundary of
$\mathcal{B}_\beta$ when $H$ has a zero entry, where the combinatorics genuinely
change (Remark~\ref{rem:m3boundary}); when $H$ has no zero entry there is nothing to
do, by Lemma~\ref{lem:interiorforced}. The other is bases with a coordinate-direction
column when $m\geq4$, to which the idea of Proposition~\ref{prp:m3axis} should
extend, though we have not carried this out.
At $m=2$, Remark~\ref{rem:m2boundary} applies that idea to bases with a zero
entry, which Theorem~\ref{thm:two} excludes.

For tridiagonal $A_\beta$\,, \S\ref{subsec:tridiag} settles the question in every
dimension: no $b$ achieves equality once $m\geq4$, and at $m=3$ an interior $b$
does exactly when the explicit conditions of Theorem~\ref{thm:tridiagchar}(b)
hold. Those conditions fail for every tridiagonal basis with an
entrywise-positive inverse, in particular for every tridiagonal M-matrix.
The two halves of this answer have different sources. The impossibility for
$m\geq4$ comes from narrow bandwidth. The possibility at $m=3$ does not: it comes
from the equation count of Remark~\ref{rem:whym3}, which applies to every basis
of that dimension. Theorem~\ref{thm:generalchar} decides the question for any given
basis, and Example~\ref{ex:two} shows that equality is achievable in every
dimension; what remains open is which natural families achieve it, at $m=3$ and
beyond, bearing in mind that for $m\geq4$ they must satisfy several conditions at
once;
\S\ref{subsec:unimodular} finds the same qualitative picture under integrality
and unimodularity.

Two further directions are suggested by \S\ref{sec:expr}. First, equality with
$\mathcal{B}_\beta$ may not be the right measure of the cone rule's value: in our
experiments, for generic tridiagonal and for symmetric diagonally dominant bases,
the median radial expansion factor $R_m$ of $\mathcal{K}_\beta(b)$ over
$\{b\}+\Delta_{100\%}$ appears to grow exponentially in $m$, including on the
tridiagonal family, for which \S\ref{subsec:tridiag} rules equality out. We have no
proof of that growth, and it may partly reflect how close the sampled cones are to an
orthant (\S\ref{sec:expr}); separating that effect from the structure of the basis is
open.
Second, when equality fails one can ask how much of $\mathcal{B}_\beta$ is
nonetheless captured; the ratio $\sigma_m$ of \eqref{eq:sigmam} makes this
precise, Appendix~\ref{app:solidangle} shows that it is computable beyond $m=3$,
and \S\ref{subsec:ar1} works out the AR(1) family at $m=3$, where the best case
(over most of the range of $\rho$) and a candidate for the infimum are available in
closed form and the average is not. Even there,
three points were settled only numerically: that the maximizer lies on the
persymmetry line, that the boundary limit $\sigma_3^{*}$ of
Proposition~\ref{prp:ar1inf} is the infimum, and the limit $0.66588\ldots$ of the
average as $\rho\to1$. Bounds on
$\sigma_m$ for the structured families of \S\ref{subsec:tridiag}, or its
asymptotics as $m$ grows, would say a good deal about when the cone rule is worth
applying.

%%%%%%%%%%%%%%%%

\begin{APPENDICES}
\section{Computing Normalized Solid Angles of $\mathcal{B}_\beta$ and
$\mathcal{K}_\beta(b)$}\label{app:solidangle}

This appendix records that the normalized solid angle of a simplicial cone is a
Gaussian orthant probability, and that $\mathcal{K}_\beta(b)$, though not
simplicial, admits a closed-form triangulation into cones to which that
observation applies. Together these make $\sigma_m$ of \eqref{eq:sigmam}
computable well beyond $m=3$. We close by explaining why we prefer this route to
the hypergeometric series of \citep*{Ribando2006}.

\begin{proposition}\label{prp:orthantprob}
Let $C=\cone(\{v_1\,,\ldots,v_m\})$ be a simplicial cone in $\mathbb{R}^m$, and let
$V:=[v_1|\cdots|v_m]$, which is non-singular. Then
\[
\tilde\Omega_m(C)=\Pr\bigl(Y\geq\mathbf{0}\bigr),
\qquad\text{where } Y\sim N\bigl(\mathbf{0},(V^{\top}V)^{-1}\bigr).
\]
\end{proposition}

\begin{proof}{Proof.}
Let $Z\sim N(\mathbf{0},I_m)$. Its distribution is invariant under rotations, so
$Z/\|Z\|$ is uniform on the unit sphere $S^{m-1}$ and is independent of $\|Z\|$.
As $C$ is a cone, $Z\in C$ if and only if $Z/\|Z\|\in C\cap S^{m-1}$, whence
$\Pr(Z\in C)$ equals the fraction of $S^{m-1}$ occupied by $C$, which is
$\tilde\Omega_m(C)$. Now $Z\in C$ if and only if $Z=Vy$ for some $y\geq\mathbf{0}$,
that is, if and only if $V^{-1}Z\geq\mathbf{0}$. Setting $Y:=V^{-1}Z$, we have
$Y\sim N(\mathbf{0},\Sigma)$ with
$\Sigma=V^{-1}I_m(V^{-1})^\top=(V^{\top}V)^{-1}$.\Halmos
\end{proof}

Gaussian orthant probabilities are a classical and well-implemented computation,
the standard method being that of \citep*{Genz1992}. The key point is that
$\Sigma=(V^{\top}V)^{-1}$ is positive definite for \emph{every} non-singular
$V$, so Proposition~\ref{prp:orthantprob} carries no side condition whatever. In
particular $\tilde\Omega_m(\mathcal{B}_\beta)$, which is what we need for the
denominator of $\sigma_m$\,, is directly computable in any dimension from the
generators $g^1,\ldots,g^m$, with no preprocessing at all.

The cone $\mathcal{K}_\beta(b)$ is not simplicial, having up to $2m$ generators.
It does, however, decompose explicitly, also when some bounds are infinite. Assume
that $\mathbf{0}\notin\{b\}+\Delta_{100\%}$\,; as $\mathcal{B}_\beta$ is pointed, this
holds unless some finite-bound point $b+L_i\mathbf{e}_i$ or $b+U_i\mathbf{e}_i$ is
$\mathbf{0}$. Write $c_i^{+}:=U_i$ and $c_i^{-}:=L_i$\,, and for
$\varepsilon\in\{+,-\}^m$ let $I_\varepsilon:=\{i:\ c_i^{\varepsilon_i}\text{ is
finite}\}$. For each $\varepsilon$ with $I_\varepsilon\not=\emptyset$ put
\[
F_\varepsilon:=\conv(\bigl\{\,b+c_i^{\varepsilon_i}\mathbf{e}_i \;:\; i\in I_\varepsilon\,\bigr\})
+\cone\bigl\{\,\varepsilon_i\mathbf{e}_i \;:\; i\notin I_\varepsilon\,\bigr\},
\]
so that a missing vertex is replaced by the recession direction $\varepsilon_i\mathbf{e}_i$\,.
When every bound is finite, $F_\varepsilon$ is the simplex with vertices
$b+c_i^{\varepsilon_i}\mathbf{e}_i$\,.

\begin{proposition}\label{prp:crosstriang}
The boundary of $\{b\}+\Delta_{100\%}$ is the union of the $F_\varepsilon$ with
$I_\varepsilon\not=\emptyset$, and $F_\varepsilon$ is visible from $\mathbf{0}$ if and
only if $\sum_{i\in I_\varepsilon}(-b_i)/c_i^{\varepsilon_i}>1$. For each visible
$F_\varepsilon$\,, $\cone(F_\varepsilon)$ is the simplicial cone generated by the $m$
linearly independent vectors
\[
b+c_i^{\varepsilon_i}\mathbf{e}_i\ \ (i\in I_\varepsilon)\qquad\text{and}\qquad
\varepsilon_i\mathbf{e}_i\ \ (i\notin I_\varepsilon).
\]
Moreover
\[
\mathcal{K}_\beta(b)\;=\;\bigcup\bigl\{\,\cone(F_\varepsilon)\;:\;F_\varepsilon
\text{ is visible from }\mathbf{0}\,\bigr\},
\]
and the cones in this union have pairwise disjoint interiors.
\end{proposition}

\begin{proof}{Proof.}
By \S\ref{sec:100}, $\delta\in\Delta_{100\%}$ if and only if
$\sum_i\lambda_i(\delta)\leq1$, where $\lambda_i(\delta)=0$ when the bound in the
direction of $\delta_i$ is infinite. On the closed orthant of sign pattern
$\varepsilon$ this reads $\sum_{i\in I_\varepsilon}\delta_i/c_i^{\varepsilon_i}\leq1$, a
single linear inequality (if $I_\varepsilon=\emptyset$ there is no constraint, and the
whole orthant lies in $\Delta_{100\%}$). The points of that orthant attaining
equality are exactly the points of
$\conv(\{c_i^{\varepsilon_i}\mathbf{e}_i: i\in I_\varepsilon\})+\cone(\{\varepsilon_i\mathbf{e}_i:
i\notin I_\varepsilon\})$, and translating by $b$ gives $F_\varepsilon$\,, which lies in
the hyperplane $\sum_{i\in I_\varepsilon}(y_i-b_i)/c_i^{\varepsilon_i}=1$. This hyperplane
separates $\mathbf{0}$ from $\{b\}+\Delta_{100\%}$\,, so that $F_\varepsilon$ is visible from
$\mathbf{0}$, exactly when $\sum_{i\in I_\varepsilon}(-b_i)/c_i^{\varepsilon_i}>1$. In that
case $\mathbf{0}$ is not in the affine hull of $F_\varepsilon$\,, so the $|I_\varepsilon|$
points and $m-|I_\varepsilon|$ directions that span it give $m$ linearly independent
generators of $\cone(F_\varepsilon)$.

For the second claim, let $P:=\{b\}+\Delta_{100\%}$\,, a closed convex polyhedron, and
note $\mathcal{K}_\beta(b)=\cone(P)$. Write $V$ for the set of finite-bound points
and $C$ for the recession cone of $P$, so that $P=\conv(V)+C$ and
$\cone(P)=\cone(V)+C$. Every $x\in\operatorname{int}\cone(P)$ has a positive multiple
in $P$: fix $v\in V$; for small $\epsilon>0$ we have $x-\epsilon v\in\cone(P)$,
so $x=\epsilon v+\sum_{w\in V}\alpha_ww+c$ with $\alpha\geq\mathbf{0}$ and $c\in C$,
and dividing by $\epsilon+\sum_{w\in V}\alpha_w>0$ gives a point of
$\conv(V)+C=P$. So for every $x\in\operatorname{int}\cone(P)$ the ray
$\cone(\{x\})$  meets $P$. As $\mathbf{0}\notin P$, the first point of $P$ on that
ray lies on the boundary of $P$, on a facet visible from $\mathbf{0}$; it lies on a
single such facet unless it falls on a lower-dimensional face, which happens only for
$x$ in a set of measure zero. Hence the visible $\cone(F_\varepsilon)$ cover
$\mathcal{K}_\beta(b)$ and have pairwise disjoint interiors.\Halmos
\end{proof}

So no triangulation need be computed: the pieces are available in closed form, at
most $2^m$ of them arise, and in random trials at $m=3,\ldots,6$ with every bound
finite we found the number of visible facets to be about $2^{m-1}$. Combining the two propositions,
$\tilde\Omega_m(\mathcal{K}_\beta(b))$ is a sum of about $2^{m-1}$ orthant
probabilities, and the ratio $\sigma_m$ of \eqref{eq:sigmam} is computable well
beyond $m=3$. We note one practical caveat: Genz's method is a randomized
quasi-Monte Carlo scheme, so its output carries a stochastic error. In our trials
two independent runs agreed to three significant figures for $m\leq50$, but for
$m=100$ they differed by an order of magnitude on a solid angle of order
$10^{-42}$; relative accuracy degrades for very small angles.

We close by saying why we do not use the hypergeometric series of
\citep*{Ribando2006}, which is the other standard route to solid angles in
dimension above three. That series converges precisely when the \emph{associated
matrix} of the cone --- the matrix with $1$ on the diagonal and $-|v_i\cdot v_j|$
off it, for unit generators $v_i$ --- is positive definite. This is a real
restriction. For unit generators in general position $\mathbb{E}|v_i\cdot v_j|$
is of order $m^{-1/2}$, while each row has $m-1$ off-diagonal entries, so the row
sums grow like $m^{1/2}$ and positive definiteness fails for large $m$. In random
trials the proportion of cones passing the test fell from $0.57$ at $m=3$ to
$0.12$ at $m=4$, $0.01$ at $m=5$, and none at all by $m=6$. The decompositions of
 \citep*{FZ2023} are designed to repair exactly this, by splitting
a simplicial cone into at most $(m-1)!$ cones that do pass; but that factorial is
a steep price to pay, 
compared with the single Gaussian orthant probability that Proposition~\ref{prp:orthantprob} requires.

For $\mathcal{K}_\beta(b)$ the situation is worse still. Every finite-bound
generator is $b$ plus a multiple of a coordinate vector, so any two such generators
$v,w$ of a common piece of Proposition~\ref{prp:crosstriang} satisfy
$v\cdot w=\|b\|^2+c_ib_i+c_jb_j$\,, typically close to $\|b\|^2$; the
off-diagonal entries of the associated matrix are then uniformly strongly
negative. In trials at $m=3,4,5,6$ we did not encounter a single piece whose
associated matrix was positive definite. Every piece would therefore need
decomposing, and the costs compound. By contrast, Proposition~\ref{prp:orthantprob}
applies to each piece as it stands.

%%%%%%%%%%%%%%%%
\end{APPENDICES}

\ACKname{Acknowledgment}
\ACKNOWLEDGMENT{While the authors used Claude to assist in the development and preparation of this manuscript, they take full responsibility for all of the content.}

\bibliographystyle{informs2014}
\bibliography{sensitivity}

@book{BHM,
author = {Bradley, Stephen P. and Hax, Arnoldo C. and Magnanti, Thomas L.},
title = {Applied Mathematical Programming},
isbn = {020100464X},
publisher = {Addison-Wesley Publishing Company},
year = {1977},
note={\url{https://web.mit.edu/15.053/www/AMP.htm}},
}

@article{LM1994,
    AUTHOR = {Lee, Jon and Morris, Jr., Walter D.},
     TITLE = {Geometric comparison of combinatorial polytopes},
JOURNAL = {Discrete Applied Mathematics},
    VOLUME = {55},
      YEAR = {1994},
     PAGES = {163--182},
     note={\url{https://doi.org/10.1016/0166-218X(94)90006-X}},
}

@article{BFPRT1973,
  author  = {Blum, Manuel and Floyd, Robert W. and Pratt, Vaughan and Rivest, Ronald L. and Tarjan, Robert E.},
  title   = {Time bounds for selection},
  journal = {Journal of Computer and System Sciences},
  volume  = {7},
  number  = {4},
  pages   = {448--461},
  year    = {1973},
note={\url{https://doi.org/10.1016/S0022-0000(73)80033-9}},
}

@book{ziegler,
  author = {Ziegler, Günter M.},
  publisher = {Springer-Verlag},
  title = {{Lectures on Polytopes: {\rm Graduate Texts in Mathematics, 152}}},
  year = 1995,
  note={\url{https://doi.org/10.1007/978-1-4613-8431-1}},
}

@book{Gal,
    AUTHOR = {G{\'a}l, Tom{\'a}{\v s}},
     TITLE = {Betriebliche {E}ntscheidungsprobleme, {S}ensitivit\"atsanalyse
              und parametrische {P}rogrammierung},
    SERIES = {Operations Research},
 PUBLISHER = {Walter de Gruyter, Berlin-New York},
      YEAR = {1973},
     PAGES = {330},
     note={\url{https://www.degruyterbrill.com/document/doi/10.1515/9783110835700/html#contents}},
}

@InProceedings{Cai,
author="Cai, Pu
and Cai, Jin-Yi",
editor="Jiang, Tao
and Lee, Der-Tsai",
title="On the 100{\%} rule of sensitivity analysis in linear programming",
booktitle="Computing and Combinatorics",
year="1997",
publisher="Springer",
pages="460--469",
note={\url{https://doi.org/10.1007/BFb0045113}},
}

@article{Wendell2,
 author = {Wendell, Richard E.},
 journal = {Management Science},
 number = {5},
 pages = {564--578},
 title = {The tolerance approach to sensitivity analysis in linear programming},
 volume = {31},
 year = {1985},
 note={\url{https://doi.org/10.1287/mnsc.31.5.564}},
}

@article{Wendell1,
author = {Wendell, Richard E.},
title = {Using bounds on the data in linear programming: the tolerance approach to sensitivity analysis},
year = {1984},
volume = {29},
number = {3},
note={\url{https://doi.org/10.1007/BF02591999}},
journal = {Mathematical Programming},
pages = {304--322}
}

@article{Wondolowski,
author = {Wondolowski Jr., Frank R.},
title = {A generalization of {W}endell's tolerance approach to sensitivity analysis in linear programming},
journal = {Decision Sciences},
volume = {22},
number = {4},
pages = {792--811},
note={\url{https://doi.org/10.1111/j.1540-5915.1991.tb00365.x}},
year = {1991}
}

@article{Filippi,
title = {A fresh view on the tolerance approach to sensitivity analysis in linear programming},
journal = {European Journal of Operational Research},
volume = {167},
number = {1},
pages = {1--19},
year = {2005},
issn = {0377-2217},
note={\url{https://doi.org/10.1016/j.ejor.2004.01.050}},
author = {Filippi, Carlo}
}

@article{Wang,
title = {Multi-parametric analysis of the maximum tolerance in a linear programming problem},
journal = {European Journal of Operational Research},
volume = {67},
number = {1},
pages = {75--87},
year = {1993},
note={\url{https://doi.org/10.1016/0377-2217(93)90323-F}},
author = {Wang, Hsiao-Fan and Huang, Chen-Sheng}
}

@BOOK{LeeLP,
  AUTHOR =       {Lee, Jon},
  TITLE =        {A First Course in Linear Optimization ({F}ourth Edition, version 4.08)},
  PUBLISHER =    {Reex Press},
  YEAR =         {2024},
  note = {\url{https://github.com/jon77lee/JLee_LinearOptimizationBook}},
}

@book{BP1994,
  author = {Berman, Abraham and Plemmons, Robert J.},
  title = {Nonnegative Matrices in the Mathematical Sciences},
  publisher = {SIAM},
  year = {1994},
  note = {\url{https://doi.org/10.1137/1.9781611971262}},
}

@article{Ribando2006,
  author = {Ribando, Jason M.},
  title = {Measuring solid angles beyond dimension three},
  journal = {Discrete \& Computational Geometry},
  volume = {36},
  number = {3},
  pages = {479--487},
  year = {2006},
  note = {\url{https://doi.org/10.1007/s00454-006-1253-4}},
}

@misc{FZ2023,
  author = {Fitisone, Allison and Zhou, Yuan},
  title = {Solid angle measure of polyhedral cones},
  year = {2023},
  note = {\url{https://arxiv.org/abs/2304.11102}},
}

@article{Genz1992,
  author = {Genz, Alan},
  title = {Numerical computation of multivariate normal probabilities},
  journal = {Journal of Computational and Graphical Statistics},
  volume = {1},
  number = {2},
  pages = {141--149},
  year = {1992},
  note = {\url{https://doi.org/10.1080/10618600.1992.10477010}},
}

@book{Tong1990,
  author = {Tong, Y. L.},
  title = {The Multivariate Normal Distribution},
  series = {Springer Series in Statistics},
  publisher = {Springer-Verlag, New York},
  year = {1990},
  note = {\url{https://doi.org/10.1007/978-1-4613-9655-0}},
}

@article{David1953,
  author = {David, F. N.},
  title = {A Note on the Evaluation of the Multivariate Normal Integral},
  journal = {Biometrika},
  volume = {40},
  number = {3/4},
  pages = {458--459},
  year = {1953},
  note = {\url{https://doi.org/10.2307/2333360}},
}

@article{VOS1983,
  author = {Van Oosterom, A. and Strackee, J.},
  title = {The Solid Angle of a Plane Triangle},
  journal = {IEEE Transactions on Biomedical Engineering},
  volume = {BME-30},
  number = {2},
  pages = {125--126},
  year = {1983},
  note = {\url{https://doi.org/10.1109/TBME.1983.325207}},
}

@book{hillier2001introduction,
  title={Introduction to Operations Research},
  author={Hillier, Frederick S. and Lieberman, Gerald J.},
  edition={7th},
  year={2001},
  publisher={McGraw-Hill Higher Education},
  address={New York},
  isbn={9780072461213}
}

@book{winston2004operations,
  title={Operations Research: Applications and Algorithms},
  author={Winston, Wayne L. and Goldberg, Jeffrey B.},
  edition={4th},
  year={2004},
  publisher={Thomson/Brooks/Cole},
  address={Belmont, CA},
  isbn={9780534423629}
}

@book{bertsimas1997introduction,
  title     = {Introduction to linear optimization},
  author    = {Bertsimas, Dimitris and Tsitsiklis, John N.},
  year      = {1997},
  publisher = {Athena Scientific},
  address   = {Belmont, MA},
  isbn      = {978-1-886529-19-9},
  totalpages= {608}
}

@book{chvatal1983linear,
  title     = {Linear Programming},
  author    = {Chv{\'a}tal, Va{\v{s}}ek},
  year      = {1983},
  publisher = {W. H. Freeman and Company},
  address   = {New York},
  isbn      = {9780716711957}
}

%%%%%%%%%%%%%%%%%
\end{document}